\documentclass[11pt]{amsart}

\usepackage{amssymb}
\usepackage{mathrsfs} % \mathscr{A} produces a script A
\usepackage{tikzexternal}
{dgs}    % tell tikz the file name to look for
\usepackage[latin1]{inputenc}
\usepackage[colorlinks=true, citecolor=blue]{hyperref} 
\allowdisplaybreaks[1] %allows line breaks in multiline equations setting argument to 4 is least restrictive

\newtheoremstyle{fancy}{}{}{\itshape}{}{\textbf\bgroup}{.\egroup}{ }{}
\newtheoremstyle{fanci}{}{}{\rm}{}{\textbf\bgroup}{.\egroup}{ }{}
\newtheoremstyle{ghost}{}{}{\itshape}{}{\textbf\bgroup}{\egroup}{ }{}
\theoremstyle{fancy}
\numberwithin{equation}{section} 
\newtheorem{cor}[equation]{Corollary}
\newtheorem{lem}[equation]{Lemma}
\newtheorem{prop}[equation]{Proposition}
\newtheorem{thm}[equation]{Theorem}

\theoremstyle{fanci}

\newtheorem{dfn}[equation]{Definition}

\newtheorem{exa}[equation]{Example}

\newtheorem{rem}[equation]{Remark}

\newtheorem{notarem}[equation]{Notation and Remarks}

\newtheorem*{DGMeasure}{Constructing the Duistermaat-Guillemin Measure}

\newtheorem{prob2}{Problem}

\newcommand{\cref}[1]{Corollary~\ref{#1}}   %use: \cref{labelname}
\newcommand{\Z}{\mathbb Z}  % the integrers
\newcommand{\Q}{\mathbb Q} % the rationals
\newcommand{\R}{\mathbb R} % the reals
\newcommand{\N}{\mathbb N} % the natural numbers
\newcommand{\C}{\mathbb C} %the complex numbers

\newcommand{\SO}{\operatorname{SO}}
\newcommand{\ad}{\operatorname{ad}}
\newcommand{\Ad}{\operatorname{Ad}}
\newcommand{\Tr}{\operatorname{Tr}}
\newcommand{\Exp}{\operatorname{exp}}
\newcommand{\GL}{\operatorname{GL}}

\newcommand{\Spin}{\operatorname{Spin}}
\newcommand{\Sp}{\operatorname{Sp}}

\newcommand{\SU}{\operatorname{SU}}

\newcommand{\Isom}{\operatorname{Isom}}

\newcommand{\Hom}{\operatorname{Hom}}

\newcommand{\Rank}{\operatorname{rank}}
\newcommand{\DegSing}{\operatorname{deg}_{\rm{sing}}}

\newcommand{\vol}{\operatorname{vol}}

\newcommand{\germ}{\mathfrak}

\newcommand{\mg}{\mathfrak{g}}

\newcommand{\inner}{\ensuremath{\langle \cdot , \cdot  \rangle }}

\renewcommand{\phi}{\varphi}
\renewcommand{\emptyset}{\varnothing}

\newcommand{\bs}{\backslash}

\newcommand{\Spec}{\operatorname{Spec}}  %the Spectrum
\newcommand{\LSpec}{\operatorname{Spec}_L}  %the Length Spectrum
\newcommand{\Trace}{\operatorname{Trace}} %the Trace
\newcommand{\Fix}{\operatorname{Fix}} %Fixed point set
\newcommand{\Span}{\operatorname{Span}}
\newcommand{\Id}{\operatorname{Id}}
\newcommand{\SingSupp}{\operatorname{SingSupp}} %Singular Support
\newcommand{\Wave}{\operatorname{Wave}} %Wave invariant
\newcommand{\taumin}{\tau_{\rm{min}}} %length of shortest non-trivial closed geodesic

\newcommand{\mathE}{\mathcal{E}}

\newcommand{\wtTheta}{\widetilde{\Theta}}

\newcommand{\rank}{\operatorname{rank}}

\begin{document}

\newcommand{\spacing}[1]{\renewcommand{\baselinestretch}{#1}\large\normalsize}
\spacing{1.14}

\title{On the Poisson relation for compact Lie groups}
%\title[{\bf Draft: On the Poisson relation for compact Lie groups}]{On the Poisson relation for compact Lie groups}
%\ititle{On recovering the length spectrum of a bi-invariant metric from its Laplace spectrum}
%\title{On hearing the length spectrum of a compact Lie group}
%\title[{\bf Draft 3: Please Don't Distribute}]{On hearing the length spectrum of a compact Lie group}

%\dedicatory{In celebration of the seventy-fifth birthdays of my parents}
%\dedicatory{To my parents on the occasion of their seventy-fifth birthdays}
%\dedicatory{In loving memory of Baya}

%%%%%%%%%%%%%% CRAIG'S INFO  %%%%%%%%%%%%%
%\author[C. J. Sutton]{Craig J. Sutton$^\sharp$}
\author[C. J. Sutton]{Craig J. Sutton$^\sharp$}
\address{Dartmouth College\\ Department of Mathematics \\ Hanover, NH 03755}
\email{craig.j.sutton@dartmouth.edu}
\thanks{$^\sharp$ Research partially supported by a Simons Foundation Collaboration Grant}
%\thanks{$^\sharp$ Draft for Gordon, Perry, Schueth, Spatzier, Uribe \& Webb.}
%\thanks{$^\sharp$ Research partially supported by NSF grant DMS 0906168}

%58J50 Spectral Geometry; 
%53C20 Global Riem. Geom. (incl. pinching); 
%53C30 Homogeneous Mnflds; 
%53C35 Symm. Spcs.


%\date{May 8, 2016}
%\date{\today}

%%%%%%%%%%%%%%%%%%%%%%%%%%%%%%%%%%%%%%%%%%%%%%%%%%
%%%%%%%%%%%%%%%%%%%%%%   ABSTRACT  %%%%%%%%%%%%%%%%%%%%%
%%%%%%%%%%%%%%%%%%%%%%%%%%%%%%%%%%%%%%%%%%%%%%%%%%

\begin{abstract}
Intuition drawn from quantum mechanics and geometric optics raises the following long-standing question: can the length spectrum of a closed Riemannian manifold be recovered from its Laplace spectrum? The Poisson relation states that for any closed Riemannian manifold $(M,g)$ the singular support of the trace of its wave group---a spectrally determined tempered distribution---is contained in the set consisting of $\pm \tau$, where $\tau$ is the length of a smoothly closed geodesic in $(M,g)$. Therefore, in cases where the Poisson relation is an equality, we obtain a method for retrieving the length spectrum of a manifold from its Laplace spectrum. The Poisson relation is known to be an equality for sufficiently ``bumpy'' Riemannian manifolds and there are no known counterexamples.

%The Poisson relation is known to be an equality for sufficiently ``bumpy'' Riemannian manifolds; thereby, establishing that the length spectrum of a generic Riemannian metric is encoded in its Laplace spectrum. And, there are no known examples where the Poisson relation fails to be an equality. 
%However, among spaces possessing ``large'' isometry groups, it is mostly unknown whether the Poisson relation is an equality.

We demonstrate that the Poisson relation is an equality for a compact Lie group equipped with a generic bi-invariant metric. Consequently, the length spectrum of a generic bi-invariant metric (and the rank of its underlying Lie group) can be recovered from its Laplace spectrum. Furthermore, we exhibit a substantial collection $\mathscr{G}$ of compact Lie groups---including those that are either tori, simple, simply-connected, or products thereof---with the property that for each group $U \in \mathscr{G}$ the Laplace spectrum of any bi-invariant metric $g$ carried by $U$ encodes the length spectrum of $g$ and the rank of $U$. The preceding statements are special cases of results concerning compact \emph{globally} symmetric spaces for which the semi-simple part of the universal cover is split-rank. The manifolds considered herein join a short list of families of non-``bumpy''  Riemannian manifolds for which the Poisson relation is known to be an equality. %Together with the compact rank-one symmetric spaces and locally symmetric spaces of non-positive curvature, the manifolds considered herein are the only \emph{locally} symmetric spaces for which the Poisson relation is known to be an equality.
\end{abstract}

\maketitle

%\tableofcontents

%%%%%%%%%%%%%%%%%%%%%%%%%%%%%%%%%%%%%%%%%%%%%%%
%%%%%%%%%%%%%%%%%%%      INTRODUCTION   %%%%%%%%%%%%%%%%%
%%%%%%%%%%%%%%%%%%%%%%%%%%%%%%%%%%%%%%%%%%%%%%%

\section{Introduction and Proof of the Main Theorem}
The (\emph{Laplace}) \emph{spectrum} of a closed Riemannian manifold $(M,g)$ is 
the sequence 
%$\langle \lambda_k \rangle_{k=0}^{+\infty}$ 
$\lambda_0 = 0 \leq \lambda_1 \leq \lambda_2 \leq \cdots \leq \lambda_k \nearrow +\infty$ 
consisting of the eigenvalues  
of its associated Laplace operator $\Delta_g$, %acting on $L^{2}(M,\nu_g)$, 
where each eigenvalue is repeated according to its multiplicity. 
%and we will say that two Riemannian manifolds are \emph{isospectral} if their spectra are identical.
A central problem in spectral geometry is to understand the extent to which 
the geometry of a Riemannian manifold can be recovered from its spectrum.
While it is well known that the geometry of an arbitrary manifold cannot be 
completely recovered from its spectrum,
the asymptotic expansion of the heat trace %, $\sum_{j=0}^{\infty} e^{-t \lambda_j}$,  
about its singularity at zero reveals that dimension, volume and 
certain integrals of local geometric invariants are spectrally determined.
Inspired by intuitive arguments drawn from quantum mechanics and geometric optics, we have the following long-standing problem.
%Motivated by intuitive arguments drawn from quantum mechanics and geometric optics, it is a long-standing problem to determine whether the length spectrum of a manifold is determined by its Laplace spectrum, whereby the \emph{length spectrum} of a Riemannian manifold $(M,g)$ we mean the set $\Spec_L(M,g)$ consisting of the lengths of its smoothly closed geodesics.  

\begin{prob2}\label{prob:HearingLengthSpectrum}
The \emph{length spectrum} of a Riemannian manifold $(M,g)$ is the set, denoted by $\Spec_L(M,g)$, consisting of the lengths of the manifold's smoothly closed geodesics. Can the length spectrum of a Riemannian manifold be recovered from its Laplace spectrum? 
%Can the length spectrum of a Riemannian manifold be recovered from its Laplace spectrum? Whereby length spectrum, we mean the set $\Spec_L(M,g)$ consisting of the lengths of the smoothly closed geodesics. 
%That is, is the length spectrum an ``audible'' property of a Riemannian manifold?
\end{prob2}
\noindent
There are currently no known examples of isospectral manifolds having distinct length spectra. In fact, isospectral manifolds arising from Sunada's method necessarily have identical length spectra \cite{Sunada} (cf. \cite[Theorem 1.3]{GorMao}).
%We are not aware of any isospectral manifolds that have distinct length spectra. In fact, any isospectral manifolds arising from Sunada's method will have identical length spectra \cite{Sunada}.
%among the numerous examples of isospectral manifolds in the literature we have yet to demonstrate that there is an isospectral pair where the manifolds have distinct length spectra. In fact, any isospectral manifolds arising from the classical Sunada technique must have identical length spectra \cite{Sunada}.
%\noindent
%Here, the length spectrum of a Riemannian manifold $(M,g)$ is defined to be the set 
%$\Spec_{L}(M,g)$ consisting of the lengths of its smoothly closed geodesics.  

%\begin{folkconjecture}
%The length spectrum of a Riemannian manifold is determined by its Laplace spectrum.
%\end{folkconjecture}
%\noindent
%Here, the length spectrum of a Riemannian manifold $(M,g)$ is defined to be the set 
%$\Spec_{L}(M,g)$ consisting of the lengths of its smoothly closed geodesics.  

A natural approach to obtaining a positive answer to this problem is to use a suitable trace formula.
Indeed, using the so-called Poisson summation formula, and generalizations thereof, 
it has been demonstrated that Problem~\ref{prob:HearingLengthSpectrum} can be answered positively for flat manifolds \cite{Miatello-Rossetti} 
and Heisenberg manifolds \cite{Pesce}.
And, through the Selberg trace formula, one can see that the length spectrum 
of a Riemann surface may be recovered from its Laplace spectrum \cite{Huber1, Huber2, McKean}. 
%A more general approach to this problem emerges from the work of Colin de Verdi\`{e}re \cite{CdV},  
A more general approach to this problem emerges from the work of Chazarain \cite{Chazarain}  and Duistermaat and Guillemin \cite{DuGu} (cf. \cite{CdV}) where it is proven that the singular support of the \emph{trace of the wave group} of a manifold is contained in the set of periods of the geodesic flow (Section~\ref{Sec:DGTraceFormula}):
$$\SingSupp (\Trace (e^{-i t \sqrt{\Delta_g}} )) \subseteq \Spec_L^{\pm}(M,g) \equiv \{ \pm \tau : \tau \in \LSpec(M,g) \}.$$ 
This containment is referred to as the \emph{Poisson relation} and, since the trace of the wave group of a manifold is a spectrally determined distribution,
we arrive at the following intriguing question.

\begin{prob2}\label{prob:PoissonRelation}
Is the Poisson relation an equality for all Riemannian manifolds?
\end{prob2}

\noindent 
It is clear that an affirmative answer to Problem~\ref{prob:PoissonRelation} provides an affirmative answer to Problem~\ref{prob:HearingLengthSpectrum}. 
%Indeed, if the Poisson relation is an equality for the Riemannian manifold $(M,g)$, then we may recover the length spectrum of $(M,g)$ from its Laplace spectrum by computing the singular support of the trace of its wave group. 
Hence, the remainder of this article will be concerned with understanding the extent to which Problem~\ref{prob:PoissonRelation} has a positive answer.

For \emph{clean} manifolds (Definition~\ref{dfn:Clean}) it is in principle
%(see p.~\pageref{def:Clean}) it is in principle
possible to use the trace formula of Duistermaat and Guillemin (Equation~\ref{eqn:Asymptotic1}) to resolve whether the Poisson relation is an equality: the main obstacle is determining whether cancellations can occur in the trace formula (Section~\ref{Sec:DGTraceFormula}). By observing that sufficiently ``bumpy'' manifolds are clean and do not raise the specter of cancellation, Duistermaat and Guillemin were able to conclude that Problem~\ref{prob:PoissonRelation} (and, consequently, Problem~\ref{prob:HearingLengthSpectrum}) has an affirmative answer for a generic Riemannian manifold \cite[p. 61]{DuGu}. It is also a straightforward consequence of their work that Problem~~\ref{prob:PoissonRelation} has a positive answer for compact rank-one symmetric spaces and, more generally, $C_\tau$-manifolds (Section~\ref{Sec:Results}). In subsequent work, Duistermaat, Kolk and Varadarajan \cite{DuKoVa} (cf. \cite{Gangoli}) answered Problem~\ref{prob:PoissonRelation} affirmatively for compact \emph{locally} symmetric spaces of the \emph{non-compact type} (e.g., closed manifolds of constant negative sectional curvature) %(cf. \cite[Sec. 10]{Prasad-Rapinchuk})
and it follows from a result of Guillemin that the Poisson relation is an equality for negatively curved manifolds \cite[Theorem 4]{Guillemin}. In establishing the last two statements the main chore is showing that the spaces are clean, while the issue of cancellations can be resolved in each case by well-known reasons (cf. \cite[Sec. 10]{Prasad-Rapinchuk}).

To the best of our knowledge, the preceding is an exhaustive account of the progress that has been made on Problems~\ref{prob:HearingLengthSpectrum} and \ref{prob:PoissonRelation} prior to the results presented in this article. In particular, setting aside compact rank-one symmetric spaces and flat tori, there has been no previous progress on these questions regarding \emph{globally} symmetric spaces, where a priori the occurrence of cancellations in the trace formula is a distinct possibility among the non-flat higher-rank symmetric spaces.
%looms large among non-flat higher-rank symmetric spaces. 

Compact Lie groups equipped with bi-invariant metrics form a natural and widely considered class of globally symmetric spaces (Section~\ref{Sec:Results}). After demonstrating that all globally symmetric spaces are clean (Theorem~\ref{thm:SymmSpcsCIH})---a property not enjoyed by every homogeneous space (Theorem~\ref{thm:UncleanMetrics})---we use the trace formula of Duistermaat and Guillemin to demonstrate that for any compact Lie group $U$ the Poisson relation is an equality for a generic bi-invariant metric on $U$ and, consequently, conclude that the length spectrum of a compact Lie group equipped with a generic bi-invariant metric can be recovered from its Laplace spectrum (Corollary~\ref{cor:MainResult} $(1)$). Furthermore, we observe that the spectrum of a generic bi-invariant metric encodes the rank of its underlying Lie group (Theoerm~\ref{thm:HearingRank}). A more careful analysis allows us to exhibit an infinite collection $\mathscr{G}$ of compact Lie groups with the property that for each $U \in \mathscr{G}$ the Laplace spectrum of any bi-invariant metric $g$ supported by $U$ encodes the length spectrum of $g$ and the rank of $U$ (Corollary~\ref{cor:MainResult} $(2)$ and Theoerm~\ref{thm:HearingRank}). The set $\mathscr{G}$ properly contains the collection of groups that are either simple, simply-connected, tori or products thereof. The preceding statements are a strong indication that one should be able to recover the length spectrum of an arbitrary bi-invariant metric from its Laplace spectrum.
%of any compact Lie group equipped with a bi-invariant metric from its Laplace spectrum.

%Specifically, recall that every compact Lie group is of the form $$U = \Gamma \backslash (T \times \widetilde{U}_{ss}),$$ where $T$ is a torus, $\widetilde{U}_{ss}$ is a compact simply-connected semi-simple Lie group and $\Gamma$ is a subgroup of the center of $T \times \widetilde{U}_{ss}$ (having trivial intersection with $T$). Then, the collection $\mathscr{G}$ consists of all compact groups $U$ where $\Gamma$ is trivial or satisfies certain constraints that depend on $\widetilde{U}_{ss}$. In particular, every compact Lie group that is either simple, simply-connected, a torus, or a product thereof is in $\mathscr{G}$. 
%All of the above is a strong indication that the Laplace spectrum of an arbitrary bi-invariant metric encodes its length spectrum and the rank of its underlying Lie group. 
%length spectrum of an arbitrary bi-invariant metric and the rank of its underlying Lie group are encoded in the Laplace spectrum.

%All of this is a strong indication that one should be able to recover the length spectrum of an arbitrary bi-invariant metric from its Laplace spectrum.
%In addition to demonstrating that in many cases the length spectrum can be recovered 
%from the spectrum of a bi-invariant metric,

Corollary~\ref{cor:MainResult} is a special case of Theorem~\ref{thm:MainResult}, which is a more general statement concerning compact symmetric spaces for which the non-Euclidean part of the universal cover is split-rank. As we will explain, the key to Theorem~\ref{thm:MainResult} is Theorem~\ref{thm:IntroMod4}, which reveals that for certain symmetric spaces the Morse index modulo $4$ of a closed geodesic depends only on the length $\tau$ of the geodesic and the dimension of the corresponding component of the fixed-point set of $\Phi_\tau$, the time-$\tau$ map of the geodesic flow. 

%The aforementioned results concerning Lie groups are a special case of Theorem~\ref{thm:MainResult}, which is a more general statement concerning compact symmetric spaces for which the non-Euclidean part of the universal cover is split-rank. As we will explain, the key to Theorem~\ref{thm:MainResult} and Corollary~\ref{cor:MainResult} is Theorem~\ref{thm:IntroMod4}, which reveals that for certain symmetric spaces the Morse index modulo $4$ of a closed geodesic depends only on the length $\tau$ of the geodesic and the dimension of the corresponding component of the fixed-point set of $\Phi_\tau$, the time-$\tau$ map of the geodesic flow. 
%As an application of our results, we show that the spectrum of a generic bi-invariant metric or any bi-invariant metric supported by $U \in \mathscr{G}$ encodes the rank of its underlying Lie group (Theorem~\ref{thm:HearingRank}).
 
Before reviewing the trace formula of Duistermaat and Guillemin in Section~\ref{Sec:DGTraceFormula} and providing a precise statement of our results along with an outline of the supporting arguments in Sections~\ref{Sec:Results} and \ref{SubSec:MainProof}, we pause briefly to provide the reader with a partial motivation for Problem~\ref{prob:HearingLengthSpectrum} based on considerations in quantum mechanics.
 
%Before providing precise statements of our results and offering an outline of the argument in Sections~\ref{Sec:DGTraceFormula} and \ref{Sec:Results}, we pause to provide the reader with a partial motivation for Problem~\ref{prob:HearingLengthSpectrum} based on considerations in quantum mechanics.

%%%%%%%%%%%%%%%%%%%%%%%
\subsection{Bohr's Correspondence Principle and the Length Spectrum}
Consider a free particle in a Riemannian manifold $(M,g)$.
Classical mechanics takes the viewpoint that the evolution of this particle is deterministic. 
Indeed, in classical mechanics the state space of the system is taken to be the co-tangent bundle 
$T^{*}M$ equipped with the symplectic form $\omega_g$ induced by the metric $g$. 
Associated to the observable $p: T^{*}M \to \R$ given by $\zeta \mapsto \|\zeta\|_{g}^2$ 
is the Hamiltonian vector field $X$ defined via the relationship 
$$dp(\cdot)  = \omega_g(X, \cdot),$$
which encodes Hamilton's equations of motion.
The flow generated by $X$ is the geodesic flow $\Phi : \R \times T^{*}M \to T^{*}M$ and 
its orbits describe the motion of the free particle under consideration.

On the other hand, quantum mechanics takes a probabilistic viewpoint. 
%the viewpoint that we can only know the probability that our particle is in some subset $U \subseteq M$.
Indeed, the state space of the quantum system is taken to be $L^{2}(M, \nu_g)$, 
where $\nu_g$ is the measure induced by $g$, and an $L^2$-normalized function $f \in L^{2}(M,\nu_g)$ is interpreted as a probability density; that is, for any measurable set $E \subset M$ the quantity $\int_E |f(x)|^2 d \nu_g$ represents 
%the probability of finding the particle in the set $E$.
the probability that the state of the particle is in the set $E$.
The evolution of a probability density $f \in L^{2}(M, \nu_g)$ is governed by Schr\"{o}dinger's equation:
$$\left\{
\begin{array}{l}
i\hbar \frac{\partial}{\partial t}\Psi(t,x) = -\hbar^2 \frac{1}{2}\Delta_g \Psi(t,x)\\
 \Psi(0,x) = f(x)
\end{array}
\right.
$$
%$$i\hbar \frac{\partial \Psi(t,x)}{\partial t} = -\hbar^2 \frac{1}{2}\Delta_g \Psi(t,x), \, \Psi(0,x) = f(x),$$
whose solution is given by $\Psi(t,x) = e^{i \hbar  t \frac{1}{2}\Delta_g} f(x)$.
Therefore, the \emph{Schr\"{o}dinger flow}
$S_g^{\hbar}(t) \equiv e^{i \hbar t \frac{1}{2}\Delta_g}$,
is the quantum-mechanical analogue of the geodesic flow and, via the functional calculus, we see that the Schr\"{o}dinger flow 
is completely determined by the spectrum of $(M,g)$.
%\begin{figure}[h]\label{fig:ClassicalQuantum}
%\begin{center}
%\begin{tabular}{|c||c|c|} \hline
%& {\bf Classical Mechanics} & {\bf Quantum Mechanics} \\ \hline
%State Space & $(X, \omega)$ symplectic manifold & $\mathcal{H}_{X}$ Hilbert space \\ \hline
%Observables & $H : X \to \R$ Hamiltonian & $\widehat{H} : \mathcal{H}_X \to \mathcal{H}_X$ self-adjoint operator \\ \hline
%Dynamics & $dH = \omega(X_H, \cdot)$ Hamilton's equations of motion & $i \hbar \frac{\partial \Psi}{\partial t} = \widehat{H} %\Psi$ Schr\"{o}dinger's Equation\\
%                 & $\Phi : \R \times X \to X$ flow of $X_H$ & $e^{-it\hbar^{-1} \widehat{H}}$ wave group \\ \hline
%\end{tabular}
%\caption{Classical vs. Quantum Mechanics: the general picture} 
%\end{center}
%\end{figure}

Now, the correspondence principle is the assertion that as $\hbar \to 0$ (i.e., as the ``characteristic action'' of the system becomes large relative to $\hbar$) the quantum dynamical system will converge (in some sense) to the classical system; so that, for $\hbar$ small, the quantum system will reflect salient features of the corresponding classical dynamical system. As the periods of periodic orbits are a fundamental feature of the geodesic flow (i.e., the classical dynamics) and the spectrum of a Riemannian manifold determines the Schr\"{o}dinger flow (i.e., quantum dynamics) it is natural to wonder whether the length spectrum can be recovered from its Laplace spectrum.
%a widely believed folk-conjecture that the length spectrum of a Riemannian manifold is determined by its spectrum. 

Mathematical motivation for Problem~\ref{prob:HearingLengthSpectrum} can be found in the work of Colin de Verdi\`{e}re \cite{CdV} and Duistermaat and Guillemin  \cite{DuGu} where it is demonstrated that for a generic (i.e., sufficiently ``bumpy'') manifold the spectrum determines the Laplace spectrum. As it is germane to our results (Section~\ref{Sec:Results}), we now summarize the approach taken by Duistermaat and Guillemin, which utilizes $\sqrt{\Delta_g}$ rather than the semi-classical parameter $\hbar$ used above.
%\footnote{If one thinks of $\hbar$ as an operator of degree $-1$, these approaches are equivalent.}

%%%%%%%%%%%%%%%%%%%%%%
\subsection{The Trace Formula and ``Bumpy'' Metrics}\label{Sec:DGTraceFormula}
%\subsection{The Trace Formula of Duistermaat and Guillemin}\label{Sec:DGTraceFormula}
%Rather than adopting the semi-classicial approach outlined above, 
%Duistermaat and Guillemin utilize $\sqrt{\Delta}$ 
%instead of introducing the semi-classicial parameter $\hbar$, which should be thought of as an operator of degree $-1$.
The \emph{wave group} of a Riemannian manifold $(M,g)$ is the family of unitary operators $U_{g}(t) \equiv e^{-i t \sqrt{\Delta_g}}$. The operators $U_g(t)$ are not trace class; however, for any Schwarz function $f(t)$, the operator $U_f \equiv \int_{-\infty}^{\infty} f(t) U_g(t) \, dt$ is of trace class. We then define the \emph{trace of the wave group} $\Trace(U_g(t))$ to be the tempered distribution:
$$f \in \mathscr{S}(\R) \mapsto \Trace(U_f) \in \R.$$
The distribution kernel of the wave group of $(M,g)$ is given by $$U_g(t, x, y) = \sum_{j=0}^{\infty} e^{-it \sqrt{\lambda_j}} \phi_{j}(x) \overline{\phi_{j}(y)},$$ where $\{\phi_{j}\}_{j=0}^{\infty}$ is an orthonormal basis of $\Delta_g$-eigenfunctions with $\Delta_g \phi_j = \lambda_j \phi_j$. It follows that 
$$ \Trace(U_g(t)) =  \int_M U_g(t, x, x) \; d \nu_g = \sum_{j=0}^{\infty} e^{-i t \sqrt{\lambda_j}},$$
which is the Fourier transform of the ``spectral distribution'' $\sigma (t) = \sum_{j=0}^{\infty} \delta (t - \sqrt{\lambda_{j}})$.
%Hence, its trace, which we will refer to as the \emph{wave trace}, is the spectrally determined tempered distribution on $\R$ given by $$ \Trace(U_g(t)) =  \int_M U_g(t, x, x) \; d \nu_g = \sum_{j=0}^{\infty} e^{-i t \sqrt{\lambda_j}},$$which is the Fourier transform of the ``spectral distribution'' $\sigma (t) = \sum_{j=0}^{\infty} \delta (t - \sqrt{\lambda_{j}})$.
It has been demonstrated by Chazarain \cite[Theorem I]{Chazarain} and Duistermaat and Guillemin \cite[Corollary 1.2]{DuGu} that the singular support of the trace of the wave group, which we will denote by $\SingSupp (\Trace(U_g(t)))$, 
is a subset of the periods of the periodic orbits of the geodesic flow: %subset of the length spectrum:
\begin{eqnarray}\label{eqn:PoissonRelation}
\SingSupp (\Trace(U_g(t))) \subseteq  \Spec_L^{\pm}(M,g) \equiv \{ \pm \tau : \tau \in \Spec_{L}(M,g) \}.
\end{eqnarray}
This containment is known as the \emph{Poisson relation} and, in light of Problem~\ref{prob:HearingLengthSpectrum}, understanding whether Equation~\ref{eqn:PoissonRelation} is an equality for all manifolds is a fascinating question (Problem~\ref{prob:PoissonRelation}).
%\begin{rem}\label{rem:Smoothness}
%The trace of the wave group, $\Trace(U_g(t))$, and its real part, $\Trace(\cos(t\sqrt{\Delta_g}))$, have the same singular support. Therefore, since $\Trace(\cos(t\sqrt{\Delta_g}))$ is an even distribution, we conclude that $\tau \in \SingSupp(\Trace(U_g(t)))$ if and only if $ -\tau \in \SingSupp( \Trace(U_g(t)) )$. It follows that we can establish equality in Equation~\ref{eqn:PoissonRelation} by showing $\Spec_L(M,g)$ is contained in $\SingSupp (\Trace(U_g(t)))$.
%\end{rem}
For manifolds satisfying the so-called ``clean intersection hypothesis,'' the trace formula of Duistermaat and Guillemin provides a means through which we can hope to analyze this problem.

\begin{dfn}\label{dfn:Clean}
Let $(M,g)$ be a closed Riemannian manifold and, for each $t \in \R$, let $\Phi_t(\cdot) = \Phi(t, \cdot)$ be the time-$t$ map of the associated geodesic flow.
A period $\tau \in \Spec_L^{\pm}(M,g)$ of the geodesic flow is said to have a \emph{clean fixed-point set} or to be \emph{clean} if %\label{def:Clean} if
\begin{enumerate}
\item the fixed-point set of $\Phi_\tau$, denoted $\Fix(\Phi_\tau)$, is a disjoint union of finitely many closed manifolds
$\Theta_1, \ldots , \Theta_r$;
\item for each $u \in \Fix(\Phi_\tau)$ the fixed point set of $D_u\Phi_{\tau}$ is precisely equal to $T_u \Fix(\Phi_\tau)$.
That is, if $J(t)$ is a periodic Jacobi field along the geodesic $\gamma_u$, with $\gamma_u'(0) = u$, then 
$(J(0), J'(0)) \in T_u \Fix(\Phi_\tau)$.
\end{enumerate}
Equivalently, $\tau \in \Spec_L^{\pm}(M,g)$ is to have a \emph{clean fixed-point set} or to be \emph{clean}, if $|\tau|$ is a non-degenerate critical value of the energy functional $E: \Omega(M,g) \to \R$ on the loop space of $(M,g)$. We will agree to say that a Riemannian manifold $(M,g)$ has \emph{clean geodesic flow} or is \emph{clean}, if every $\tau \in \Spec_L^{\pm}(M,g)$ is clean. 
\end{dfn}

\begin{rem}
Clearly, $\tau \in \Spec_{L}^{\pm}(M,g)$ is clean if and only if $-\tau$ is clean.
\end{rem}

Under the assumption that the period $\tau \in \Spec_{L}^{\pm}(M,g)$ 
has a clean fixed-point set, Duistermaat and Guillemin determined that there is an interval $I_{\tau} \subset \R$ for which $I_{\tau} \cap \LSpec(M,g) = \{\tau \}$ and such that on $I_\tau$ the wave trace can be expressed as a sum of compactly supported distributions 
\begin{eqnarray}
\Trace(U_g(t)) = \beta^{\rm{even}}(t - \tau) + \beta^{\rm{odd}}(t-\tau),
\end{eqnarray}
where $\beta^{\rm{even}} (x)$ (respectively, $\beta^{\rm{odd}}(x)$) is a distribution determined by 
the even-dimensional (respectively, odd-dimensional) components of $\Fix (\Phi_\tau)$
and whose only possible singularity occurs at $x = 0$ \cite[Theorem 4.5]{DuGu}.
Of particular interest is the fact that the Fourier transforms of 
$\beta^{\rm{even}}(x)$ and $\beta^{\rm{odd}}(x)$ are given by smooth functions $\alpha^{\rm{even}}(s)$ and 
$\alpha^{\rm{odd}}(s)$, respectively, 
possessing the following asymptotic expansions at infinity:
%that possess asymptotic expansions of the form

\begin{eqnarray}\label{eqn:Asymptotic1}
\alpha^{\bullet}(s) &\stackrel{s \to + \infty}{\sim}& 
\sum_{k=0}^{\infty} \operatorname{Wave}_{k}^{\bullet}(\tau) s^{(D_{\bullet} -2k -1)/2},
\end{eqnarray}

%and

%\begin{eqnarray}\label{eqn:Asymptotic2}
%\alpha^{\rm{odd}}(s) &\stackrel{s \to + \infty}{\sim}&
% \sum_{k=0}^{\infty} \operatorname{Wave}_{k}^{\rm{odd}}(\tau) s^{(D_{\rm{odd}} -2k -1)/2},
%\end{eqnarray}

\noindent
where $\bullet$ denotes ``even'' or ``odd'' and $D_{\bullet}$ is the maximum taken over the 
dimensions of the $\bullet$-dimensional components of $\Fix(\Phi_\tau)$ (see \cite[Theorem 4.5]{DuGu}). 

%where $D_{\rm{even}}$ (respectively, $D_{\rm{odd}}$) equals the maximum taken over the 
%dimensions of the even-dimensional (respectively, odd-dimensional) components of 
%$\Fix(\Phi_\tau)$ (see \cite[Theorem 4.5]{DuGu}). 

The coefficients $\operatorname{Wave}_{k}^{\rm{even}}(\tau)$ and $\operatorname{Wave}_{k}^{\rm{odd}}(\tau)$ in the asymptotic expansions above are complex numbers known as the \emph{$k$-th wave invariants} of $\tau$, and, in light of the definition of the wave group, these are spectral invariants of the Riemannian manifold $(M,g)$. The moral of the trace formula is that the faster $\alpha^{\rm{even}}$ (resp. $\alpha^{\rm{odd}}$) decays at infinity the less singular $\beta^{\rm{even}}(t -\tau)$ (respectively, $\beta^{\rm{odd}}(t-\tau)$) is at $\tau$. %\footnote{The terms of negative order in the asymptotic expansion of $\alpha^{\rm{even}}$ (respectively, $\alpha^{\rm{odd}}$) contribute terms to $\beta^{\rm{even}}$ (respectively, $\beta^{\rm{odd}}$)  which are higher-order anti-derivatives of the Dirac-delta function.}
Consequently, the trace formula informs us that a period $\tau$ with a clean fixed-point set is in $\SingSupp (\Trace(U_g(t)))$ if and only if at least one of its wave invariants is non-zero. That is, $\Trace(U_g(t))$ is smooth at a clean period $\tau$ if and only if the Fourier transform of its restriction to $I_\tau$ is asymptotic to zero at infinity. It will be useful for us to recall that, since the subprincipal symbol of $\sqrt{\Delta}$ is zero \cite[p.58]{DuGu}, the $0$-th wave invariants of a clean period $\tau$ are given by the following formula \cite[Equation 4.8]{DuGu}:

\begin{eqnarray}\label{eqn:WaveInvariants1}
\Wave_{0}^{\bullet}(\tau) = \left( \frac{1}{2\pi i}\right)^{\frac{D_{\bullet} -1}{2}} 
\sum_{\stackrel{j = 1}{\dim \Theta_j = D_{\bullet}}}^{r} i^{-\sigma_j} \int_{\Theta_j} d \mu_j^{\tau},
\end{eqnarray}

%\noindent
%and 

%\begin{eqnarray}\label{eqn:WaveInvariants2}
%\Wave_{0}^{\rm{odd}}(\tau) = \left( \frac{1}{2\pi i}\right)^{\frac{D_{\rm{odd}} -1}{2}} 
%\sum_{\stackrel{j = 1}{\dim \Theta_j = D_{\rm{odd}}}}^{r} i^{-\sigma_j} \int_{\Theta_j} d \mu_j,
%\end{eqnarray}

\noindent
where $\bullet$ denotes ``even'' or ``odd,'' $\sigma_j$ is equal to the Morse index (in the free loop space) of a closed geodesic of length $|\tau|$ with initial velocity in the component $\Theta_j$ and 
$\mu_j^\tau$ is a canonical positive measure on the submanifold $\Theta_j$, 
which we will refer to as the \emph{Duistermaat-Guillemin measure}\label{DGMeasure} 
\cite[Theorem 4.5 and p. 69-70]{DuGu} (cf. \cite[Appendix]{BPU}).
In general, if $\tau$ is clean, the $k$-th wave invariant of $\tau$ is of the form 
$\Wave_k^{\bullet}(\tau) = \sum_{j = 1}^{r} C_{k,j}^{\bullet}$, 
where $\bullet$ denotes ``even'' or ``odd'' and $C_{k,j}^{\bullet} \in \C$ 
is a constant determined by the component $\Theta_j$ of $\Fix(\Phi_\tau)$.\footnote{We note that Chazarain obtains similar results to those of Duistermaat and Guillemin; however, he does not provide the explicit formula for the leading term of the asymptotic expansions given in Equation~\ref{eqn:WaveInvariants1}.} 
%and \ref{eqn:WaveInvariants2}.}
It is clear from the shape of the wave-invariants that for a clean period $\tau$ the issue 
of whether it is in the singular support is resolved according to whether cancellations occur
in each $\Wave_{k}^{\bullet}(\tau)$.

\begin{rem}\label{rem:TimeReversal}
For each period $\tau \in \Spec_{L}^{\pm}(M,g)$ with a clean fixed-point set, the Duistermaat-Guillemin densities $\mu^{\tau}$ and $\mu^{-\tau}$ associated to $\Fix(\Phi_\tau) = \Fix(\Phi_{-\tau})$ agree \cite[Lemma A.3]{BPU}. Therefore, we obtain $\Wave_{0}^{\bullet}(\tau) = \Wave_{0}^{\bullet}(-\tau)$ for each $\tau \in \Spec_{L}^{\pm}(M,g)$. 
\end{rem}

%We will agree to say that a Riemannian manifold is \emph{clean} 
%if each length in its length spectrum is clean. 
As we noted earlier, a Riemannian manifold is clean if each period of the geodesic flow is clean. Duistermaat and Guillemin have demonstrated that ``cleanliness'' is a generic property among Riemannian manifolds \cite[p. 61]{DuGu} and that generically the $0$-th wave invariant does not vanish.\footnote{Let $\mathscr{R}(M)$ denote the space of Riemannian metrics on $M$. We will say that a metric property is \emph{generic}, if the set of all metrics having this characteristic  contains a residual set.} Specifically, they argue that sufficiently ``bumpy'' metrics (cf. \cite{Abraham, Anosov}) are clean and that up to orientation these metrics have exactly one geodesic of a given length, so that the issue of cancellation does not arise for such spaces. Consequently, for sufficiently ``bumpy'' metrics, the wave-invariant $\Wave_{0}^{\rm{odd}}(\tau)$ never vanishes, thereby, establishing that for a generic Riemannian manifold Problem~\ref{prob:PoissonRelation} and, hence, Problem~\ref{prob:HearingLengthSpectrum} can be answered affirmatively.

%%%%%%%%%%%%%%%%%%%%%%%%%
\subsection{Can You Hear the Length Spectrum of a Symmetric Space?}\label{Sec:Results}
%\subsection{The Trace Formula in the Symmetric Setting}\label{Sec:Results}
%\subsection{Beyond Generic Manifolds \& Statement of Results}\label{Sec:Results}
The antithesis of a ``bumpy'' manifold is a (globally) symmetric space. 
A priori, it is not clear that symmetric spaces are clean and, 
even in the event that they are clean, the threat of cancellations in 
the trace formula looms large as $\Fix(\Phi_\tau)$ will consist of 
many components for higher-rank symmetric spaces, which is in sharp contrast with 
the sufficiently ``bumpy'' spaces discussed previously.\footnote{It is known that locally 
symmetric spaces of the non-compact type (e.g., Riemann surfaces) are clean \cite{DuKoVa}, 
but in this case cancellations cannot occur in the trace formula since the Morse index of a closed geodesic 
is always zero in this setting.}
However, it is clear that the compact rank-one symmetric spaces
(i.e, $S^n$, $\R P^n$, $\C P^n$, $\mathbb{H} P^n$ and $Ca^2$), 
which, henceforth, will be referred to as CROSSes, are clean.
This follows immediately from the fact that for a CROSS the non-trivial orbits of the geodesic flow are all periodic with common primitive period $\tau_0$. 
%so that for any $\tau$ in the length spectrum, we have that $\Fix(\Phi_\tau)$ is the entire unit tangent bundle.
In Section~\ref{SubSec:GeodSymmSpaces}  we address the case of higher-rank
symmetric spaces by exploiting the fact that 
(1) all geodesics in a symmetric space can be conjugated into a maximal flat
and (2) in a symmetric space all periodic Jacobi fields are restrictions of Killing fields,
which establishes that all compact symmetric spaces are clean.
%We demonstrate that all symmetric spaces are clean by exploiting the fact that 
%(1) all geodesics in a symmetric space can be conjugated into a maximal flat
%and (2) in a symmetric space all periodic Jacobi fields are restrictions of Killing fields,
%we will see in Section~\ref{SubSec:GeodSymmSpaces} that all compact symmetric spaces are clean.

\begin{thm}\label{thm:SymmSpcsCIH}
A compact globally symmetric space is clean.
\end{thm}

This theorem tells us that in addition to the heat invariants, we have the wave invariants at our disposal in addressing inverse spectral problems that involve symmetric spaces.
One might expect that all homogeneous manifolds are clean; unfortunately, this is already false for left-invariant metrics on $\SO(3)$ and $S^3$. 

\begin{thm}\label{thm:UncleanMetrics}
Within the class of left-invariant naturally reductive metrics on $\SO(3)$ (respectively $S^3$, in which case these are the Berger metrics), the clean metrics contain a residual set; however, the collection of unclean metrics is dense and contains certain normal homogeneous metrics. %(In the case of $S^3$, the left-invariant naturally reducitve metrics are Berger metrics.)
\end{thm}

%Indeed, in Appendix~\ref{Sec:UncleanMetrics} we exhibit left-invariant naturally reductive metrics on $\SO(3)$ that are not clean: there are lengths $\tau$ in the length spectrum for which $\Fix(D_u \Phi_\tau) \neq T_u \Fix(\Phi_\tau)$ (see Definition~\ref{dfn:Clean} (2)). Gornet has also discovered locally homogeneous metrics that fail to be clean for the same reason \cite{Gornet}. We are not aware of any examples where condition (1) of Definition~\ref{dfn:Clean} fails.

\noindent 
As we will see through an explicit computation of the geodesic flow, every left-invariant naturally reductive metric on $\SO(3)$ (respectively $S^3$) satisfies condition (1) of Definition~\ref{dfn:Clean}. The geodesic flow of the unclean metrics possess periods $\tau$ for which condition $(2)$ of Definition~\ref{dfn:Clean} is not met.
%\footnote{Although, Theorem~\ref{thm:UncleanMetrics} is stated for $\SO(3)$ the same statement holds for the left-invariant naturally reductive metrics on $S^3$; i.e., the Berger metrics.} 
Gornet has also discovered locally homogeneous metrics that fail to be clean for the same reason \cite{Gornet}. Therefore, cleanliness cannot be taken for granted even among ``nice'' metrics.

Returning to CROSSes, since for each period $\tau$ of the geodesic flow 
the fixed-point set of $\Phi_\tau$ is the entire unit tangent bundle, we see from 
Equation~\ref{eqn:WaveInvariants1} that $\Wave_0^{\rm{odd}}(\tau) \neq 0$.
Therefore, the Poisson relation is an equality for CROSSes and it follows that the length spectrum of a CROSS can be recovered from its spectrum by computing the singular support of the trace of its wave group. More generally, if $(M,g)$ is any manifold in which all non-trivial geodesics are closed and have a common primitive period $\tau$---such a Riemannian manifold is commonly referred to as a \emph{$C_\tau$-manifold} (cf. \cite{Besse})---then the Poisson relation is an equality.

In light of the fact that the conjecture is true for CROSSes, it is natural to wonder whether it 
is valid for every compact symmetric space.

A compact irreducible symmetric space $(M = G/K, g)$  comes in one of two flavors:
\begin{itemize}
\item {\bf Type I:} the isometry group of $M$ is a compact simple Lie group (e.g., CROSSes and Grassmannians)
\item {\bf Type II:} $M$ is a compact simple Lie group $U$ equipped with a bi-invariant metric,
in which case $G = U \times U$ and $K = \Delta U$. 
\end{itemize}
In both cases the symmetric metric on $M$ is, up to scaling, the \emph{standard metric} $g_0$ induced by the restriction of the negative of the Killing form $B_\germ{g}$ of $\germ{g}$ to an $\Ad(K)$-invariant complement of $\germ{K}$ in $\germ{g}$, where $\germ{g}$ and $\germ{K}$ denote the Lie algebras of $G$ and $K$, respectively.

In general, a compact symmetric space $(M = G/K, g)$ is of the form 
\begin{equation}\label{eqn:HomogeneityType} 
M = %\Gamma \backslash (\R^\ell \times \widetilde{M}_1 \times \cdots \times \widetilde{M}_q) = 
\Gamma \backslash (M_0 \times \widetilde{M}_1 \times \cdots \times \widetilde{M}_q),
\end{equation}

\noindent
where $M_0$ is a compact torus, $\widetilde{M}_j$ is a simply-connected  
compact irreducible symmetric space, 
the metric $g$ is induced by the metric $h \times c_1g^1_0 \times \cdots \times c_qg^q_0$ on 
$M_0 \times \widetilde{M}_1 \times \cdots \times \widetilde{M}_q$ for some choice of constants $c_j >0$
and flat metric $h$ on $M_0$,
and $\Gamma$ is a discrete subgroup of 
the \emph{center} of $M_0 \times \widetilde{M}_1 \times \cdots \times \widetilde{M}_q$ 
(see Section~\ref{SubSec:TheCenter}):

\begin{equation}\label{eqn:CenterOfSymmetricSpace}
Z(M_0 \times \widetilde{M}_1 \times \cdots \times \widetilde{M}_q) = 
M_0 \times Z(M_1) \times \cdots \times Z(M_q).
\end{equation}

\noindent 
After replacing $M_0$ with $M_0/(\Gamma \cap M_0)$ we will assume, henceforth,
that $\Gamma \cap M_0$ is trivial.
It then follows that in each dimension there are only finitely many homogeneity types of compact symmetric spaces 
(cf. \cite[Lemma 2.8]{GorSut}).
The compact simply-connected space 
$\widetilde{M}_{\rm{cpt}} \equiv \widetilde{M}_1 \times \cdots \times \widetilde{M}_q$ 
is the \emph{non-Euclidean} part of the universal cover of $M$.
The space $M$ is said to be of \emph{compact type} precisely when $M_0$ is trivial; i.e., $M$ 
has no Euclidean factors.
We note that the non-Euclidean part of the metric is Einstein precisely when 
$c_1 = c_2 = \cdots = c_q$ \cite[Theorem 7.74]{Besse2}.
Ignoring the metric $g$, we refer to Equation~\ref{eqn:HomogeneityType} as the \emph{homogeneity type} 
of the symmetric space.
Given a homogeneity type $M = \Gamma \backslash (M_0\times \widetilde{M}_1 \times \cdots \times \widetilde{M}_q)$ of a compact symmetric space, the space of symmetric metrics on $M$ is 
the finite dimensional space

\begin{equation}\label{eqn:ModuliSpaceOfSymmerticMetrics}
\mathscr{R}_{\rm{symm}}(M) \equiv \mathcal{S}^+(d) \times \R_{+} \times \cdots \times \R_{+},
\end{equation}

\noindent
where $d$ is the dimension of the torus $M_0$ and $\mathcal{S}^+(d)$ is the space of positive definite real symmetric $d \times d$-matrices. We will let $\mathscr{P}(M)$ be the set of metrics in $\mathscr{R}_{\rm{symm}}(M)$ for which the Poisson relation (Equation~\ref{eqn:PoissonRelation}) is an equality and we let  
$$\mathscr{W}_0(M) \subseteq \mathscr{P}(M)$$ 
denote the collection of metrics $g$ such that for each period $\tau$ of the geodesic flow the leading term, $\Wave_{0}^{\rm{lead}}(\tau) \neq 0$, of the asymptotic expansion given in Equation~\ref{eqn:Asymptotic1} is non-zero.

Finally, we recall that a symmetric space $M = G/K$ 
is said to be \emph{split-rank} if $\operatorname{rank} G = \operatorname{rank} M + \operatorname{rank} K$
or, equivalently, the restricted roots of $M$ (see Section~\ref{SubSec:RestrictedRoots}) 
all have even multiplicity \cite[Theorem VI.4.3]{Loos}.
By the classification of symmetric spaces the irreducible compact simply-connected split-rank spaces are the simple Lie groups, spaces of type $A_{2n+1}^{H}$ ($n \geq 1$) and the exceptional space $E_{6(-36)}$, 
where we have adopted the notation of Loos (cf. \cite[Theorem~VI.4.4 and Tables 4 \& 8]{Loos}), 
and all compact split-rank spaces are finitely covered by products of these irreducible factors.

With these preliminaries out of the way, we now state our main result.% which, in a sense, 
%establishes the conjecture for over ``half'' of the split-rank symmetric spaces.

\begin{thm}\label{thm:MainResult}
Let $M = \Gamma \backslash (M_0 \times \widetilde{M}_1 \times \cdots \times \widetilde{M}_q)$ be 
the homogeneity type of a compact symmetric space, where 
$\widetilde{M}_{\rm{cpt}} \equiv \widetilde{M}_1 \times \cdots \times \widetilde{M}_q$ is 
trivial or split-rank.
%and let $\mathscr{P}(M)$ denote the set of metrics in $\mathscr{R}_{\rm{symm}}(M)$
%for which the Poisson relation, Equation~\ref{eqn:PoissonRelation}, is an equality.
Also, for $j =1, \ldots, q$, let $\pi_j$ be the natural projection of 
$M_0 \times Z(\widetilde{M}_1) \times \cdots \times Z(\widetilde{M}_q)$ onto its $j$-th factor. 
%Also, let $\mathscr{W}_0(M) \subseteq \mathscr{P}(M)$ be the collection of metrics $g \in \mathscr{R}_{\rm{symm}}(M)$ 
%having the property that for each $\tau \in \Spec_L(M,g)$, $\Wave_{0}^{\bullet}(\tau) \neq 0$.
%in particular, the Poisson relation is an equality for each metric in $\mathscr{W}_0(M)$. 
Then the following hold.

\begin{enumerate}
\item $\mathscr{W}_0(M) \subseteq \mathscr{P}(M)$ contains a residual set. 
Consequently, for a generic symmetric metric on $M$, the Poisson relation is an equality and 
its length spectrum can be recovered from the Laplace spectrum of the metric. 
%Consequently, the length spectrum of a generic symmetric metric on $M$  
%can be recovered from the Laplace spectrum of the metric. 

\item Let $\mathscr{H}$  be the collection of homogeneity types $M$, as above, such that:
\begin{enumerate}
\item $M$ is irreducible (e.g., simple Lie group), or
\item $M$ is not irreducible and the subgroup 
$\Gamma \leq M_0 \times Z(\widetilde{M}_1) \times \cdots \times Z(\widetilde{M}_q)$
is trivial or, for $j = 1, \ldots, q$, satisfies the following:
\begin{enumerate}
\item if $\widetilde{M}_j = \SU(n+1)$, where $n \equiv 1 \mod 2$, then 
$\pi_j(\Gamma)$ is a proper subgroup of  $Z(\SU(n+1)) \simeq \Z_{n+1}$;
\item if $\widetilde{M}_j = \Spin(2n+1)$, then $\pi_j(\Gamma)$ is the trivial subgroup in $Z(\Spin(2n+1)) \simeq \Z_2$;
\item if $\widetilde{M}_j = \Sp(n)$, where $n \equiv 1,2 \mod 4$, 
then $\pi_j(\Gamma)$ is the trivial subgroup in $Z(\Sp(n)) \simeq \Z_2$;
\item if $\widetilde{M}_j = \Spin(2n)$, where $n \equiv 2 \mod 4$, then 
$\pi_j(\Gamma)$ is trivial or $\Z_2 \oplus 1$ in $Z(\Spin(2n)) \simeq \Z_2 \oplus \Z_2$;
\item if $\widetilde{M}_j = \Spin(2n)$, where $n \equiv 3 \mod 4$, then 
$\pi_j(\Gamma)$ is trivial or $\Z_2$ in $Z(\Spin(2n)) \simeq \Z_4$;
\item if $\widetilde{M}_j = E_7$, then $\pi_j(\Gamma)$ is trivial in $Z(E_7) \simeq \Z_2$;
\item if $\widetilde{M}_j$ is any other irreducible simply-connected split-rank space, 
then there are no restrictions on $\pi_j(\Gamma)$.
\end{enumerate}
%\item $\Gamma$ is trivial; that is $M$ is a product of a torus with a the homogeneity type of a simply-connected split-rank symmetric space,
%\item for $j=1, \ldots, q$, the compact irreducible space $\widetilde{M}_j$ is 
%either a split-rank space of Type I or a simple Lie group of type $A_{2n}$, $C_{4n}$, $C_{4n+3}$, $D_{4n}$, 
%$D_{4n+1}$, $F_2$, $G_2$, $E_6$ or $E_8$ ($n \geq 1$).
\end{enumerate}
%Therefore, $\mathscr{H}$ consists of $M = M_0 \times \widetilde{M}_{\rm{split}}$
If $M \in \mathscr{H}$, then $\mathscr{W}_0(M)$ equals $\mathscr{R}_{\rm{symm}}(M)$
and we conclude that the Poisson relation is an equality for any symmetric metric on $M$.
Consequently, if $M$ is in $\mathscr{H}$, %satisfies $(a)$ or $(b)$, 
then for any metric $g \in \mathscr{R}_{\rm{symm}}(M)$ 
Poisson relation is an equality and the length spectrum of $g$ 
can be recovered from its Laplace spectrum.
%then the length spectrum of any metric $g \in \mathscr{R}_{\rm{symm}}(M)$ 
%can be recovered from its Laplace spectrum. 
%Poisson relation is an equality for every symmetric metric on $M$.
\item If $M$ is split-rank; i.e., $M_0$ is trivial, then
$\mathscr{W}_0(M)$ contains the Einstein metric induced by 
$c(g^1_0 \times \cdots \times g^q_0)$, for each $c >0$.

\end{enumerate}
\end{thm}

A compact Lie group $U$ equipped with a bi-invariant metric 
is a symmetric space having homogeneity type:
\begin{eqnarray}\label{eqn:HomogeneityTypeLieGroup}
U = \Gamma \backslash (T \times \widetilde{U}_1 \times \cdots \times \widetilde{U}_q),
\end{eqnarray}
where $T$ is a torus, the factors $\widetilde{U}_1, \ldots , \widetilde{U}_q$ 
are simply-connected simple Lie groups and $\Gamma$ is a discrete subgroup 
of $Z(U) = T \times Z(\widetilde{U}_1) \times \cdots \times Z(\widetilde{U}_q)$, the center of $U$.
Thus, we immediately obtain the following specialization of the previous result.
%Since, the semi-simple part of $U$, which is given by $\widetilde{U}_{ss} = \widetilde{U}_1 \times \cdots \times \widetilde{U}_q$ can be expressed as $\widetilde{U}_{ss} = \widetilde{U}_{ss} \times \widetilde{U}_{ss}is split-rank, we obtain the following specialization of the previous result.

\begin{cor}\label{cor:MainResult}
Let $U = \Gamma \backslash (T \times \widetilde{U}_1 \times \cdots \times \widetilde{U}_q)$ 
be a compact Lie group and let $\mathscr{R}_{\rm{bi}}(U) = \mathscr{R}_{\rm{symm}}(U)$ 
denote the associated space of bi-invariant metrics on $U$. 
Additionally, let $\mathscr{G}$ be the collection of compact Lie groups in $\mathscr{H}$;
that is, $\mathscr{G}$ consists of all compact Lie groups  
for which $\Gamma$ is trivial or satisfies certain constraints
depending on $\widetilde{U}_1 \times \cdots \times \widetilde{U}_q$,
%$T \times \widetilde{U}_1 \times \cdots \times \widetilde{U}_q$, 
as described in Theorem~\ref{thm:MainResult}(2). 
%that depend on $T \times \widetilde{U}_1 \times \cdots \times \widetilde{U}_q$.
Then the following hold.

\begin{enumerate}
\item $\mathscr{W}_0(U) \subseteq \mathscr{P}(U)$ contains a residual set. 
Consequently, for a generic bi-invariant metric on $U$ the Poisson relation 
is an equality and its length spectrum can be recovered from its Laplace spectrum.
%Consequently, the length spectrum of a generic bi-invariant metric 
%on $U$ can be recovered from its Laplace spectrum.

%\item $\mathscr{P}(U)$ contains any bi-invariant metric on $U$ induced by 
%a metric of the form $h \times g_{\rm{K}}$ on $T \times \widetilde{U}_{ss}$, 
%where $h$ is a flat metric on $T$ and (up to scaling) $g_{\rm{K}}$ is the Einstein metric on 
%$\widetilde{U}_{ss} = \widetilde{U}_1 \times \cdots \times \widetilde{U}_q$ 
%induced by the Killing form.

\item %Let $\mathscr{G}$ consist of the compact Lie groups in $\mathscr{H}$.
If $U$ is in $\mathscr{G}$, then $\mathscr{W}_0(U)$ equals $\mathscr{R}_{\rm{bi}}(U)$
and we conclude that the Poisson relation is an equality for every bi-invariant metric 
supported by $U$.
Consequently, if $U$ is a member of $\mathscr{G}$, %satisfies $(a)$ or $(b)$, then the 
then the length spectrum of any bi-invariant metric $g \in \mathscr{R}_{\rm{bi}}(U)$ 
can be recovered from its Laplace spectrum.

\item If $U$ is semi-simple, then $\mathscr{W}_0(U)\subseteq \mathscr{P}(U)$ contains the 
(unique up to scaling) bi-invariant Einstein metric 
induced by the negative of the Killing form on the Lie algebra of $U$.

\end{enumerate}
\end{cor}

\noindent
Therefore, we obtain a positive answer to Problem~\ref{prob:PoissonRelation} and, consequently, 
Problem~\ref{prob:HearingLengthSpectrum} for a substantial portion of bi-invariant metrics.
In particular, we see that Problem~\ref{prob:PoissonRelation} has an affirmative answer for every bi-invariant metric on a compact Lie group that is simple, simply-connected, a torus or a product thereof.
%compact simple Lie group or a compact Lie group that is the product of a torus and a simply-connected semi-simple Lie group.

%In light of our earlier discussion, it follows that Theorem~\ref{thm:MainResult} 
%establishes the folk-conjecture for over ``half'' of the compact irreducible symmetric spaces.

%{\bf It seems that there should be some natural probability measure on the space of all compact Lie groups.
%And, with respect to this measure, the set of compact Lie groups for which the 
%Poisson relation is an equality for every bi-invariant metric it supports should constitute over half of this space.
%Can I make this precise?}

We conclude this section with an observation---proven in 
Section~\ref{SubSec:RestrictedRoots}---suggesting that the rank of a 
compact Lie group is encoded in the spectrum of any of its associated bi-invariant metrics.
%concerning the relationship between the spectrum 
%of a bi-invariant metric and the rank of its underlying compact Lie group.

\begin{thm}\label{thm:HearingRank}
The rank of a compact Lie group $U$ is encoded in the spectrum of any bi-invariant metric $g \in \mathscr{W}_0(U)$. Consequently, within $\mathscr{W}_0^{\rm{Lie}} \equiv \cup_U \mathscr{W}_0(U)$, where the union is over all compact Lie groups, rank is a spectral invariant.
\end{thm}

%\noindent
%In Remark~\ref{rem:LeadingTerm} we note that $\mathscr{W}_0(U)$ 
%contains a residual set (and in many cases is all of $\mathscr{R}_{\rm{bi}}(U)$). 
%Therefore, Theorem~\ref{thm:HearingRank}, whose proof is supplied in Section~
%\ref{SubSec:RestrictedRoots}, suggests that the rank of a compact Lie group is encoded in 
%the spectrum of any of its associated bi-invariant metrics.

%%%%%%%%%%%%%%%%%%%%%%%%%%%
\subsection{Proof of Theorem~\ref{thm:MainResult}}\label{SubSec:MainProof}

We present the proof of Theorem~\ref{thm:MainResult} 
modulo some technical details that will be addressed in 
Section ~\ref{sec:HearingLengthSpectrum} after we have reviewed further aspects of symmetric spaces
in Section~\ref{Sec:MorseIndex}.
%The reader wishing to be distracted by the outline may 
%skip to some concluding remarks in Section~\ref{SubSec:ConcludingRemarks}.

\begin{proof}[Proof of Theorem~\ref{thm:MainResult}]\label{OutlineOfProof} 
%As we noted in Remark~\ref{rem:Smoothness} it is enough to show $\Spec_L(M,g) \subset \SingSupp (\Trace(U_g(t)))$.
Theorem~\ref{thm:SymmSpcsCIH} states that compact symmetric spaces are clean; therefore, our strategy is to show that for each $\tau \in \LSpec^{\pm}(M,g)$, the leading term, $\Wave_{0}^{\rm{lead}}(\tau)$, in Equation~\ref{eqn:Asymptotic1} is non-zero. By Remark~\ref{rem:TimeReversal}, it is enough to show this for $\tau \in \LSpec(M,g)$; i.e., the lengths of the closed geodesics.

Let $(M,g)$ be a compact symmetric space without any further restrictions, for now,  
and let $\tau \in \LSpec(M,g)$.
Since $\tau$ is clean, by Theorem~\ref{thm:SymmSpcsCIH}, we know that $\Fix(\Phi_\tau)$ 
is a disjoint union of finitely many closed submanifolds
$\Theta_1, \ldots , \Theta_r$ in the unit tangent bundle of $M$. 
If we let $D$ denote the dimension of $\Fix(\Phi_\tau)$ 
and let $\bullet$ denote ``even'' or ``odd'' according to the parity of $D$, 
then by Equation~\ref{eqn:WaveInvariants1} %and \ref{eqn:WaveInvariants2}
we have 
$$\Wave_{0}^{\rm{lead}}(\tau) = \Wave_{0}^{\bullet}(\tau) = \left( \frac{1}{2\pi i}\right)^{\frac{D -1}{2}} \sum_{\stackrel{j=1}{\dim \Theta_j = D}}^{r} i^{-\sigma_j} \int_{\Theta_j} d \mu_j^\tau.$$
Since the measures $\mu_j^\tau$ are positive, the vanishing of $\Wave_{0}^{\bullet}(\tau)$ 
depends in part on the value of the Morse indices, 
which can be difficult to compute for a general Riemannian manifold.
However, as we note in Section~\ref{SubSec:MorseIndex}, since $(M,g)$ is
a compact symmetric space the Morse index of a closed geodesic in $(M,g)$ 
can be computed in terms of the restricted roots of the symmetric space \cite{Ziller3}. 
From this we deduce (see Equation~\ref{eqn:MorseIndex}) 
that for a \emph{non-trivial} closed geodesic $\gamma: [0,1] \to (M,g)$,
the Morse index has the following form:
\begin{equation}\label{eqn:IntroMorseIndex1}
\sigma_{\Delta M}(\gamma) = F_{(M,g)}(\gamma'(0)) - \dim \Fix_{\gamma}(\Phi_\tau) + \dim M,
%\sigma_{\Delta M}(\gamma) = F_{(M,g)}(\gamma'(0)) - \dim \Fix_{\gamma}(\Phi_\tau) + C(M),
\end{equation}
where $\Fix_{\gamma}(\Phi_\tau)$ is the component of $\Fix(\Phi_\tau) \subseteq SM$
containing $\gamma'(t) / \|\gamma'(t)\|$. %and $C(M)$ is some constant depending only on $M$.

It will be useful to observe that the Morse index of a closed geodesic 
in certain symmetric spaces is influenced by the lengths of 
the components of its velocity vector (in the non-Euclidean factor).

\begin{prop}\label{prop:Component Dependence}
Let $M = \Gamma \backslash (M_0 \times \widetilde{M}_1 \times \cdots \times \widetilde{M}_q)$
be the homogeneity type of a compact symmetric space, 
where $\widetilde{M}_{\rm{cpt}} = \widetilde{M}_1 \times \cdots \times \widetilde{M}_q$ is split-rank,
and for $j = 1, \ldots , q$, let $M_j \equiv \widetilde{M}_j / \Gamma_j$, 
where $\Gamma_j$ is the projection of  $\Gamma$ onto its $j$-th factor.
Now, consider the symmetric metric $g$ on $M$ induced by the metric 
$h \times \tilde{g_1} \times \cdots \times \tilde{g}_q$
and, for $j = 1, \ldots , q$, let $g_j$ be the metric on $M_j$ induced by $\tilde{g}_j$.
%Let $(M,g)$ be a compact symmetric space where 
%$M = split-rank symmetric space with universal cover 
%$(\widetilde{M}_1 \times \cdots \times \widetilde{M}_k, \tilde{g}_1 \times \cdots \times \tilde{g}_k)$.
Then there is a function $f_{(M,g)}(x_0, x_1, \cdots, x_q) = \sum_{j=1}^{q} f_{(M_j, g_j)}(x_j)$, where $f_{(M_j, g_j)}(0) = 0$ for each $j = 1, \ldots , q$,
such that if $\gamma : [0,1] \to (M,g)$ is a non-trivial closed geodesic with 
$\gamma'(0) \equiv (v_0, v_1, \cdots, v_q)$, then 
\begin{equation}\label{eqn:IntroBigF}
F_{(M,g)}(\gamma'(0)) \equiv f_{(M,g)}(\| v_0\|, \|v_1\|, \ldots, \| v_q\|) \mod 4.
\end{equation}
Consequently, the function 
$$h_{(M,g)}(x_0, x_1, \ldots, x_q, y) = \left\{ \begin{array}{ll}
f_{(M,g)}(x_0, x_1, \ldots, x_q) - y + \dim M & \mbox{, when $x_j \neq 0$ for some $j$} \\
0 & \mbox{, otherwise}\\
\end{array}
\right.
$$
%$h_{(M,g)}(x_0, x_1, \ldots, x_k, y) = f_{(M,g)}(x_0, x_1, \ldots, x_k) - y + C(M)$
satisfies
\begin{equation}\label{eqn:IntroMorseIndex2}
\sigma_{\Delta M}(\gamma) \equiv h_{(M,g)}(\| v_0\|, \|v_1\|, \ldots, \| v_q\|, \dim \Fix_{\gamma}(\Phi_\tau)) \mod 4.
\end{equation} 
\end{prop}

\begin{proof}
See Section~\ref{SubSec:PropComponentDependence}. %\ref{sec:HearingLengthSpectrum}.
\end{proof}

With an eye towards establishing that the Poisson relation holds 
for a generic split-rank symmetric space we make the following definition. 
%Theorem~\ref{}.

\begin{dfn}\label{dfn:CLU}
Let $(N,g)$ be a closed Riemannian manifold with universal cover 
$(\widetilde{N}, \tilde{g}) = 
(\widetilde{N}_1 \times \cdots \times \widetilde{N}_k, \tilde{g}_1 \times \cdots \times \tilde{g}_k)$.
We will say the length spectrum of $(N,g)$ is \emph{component length unique} (CLU)
if for any $\tau \in \Spec_L(N,g)$, there are nonnegative constants 
$c_1(\tau), \cdots , c_k(\tau)$
such that for any closed geodesic $\gamma: [0,1] \to (N,g)$ of length $\tau$
the velocity vector $\gamma'(0) = (v_1, \cdots , v_k)$ satisfies 
$\|v_j\|^2 = c_j (\tau)$ for $j = 1, \ldots , k$.
That is, the length of the closed geodesic $\gamma$ determines the lengths of the components
of $\gamma'(t)$.
\end{dfn}

\begin{lem}\label{lem:CLUResidual}
Let $M = \Gamma \backslash (M_0 \times \widetilde{M}_1 \times \cdots \times \widetilde{M}_q)$ 
have the homogeneity type of a compact symmetric space.
The set of symmetric metrics on $M$ with CLU length spectrum 
is a residual set in $\mathscr{R}_{\rm{symm}}(M)$.
\end{lem}

\begin{proof}
See Section~\ref{SubSec:CLUResidual}.
\end{proof}

The following proposition tells us that for many compact symmetric spaces 
having a non-Euclidean part that is split-rank,  
the Morse index of a closed geodesic $\gamma$ is determined modulo $4$ 
by its length $\tau$ and the dimension of its corresponding component in $\Fix(\Phi_\tau)$
(cf. \cite[Theorem 3 and Corollary 3.6]{Wilking}).
 
\begin{thm}\label{thm:IntroMod4}
Let $(M,g)$ be a compact symmetric space, where $\widetilde{M}_{\rm{cpt}}$ is split-rank, 
that satisfies any one of the following:

\begin{enumerate}
\item the length spectrum of $(M,g)$ is component length unique,

\item $M \in \mathscr{H}$, where $\mathscr{H}$ is defined as in Theorem~\ref{thm:MainResult},

%\item $M$ is irreducible,

%\item $M$ is not irreducible and, letting $\pi_j$ be the projection of $M_0 \times Z(\widetilde{M}_1) \times \cdots \times Z(\widetilde{M}_q)$ onto its $j$-th factor,  the subgroup 
%$\Gamma \leq M_0 \times Z(\widetilde{M}_1) \times \cdots \times Z(\widetilde{M}_q)$ 
%satisfies the following
%\begin{enumerate}
%\item if $\widetilde{M}_j = \SU(n+1)$, where $n \equiv 1 \mod 2$, then 
%$\pi_j(\Gamma)$ is a proper subgroup of  $Z(\SU(n+1)) \simeq \Z_{n+1}$;
%\item if $\widetilde{M}_j = \Spin(2n+1)$, then $\pi_j(\Gamma)$ is the trivial subgroup in $Z(\Spin(2n+1)) \simeq \Z_2$;
%\item if $\widetilde{M}_j = \Sp(n)$, where $n \equiv 1,2 \mod 4$, 
%then $\pi_j(\Gamma)$ is the trivial subgroup in $Z(\Sp(n)) \simeq \Z_2$;
%\item if $\widetilde{M}_j = \Spin(2n)$, where $n \equiv 2 \mod 4$, then 
%$\pi_j(\Gamma)$ is trivial or $\Z_2 \oplus 1$ in $Z(\Spin(2n)) \simeq \Z_2 \oplus \Z_2$;
%\item if $\widetilde{M}_j = \Spin(2n)$, where $n \equiv 3 \mod 4$, then 
%$\pi_j(\Gamma)$ is trivial or $\Z_2$ in $Z(\Spin(2n)) \simeq \Z_4$;
%\item if $\widetilde{M}_j = E_7$, then $\pi_j(\Gamma)$ is trivial in $Z(E_7) \simeq \Z_2$;
%\item if $\widetilde{M}_j$ is any other irreducible simply-connected split-rank space, 
%then there are no restrictions on $\pi_j(\Gamma)$.
%\end{enumerate}

\item the universal cover of $M$ has no Euclidean part and $g$ is the unique up to scaling 
$G$-invariant Einstein metric on $M$.
\end{enumerate}
\noindent
Then, there is a function $H_{(M,g)} : \Spec_L(M,g) \times \{0, 1, \ldots , 2\cdot \dim M -1\} \to \{0,1,2,3\}$
such that the Morse index of any closed geodesic $\gamma$ in $(M,g)$ of length $\tau$ satisfies
\begin{eqnarray}\label{Eq:MorseIndex1}
\sigma_{\Delta M}(\gamma) \equiv H_{(M,g)}(\tau, \dim \Fix_{\gamma}(\Phi_\tau)) \mod 4.
\end{eqnarray}
In the event that $(M,g)$ satisfies either $(3)$ or $(4)$ above, this 
relationship only depends on $\dim(\Fix_\gamma(\Phi_\tau))$:
\begin{eqnarray}\label{Eq:MorseIndex2}
\sigma_{\Delta M}(\gamma) \equiv H_{(M,g)}(\dim \Fix_{\gamma}(\Phi_\tau)) \mod 4,
\end{eqnarray}
where $H_{(M,g)}(x) = x + C(M)$.
\end{thm}

\begin{proof}
See Section~\ref{SubSec:ThmIntroMod4}. %\ref{sec:HearingLengthSpectrum}.
\end{proof}

Now, let $(M,g)$ be a compact symmetric space where $\widetilde{M}_{\rm{cpt}}$, the non-Euclidean part of the universal cover of $M$, is split-rank. Suppose further that $(M,g)$ satisfies any one of conditions $(1)$-$(3)$ in Theorem~\ref{thm:IntroMod4}, then it is clear from the conclusion of the theorem that for any $\tau \in \Spec_L(M,g)$ we have $\Wave_{0}^{\rm{lead}}(\tau) =\Wave_{0}^{\bullet}(\tau) \neq 0$, where $\bullet$ equals ``even'' or ``odd'' according to the parity of $\dim (\Fix(\Phi_\tau))$. That is, there are no cancellations in the leading term of the wave trace associated to $\tau$. And, we conclude that $\tau$ is in the singular support of the wave trace and $\dim (\Fix(\Phi_\tau))$ is a spectral invariant. In particular, for any symmetric space as in the statement of Theorem~\ref{thm:IntroMod4}, the Poisson relation is an equality. As we also established in Lemma~\ref{lem:CLUResidual} that the symmetric metrics on $M$ with CLU length spectrum are generic in $\mathscr{R}_{\rm{symm}}(M)$, the theorem follows. 
\end{proof}

Since the dimension of a Riemannian manifold is a spectral invariant and 
we know that for any non-zero $\tau$ in the length spectrum of a compact 
symmetric space $M$ of rank at least $2$ we have $\dim( \Fix(\Phi_\tau)) < 2\cdot \dim M -1$, 
the following observation is an immediate consequence of Theorem~\ref{thm:MainResult}.

\begin{cor}\label{cor:HigherRank}
Let $M = \Gamma \backslash (M_0 \times \widetilde{M}_1 \cdots \times \widetilde{M}_q)$ be the homogeneity type of a compact symmetric space where $\widetilde{M}_{\rm{cpt}}$ is trivial or split-rank.
Then a generic metric $g \in \mathscr{R}_{\rm{symm}}(M)$ cannot 
be isospectral to a compact rank-one symmetric space or, more generally, 
a $C_\tau$-manifold.
In the case that $M$ satisfies condition $2(a)$ or $2(b)$ of Theorem~ \ref{thm:MainResult}, 
then the preceding conclusion holds for every metric $g \in \mathscr{R}_{\rm{symm}}(M)$.
\end{cor}

\noindent
This suggests that, in general, the wave invariants might distinguish 
compact higher-rank symmetric spaces from compact rank-one symmetric spaces (cf. Theorem~\ref{thm:HearingRank}).
%This suggests that it should be possible to use the wave-trace to show that the spectrum can 
%distinguish compact higher-rank symmetric spaces from compact rank-one symmetric spaces.

We round out this section by making the following observation, 
which follows directly from Lemma~\ref{lem:CLUResidual}.

\begin{thm}\label{thm:MainResultGeneral}
Let $M = \Gamma \backslash (M_0 \times \widetilde{M}_1 \times \cdots \times \widetilde{M}_q)$ 
be the homogeneity type of a compact symmetric space having the property 
that for every $g \in \mathscr{R}_{\rm{symm}}(M)$,
the function $F_{(M,g)}$ in Equation~\ref{eqn:IntroMorseIndex1} 
associated with $(M,g)$ satisfies Equation~\ref{eqn:IntroBigF}  
for any closed geodesic $\gamma : [0,1] \to (M, g)$.
Then $\mathscr{P}(M)$ contains a residual set.
Consequently, the length spectrum of a generic symmetric metric on $M$  
can be recovered from the Laplace spectrum of the metric. 
\end{thm}

\noindent
Theorem~\ref{thm:MainResult}\,(1) is an instance of this result and we suspect that 
Theorem~\ref{thm:MainResultGeneral} is applicable to a broader class of symmetric spaces.  
However, in Example~\ref{exa:Mod4Counterexample2} we see that it cannot be applied to all symmetric spaces.
%We suspect this is applicable to more symmetric spaces than those considered in Theorem~\ref{thm:MainResult};
%however, in Example~\ref{exa:Mod4Counterexample2} we see that it cannot be applied to all symmetric spaces.

%%%%%%%%%%%%%%%%%%%%%
\subsection{Concluding Remarks}\label{SubSec:ConcludingRemarks}
 To the best of our knowledge, the manifolds considered in
Theorem~\ref{thm:MainResult} along with $C_\tau$-manifolds (e.g., CROSSes), 
\emph{flat} manifolds \cite{Miatello-Rossetti}, Heisenberg manifolds \cite{Pesce}, 
compact locally symmetric spaces of the non-compact type 
(i.e., \emph{non-positively} curved locally symmetric spaces without a Euclidean factor) 
\cite{Huber1, Huber2,  McKean, Gangoli, DuKoVa} (cf. \cite[Sec. 10]{Prasad-Rapinchuk}), \emph{negatively} curved manifolds \cite[Theorem 4]{Guillemin} and, recently, 
certain lens spaces \cite{Cianci} are the only non-generic spaces 
for which it is known that one may recover 
the length spectrum directly from the Laplace spectrum.
%\footnote{The reader will notice that most of the previously known examples have non-positive curvature.}
%For most of these results one can apply the trace formula of Duistermaat and Guillemin, 
%but the threat of cancellation does not exist in non-positive curvature. 
%So, it is in nonnegative curvature---i.e., the setting we are addressing---where 
%applying this technique becomes more delicate.
And, to date, there are no examples to the contrary. 
%For instance, Sunada's method produces manifolds possessing 
%identical Laplace spectra and length spectra \cite[Section 4]{Sunada}.
Indeed, as was noted earlier, all isospectral pairs arising from Sunada's method must 
have identical length spectra \cite[Section 4]{Sunada} (cf. \cite[Theorem 1.3]{GorMao}).
 
%The results of Huber, McKean, Gangoli and Pesce referenced in the previous sentence make 
%use of the Selberg trace formula, Poisson summation formula and generalizations thereof. 
%We note that in the case of flat tori, Riemann surfaces and Heisenberg manifolds this follows %from a Poisson summation formula, which demonstrates that the spectrum actually determines %the lengths of the closed geodesics along with their multiplicities, 
%where the multiplicity of a length $\tau$ is given by the number of 
%free homotopy classes  containing a geodesic of length $\tau$. 
%However, there are examples of isospectral manifolds with identical length spectra and 
%differing multiplicities \cite{}. 
%%In particular, we provide several infinite families of manifolds of non-negative curvature for which the conjecture is true.

In light of the above it is tempting to prove the conjecture for an
arbitrary symmetric space by extending Theorem~\ref{thm:IntroMod4}.
%i.e., understanding the behavior of the Morse index. 
However, Example~\ref{exa:Mod4Counterexample1} shows that this 
is not possible for compact Lie groups, in general, 
and Example~\ref{exa:Mod4Counterexample2} demonstrates that restricting
our attention to irreducible symmetric spaces will not be successful.
%In particular, through Example~\ref{exa:Mod4Counterexample2} we see that extending Theorem~\ref{thm:IntroMod4}
%to cover irreducible symmetric spaces is not feasible.
In both cases we find two closed geodesics $\gamma_1$ and $\gamma_2$ of the same length $\tau$ and for which 
$\dim \Fix_{\gamma_1}(\Phi_\tau)$ and $\dim \Fix_{\gamma_2}(\Phi_\tau)$ are equal; 
however, the difference between $\sigma_{\Delta M}(\gamma_1)$  and 
$\sigma_{\Delta M}(\gamma_2)$ is congruent to two modulo $4$, which raises 
the possibility of cancellation in the trace formula.
Therefore, in contrast with the case of locally symmetric spaces of the non-compact type,  
determining whether the Poisson relation is an equality for an arbitrary 
compact Lie group or, more generally, symmetric space will ostensibly require a better understanding of the Duistermaat-Guillemin measures associated to the components of $\Fix(\Phi_\tau)$ (see p.~\pageref{DGMeasure}).
%\cite[Theorem 4.5 and p. 69-70]{DuGu}.
%of how the Duistermaat-Guillemin measure behaves on the various components of $\Fix(\Phi_\tau))$. 
We will take this up in a separate article.

More generally, one might be curious about the extent to which this type of analysis can be carried out for compact Lie groups equipped with arbitrary left-invariant metrics and other homogeneous spaces. In Section~\ref{Sec:UncleanMetrics} we observe that when considering left-invariant metrics on $\SO(3)$ (respectively, $S^3$) we encounter two hurdles. First, not all left-invariant metrics are clean. As we noted in Theorem~\ref{thm:UncleanMetrics}, we exhibit left-invariant naturally reductive metrics on $\SO(3)$ (respectively, $S^3$) that are not clean. Therefore, being a ``nice'' left-invariant metric does not guarantee access to the trace formula of Duistermaat and Guillemin. Second, even among the clean left-invariant naturally reductive metrics on $\SO(3)$ (respectively, $S^3$), it appears quite difficult to determine whether the leading term of the trace formula is non-zero for each length $\tau$ (cf. Proposition~\ref{prop:ConjugatePointsTypeIII}). These obstacles highlight the need for a different approach to the trace formula in the homogeneous setting.

%However, an examination of the proof of Theorem~\ref{thm:IntroMod4} will show that this is not 
%possible. Indeed, one can deduce that any compact \emph{reducible} split-rank symmetric space $M$ for which
%the universal cover $\widetilde{M}$ contains at least one irreducible factor not 
%listed in Theorem~\ref{thm:IntroMod4}(4) will admit symmetric metrics for which 
%the Morse index modulo $4$ is not controlled as in Theorem~\ref{thm:IntroMod4} (see Remark~\ref{rem:IntroMod4}).
%In fact, Example~\ref{exa:Mod4Counterexample} points out that we already run into a problem 
%in extending Theorem~\ref{thm:IntroMod4} when we try to consider arbitrary irreducible symmetric spaces.
%Indeed, we find an infinite family of irreducible symmetric spaces each of which contains 
%two geodesics $\gamma_1$ and $\gamma_2$ of the same length $\tau$ and for which 
%$\dim \Fix_{\gamma_1}(\Phi_\tau)$ and $\dim \Fix_{\gamma_2}(\Phi_\tau)$ are equal; 
%however, the difference between $\sigma_{\Delta M}(\gamma_1)$  and 
%$\sigma_{\Delta M}(\gamma_2)$ is congruent to two modulo $4$, which raises 
%the possibility of cancellation in the trace formula.
%Therefore, in contrast with the case of locally symmetric spaces of the non-compact type,  
%determining whether the Poisson relation is an equality for an arbitrary 
%compact symmetric space will ostensibly require a better understanding 
%of how the Duistermaat-Guillemin measure behaves on the various components of $\Fix(\Phi_\tau))$. 
%We will take this up in a separate article. 

\medskip
\noindent
{\bf Structure of the Paper.}
%The outline of this article is as follows. 
%In Section~\ref{Sec:RootSystems} we compute the central lattice, co-root lattice 
%and lowest strongly dominant form associated to each of the indecompasble root systems
In Section~\ref{Sec:MorseIndex}, we prove Theorem~\ref{thm:SymmSpcsCIH}, 
which establishes that compact symmetric spaces are clean, and we review Ziller's 
method (see Theorem~\ref{thm:MorseIndexRoots}) for computing the Morse index 
of a closed geodesic in a symmetric space via the restricted roots.
We also prove Theorem~\ref{thm:HearingRank}, which establishes that 
among generic bi-invariant metrics, rank is a spectral invariant.
In Section~\ref{sec:HearingLengthSpectrum}, we complete the argument for Theorem~\ref{thm:MainResult}
by proving Proposition~\ref{prop:Component Dependence}, Lemma~\ref{lem:CLUResidual} and Theorem~\ref{thm:IntroMod4} which shows that for certain compact symmetric spaces, the Morse index modulo $4$ of a closed geodesic $\gamma$ of length $\tau$ 
only depends on $\tau$ and the dimension of the component of 
$\Fix(\Phi_\tau)$ containing $\gamma'(0)$. 
The arguments in Section~\ref{sec:HearingLengthSpectrum} rely on an explicit 
description of the co-root lattice, central lattice and lowest strongly dominant form 
associated to each of the indecomposable abstract root systems. 
This data is computed and catalogued in Appendix~\ref{Sec:RootSystems}. 
In Section~\ref{Sec:UncleanMetrics}, we present the proof of Theorem~\ref{thm:UncleanMetrics}, which demonstrates that not all homogeneous metrics are clean: the culprit will be the \emph{Type II geodesics} (Definition~\ref{dfn:GeodesicTypes}). This argument is lengthy as we must explicitly compute the closed geodesics of a left-invariant naturally reductive metric on $\SO(3)$ and compute the Poincare map along these geodesics. We conclude the section by examining the Poisson relation for the clean left-invariant naturally reductive metrics on $\SO(3)$ and observe that resolving the cancellation issue in the leading term of the trace formula appears problematic for lengths arising from \emph{Type III geodesics} (Propositions~\ref{prop:ConjugatePointsTypeIII} and \ref{prop:PoissonRelationNatReductive}). Certain technical concerns related to naturally reductive metrics are relegated to Appendix~\ref{sec:AdK}.
%This data is computed and compiled in Appendix~\ref{Sec:RootSystems}.

%Afterwards, we observe that the Morse index of a closed (unit speed) geodesic $\gamma$ of length $\tau$  
%depends only on $\tau$ and the 
%dimension of the component of $\Fix(\Phi_\tau)$ containing $\gamma'(0)$ (see Equation~\ref{eqn:MorseIndex}). .
%Finally, in Section~\ref{sec:HearingLengthSpectrum}, we present the proof of Theorem~\ref{thm:PoissonRelationNonSimple}.

\medskip
\noindent
{\bf Acknowledgments.} I wish to thank Alejandro Uribe for useful 
conversations concerning the trace formula. 
%during the conference ``K\"{a}hler Geometry and Quantization'' (Cologne, July 2012).
I am also appreciative of the hospitality extended by the University of Pennsylvania, 
Humboldt Universit\"{a}t zu Berlin, Universit\'{e} de Fribourg, Indiana University and the University of Michigan during the 2013-14 academic year while I was writing a previous version of this article.

%%%%%%%%%%%%%%%%%%%%%%%%%%%%%%%%%%%%%%%%%%
%%%%%%%%%%%%%%%%%%%%%%%%%%%%%%%%%%%%%%%%%%
%%%%%%%%%%%%%%%%%%%%%%%%%%%%%%%%%%%%%%%%%%
%%%%%%%%%%%%%%%%%%%%%%%%%%%%%%%%%%%%%%%%%%
\section{Computing the Morse Index in a Symmetric Space}\label{Sec:MorseIndex}

The purpose of this section is to fix some notation, terminology and useful facts concerning
symmetric spaces. In particular, we will demonstrate in Proposition~\ref{prop:FixedSet} 
that for any $\tau$ in the length spectrum of 
a symmetric space, the fixed-point set $\Fix(\Phi_\tau)$ is a disjoint union of finitely many 
homogeneous manifolds.
We will then use this to establish Theorem~\ref{thm:SymmSpcsCIH}.
% which goes halfway towards establishing Theorem~\ref{thm:SymmSpcsCIH}.
In Theorem~\ref{thm:MorseIndexRoots} we will recall a result of Ziller which states 
the Morse index of a closed geodesic $\gamma$ in a \emph{compact} symmetric space may be computed in terms of 
the restricted roots of the symmetric space and, upon closer inspection, we will notice in Equation~\ref{eqn:MorseIndex}
that this expression for the Morse index also involves $\dim(\Fix(\Phi_\tau))$, 
where $\tau$ is the length of the closed geodesic. 
Finally, we will recall the definition of the center of a symmetric space and 
remind the reader of the co-root, integral and central lattices of a symmetric space.
%We will then discuss Ziller's result in the context of compact Lie groups 
%and conclude by defining the co-root, integral and central lattices of a compact Lie group.

%%%%%%%%%%%%%%%%%%%%%%%%%%%%%
\subsection{Geodesics and Cleanliness in a Symmetric Space}\label{SubSec:GeodSymmSpaces}

Let $M$ be a symmetric space with $G$ equal to the connected component of the identity 
in $\Isom(M)$ and $K = G_{p_0}$ the connected component of the isotropy group of 
a fixed point $p_0 \in M$, so that $M = G/K$. 
The Lie algebra $\germ{g}$ of $G$ can be written as $\germ{g} = \germ{K} \oplus \germ{p}$, 
where $\germ{K}$ is the Lie algebra of $K$ and $\germ{p}$ is an $\Ad(K)$-invariant 
complement. We then have the relations
$$[\germ{K}, \germ{p}] \subseteq \germ{p}, [\germ{p}, \germ{p}] \subseteq \germ{K}, \mbox{ and } [\germ{K}, \germ{K}] \subseteq \germ{K}$$
so that the linear transformation $s: \germ{g} \to \germ{g}$ which is the identity on $\germ{K}$ and 
negative the identity on $\germ{p}$ is a Lie algegra homomorphism known as the Cartan involution.
The elements of $\germ{g}$ generate Killing vector fileds on $M$ by defining
$$X^*_q \equiv \frac{d}{dt} |_{t = 0} \exp_G(tX) \cdot q,$$
for any $q \in M$ and $X \in \germ{g}$.
It follows that $\germ{p}$ can be naturally identified with $T_pM$ via
$$X \in \germ{p} \mapsto X^*_{p_0} \in T_{p_0}M.$$
Under this identification we see that the Levi-Civita connection and curvature tensor are given by: 
$$(\nabla_{X^*}Y)(p_0) = [X^*, Y]
\mbox{ and } 
R_{p_0}(X^*, Y^*)Z^* = - [[X,Y],Z]^*_{p_0}.$$
Of particular interest to us is the fact that the geodesics in $M$ are the integral curves of the Killing vector fields and it follows that all self-intersections of a geodesic in $M$ are smooth (cf. \cite[Theorem~3.11]{Sut2}).
%As we noted in \cite[Theorem~3.11]{Sut2}, this is also true for any homogeneous Riemannian manifold.

Now, let $\germ{a}$ be a maximal abelian subspace contained in $\germ{p}$.
All such subspaces are conjugate via $\Ad(K)$ and their common dimension 
is referred to as the \emph{rank} of $M$. Furthermore, we know that 
$\germ{p} = \cup_{k \in K} \Ad(k) \germ{a}$ and 
$\germ{a} \leq \germ{p} \equiv T_{p_0}M$ exponentiates to a totally geodesic maximal flat 
$F_{\germ{a}} \equiv \Exp_{M}(\germ{a}) \simeq T^m \times R^d$ in $M$, 
where $T^m$ is a flat $m$-torus and $\R^d$ is $d$-dimensional euclidean space.
Therefore, since every vector $v \in \germ{p} \equiv T_{p_0} M$ 
can be conjugated into $\germ{a}$, we see that every geodesic in $M$ can be conjugated via the isometry group 
into the maximal flat $F_\germ{a}$ in $M$. We then define the \emph{integral lattice} of $M$ to be 
\begin{eqnarray}\label{eqn:IntegralLattice}
\Lambda_{I}(M) = \{ v \in \germ{a} : \exp_{G}(v) \cdot p_0 = p_0 \} = \{ v \in \germ{a} : \exp_{G}(v) = e \}
\end{eqnarray}
Now, if $\tau$ is an element in the length spectrum of $M$,
then there are finitely many closed geodesics in $F_{\germ{a}}$ of length $\tau$. It then follows that 
$\Fix(\Phi_\tau)$ is the disjoint union of finitely many homogeneous submanifolds 
$\Theta_1 = G\cdot v_1, \ldots , \Theta_r = G \cdot v_r$, 
where $v_1, \ldots v_r \in \germ{a}$ is a maximal collection of 
non-conjugate unit vectors such that $\gamma_j(t) \equiv \exp_G(t v_j) \cdot p_0$, $0 \leq t \leq \tau$, 
is a closed geodesic of length $\tau$ for each $j = 1, \ldots , r$. In summary, we have the following.

\begin{prop}\label{prop:FixedSet}
Let $M$ be a symmetric space and $\tau$ an element of its length spectrum.
Then $\Fix(\Phi_\tau)$ is a disjoint union of finitely many homogeneous manifolds 
$\Theta_1, \ldots , \Theta_r$, where for each $j = 1, \ldots, r$, we have  
$\Theta_j = G \cdot \gamma_j'(0)$ for some unit speed closed geodesic $\gamma_j$ of length $\tau$
with $\gamma_j'(0) \in \Theta_j$.
\end{prop}

With this in hand we may now prove Theorem~\ref{thm:SymmSpcsCIH}.

\begin{proof}[Proof of Theorem~\ref{thm:SymmSpcsCIH}]
Let $M = G/K$ be a compact symmetric space and fix $\tau$ in the length spectrum of $M$.
By Proposition~\ref{prop:FixedSet}, we see that $\Fix(\Phi_\tau)$ is a disjoint union of finitely many closed manifolds, 
where each component is of the form $G\cdot u$ for some $u \in \Fix(\Phi_\tau)$.
Since $M$ is a symmetric space, we know that the \emph{periodic} Jacobi fields along any geodesic $\gamma(t)$ in $M$ 
are of the form 
$$J(t) = \frac{d}{ds}|_{s=0} \Exp_G(sX) \cdot \gamma(t) \equiv X^*_{\gamma (t)},$$
where $X \in \germ{g}$.
That is, all Jacobi fields in a (not necessarily compact) symmetric space are restrictions of Killing vector fields (cf. \cite[proof of Theorem 2]{Ziller3}).
It then follows that for any $u \in \Fix(\Phi_\tau)$ we have $\Fix(D_u \Phi_\tau) = T_u G\cdot u = T_u \Fix(\Phi_\tau)$.
Therefore, $\tau$ is a clean length.
\end{proof}

%%%%%%%%%%%%%%%%%%%%%%%%%%%%%
\subsection{Restricted Root Systems}\label{SubSec:RestrictedRoots}
%Let $\germ{a}$ be a maximal abelian subalgebra contained in $\germ{p}$.
%All such subalgebras are conjugate via $\Ad(K)$ and their common dimension 
%is referred to as the \emph{rank} of $M$. Furthermore, we know that 
%$\germ{p} = \cup_{k \in K} \Ad(k) \germ{a}$ and 
%$\germ{a} \leq \germ{p} \equiv T_{p_0}M$ exponentiates to a totally geodesic maximal flat 
%$F_{\germ{a}} \equiv \Exp_{M}(\germ{a}) \simeq T^n \times R^m$ in $M$, 
%where $T^n$ is a flat $k$-torus and $\R^{\ell}$ is $\ell$-dimensional euclidean space.
%We also note here that since every vector $v \in \germ{p} \equiv T_{p_0} M$ 
%can be conjugated into $\germ{a}$, then every geodesic in $M$ can be conjugated via the isometry group 
%into the maximal flat $F_\germ{a}$ in $M$.

Now, let $\germ{g}_{\C} = \germ{g} \oplus i \germ{g}$ denote the complexification of $\germ{g}$.
Since  $\germ{a} \subset \germ{p} \subset \germ{g}$ is abelian we see that $\{\ad(x) : x \in \germ{a} \}$ is a commuting family 
of skew-adjoint linear transformations on $\germ{g}_{\C}$. Hence, they are simultaneously diagonalizable.
For any $\beta \in \germ{a}^* \equiv \Hom(\germ{a}, \R)$ we let $\germ{g}_{\C}^{\beta}$ 
be the subspace of $\germ{g}_{\C}$ defined by 
$$\germ{g}_{\C}^{\beta} = \{ y \in \germ{g}_{\C} : [x,y] = i \pi \beta(x) y \mbox{ for all } x \in \germ{a} \}.$$
The set $R =R(M) \equiv  \{ \beta \neq 0 \in \germ{a}^* : \germ{g}_{\C}^{\beta} \mbox{ is non-trvial } \}$ is 
referred to as the set of \emph{restricted roots of $M$} with respect to $\germ{a}$.
The restricted roots of $M$ only depend on the universal cover $\widetilde{M}$ of $M$
and when $\widetilde{M}$ has no Euclidean factor $R = R(M)$ is a 
root system of $\germ{a}$ (cf. Appendix~\ref{Sec:RootSystems}).
In any event, we have the following decomposition. 
%In the case where $R$ is non-empty (i.e., $M$ is non-flat), 
%we see that $R \subset \germ{a}^*$ forms a root system of $\germ{a}$ (cf. Appendix~\ref{Sec:RootSystems}), 
%which we refer to as the (restricted) \emph{root system of $M$ with respect to $\germ{a}$} and we have
$$\germ{g}_{\C} = \oplus_{\beta \in R} \germ{g}_{\C}^{\beta} = \germ{g}_{\C}^{0} \oplus_{\beta \in R} \germ{g}_{\C}^{\beta}.$$ 

%The symmetric space $M$ is non-flat if and only if $R$ is non-empty ({\bf ??}), 
%and in this case $R$ 
%it is a root system associated to the vector space $\germ{a}$ and we have
%$$\germ{g}_{\C} = \oplus_{\beta \in R} \germ{g}_{\C}^{\beta} = \germ{g}_{\C}^{0} \oplus_{\beta \in R} \germ{g}_{\C}^{\beta},$$
%where $R^*$ denotes the collection of non-zero elements in $R$.
%%And, it can be seen that $\beta \in R$ if and only if $-\beta \in R$ and $s(\germ{g}_{\C}^{\beta}) = \germ{g}_{\C}^{-\beta}.$
%%The set of non-zero roots of $M$ is denoted by $R^*$  and we notice that 
%%for each $\beta \in R^*$, $\ker(\beta)$ is a hyperplane in $\germ{a}$.

When $M$ is not flat, the connected components of $\germ{a} - \cup_{\beta \in R} \ker(\beta)$
are called Weyl chambers.
A choice of a Weyl chamber $\mathcal{C}$ leads to a decomposition 
$R = R^+ \cup R^{-}$ of the roots into positive and negative roots, 
where $R^+ = R^+(\mathcal{C}) \equiv \{ \beta \in R : \beta > 0 \mbox{ on } \mathcal{C} \}$.
%Now fix a Weyl chamber $\mathcal{C} \subset \germ{a} - \cup_{\beta \in R} \ker(\beta)$ 
%and let $R = R^+ \cup R^{-}$ be the corresponding decomposition into positive and negative roots. 
For each positive root $\beta$ we define the real vector space
$$\germ{g}^{\beta} = \germ{g} \cap (\germ{g}_{\C}^{\beta} + \germ{g}_{\C}^{-\beta})$$
and notice that it is  $s$-invariant, since $s(\germ{g}_{\C}^{\beta}) = \germ{g}_{\C}^{-\beta}$ for each $\beta \in R$.
The $s$-invariance of $\germ{g}^{\beta}$ implies that we have the decomposition
$$\germ{g}^{\beta} = \germ{K}^{\beta} + \germ{p}^{\beta}, $$
where $\germ{K}^{\beta} = \germ{K} \cap \germ{g}^{\beta}$ and $\germ{p}^{\beta} = \germ{p} \cap \germ{g}^{\beta}$.
Also, since $\germ{a}$ is a maximal abelian subspace in $\germ{p}$, 
we see that $\germ{g}^{0} = \germ{g} \cap \germ{g}_{\C}^{0} = Z_{\germ{g}}(\germ{a}) = Z_{\germ{K}}(\germ{a}) + \germ{a}$.
One can check that 
\begin{eqnarray}
\germ{p}^{\beta} &=& \{ y \in \germ{p} : \ad(x)^2(y) = -\pi^2 \beta(x)^2 y \mbox{ for all } x \in \germ{a} \} \\
\germ{p} &=& \germ{a} \oplus (\oplus_{\beta \in R^{+}} \germ{p}^{\beta}) \\
n_{\beta} &\equiv& \dim \germ{p}^\beta = \dim \germ{K}^{\beta}, 
\end{eqnarray}
where the integer $n_\beta$ is said to be the \emph{multiplicity} of $\beta$ (in $M$).
Then for any $v \in \germ{a}$ we have 
$$Z_{\germ{g}}(v) = \germ{g}^0 + \sum_{\stackrel{\beta \in R^+ }{ \beta(v) = 0}} \germ{g}^{\beta} = 
Z_{\germ{K}}(\germ{a}) + ( \germ{a} + \sum_{\stackrel{\beta \in R^+ }{ \beta(v) = 0}} \germ{p}^{\beta})$$
If we consider the geodesic $\gamma(t) = \exp_{G}(t v) \cdot p_0$, 
then $K_\gamma \equiv \{ k \in K : k\circ \gamma = \gamma \}$ is identical to the group
$K_{\gamma'(0)} = \{ k \in K : k \cdot \gamma'(0) = \gamma'(0) \}$ and has Lie algebra
$\germ{K}_v \subseteq \germ{K}$ given by 

\begin{eqnarray}\label{eqn:LieKv}
\germ{K}_{v} = \{ x \in \germ{K} : [x,v] = 0 \} = Z_{\germ{K}}(v) = 
Z_{\germ{K}}(\germ{a}) + \sum_{\stackrel{\beta \in R^+ }{ \beta(v) = 0}} \germ{K}^\beta.
\end{eqnarray}
Therefore, letting $G\cdot v \simeq G/G_v = G/K_{v}$ denote the orbit of $v$, 
we see that $\dim G \cdot v \leq \dim G - \dim Z_{\germ{K}}(\germ{a})$ with equality 
if and only if $v$ is a regular vector.
In the case where $M$ is a compact Lie group $U$, we can state this as follows.

\begin{lem}\label{lem:BoundingRank}
Let $(U,g)$ be a compact Lie group equipped with a bi-invariant metric and
$v \in \Lambda_I(U)$ give rise to a closed geodesic of length $\tau$. Then 
$$\dim \Fix_{v}(\Phi_\tau) \leq 2 \dim U - \rank (U),$$
with equality if and only if $v$ is a regular vector.
Here, $\Fix_{v}(\Phi_\tau) = (U \times U) \cdot \frac{v}{\|v\|}$ denotes the component 
of $\Fix(\Phi_\tau)$ containing $\frac{v}{\| v \|}$.
\end{lem}

We now supply the proof of Theorem~\ref{thm:HearingRank}, which states 
that the rank of a compact Lie group $U$ can be recovered from a generic bi-invariant metric
carried by $U$. 

\begin{proof}[Proof or Theorem~\ref{thm:HearingRank}]
For any $g \in \mathscr{W}_0(U)$, it follows from Equation~\ref{eqn:Asymptotic1} 
that $\dim(\Fix(\Phi_\tau))$ is encoded in the spectrum of $(U,g)$.
Also, it follows from the classification of root systems that every integral lattice 
contains a regular vector (cf. Appendix~\ref{Sec:RootSystems}). 
Therefore, by Lemma~\ref{lem:BoundingRank} and 
the fact that the dimension of $U$ is a spectral invariant, we observe that 
$$\rank(U) = 2\dim(U) - \operatorname{max}_{\tau \in \Spec_{L}(U,g)} \dim(\Fix(\Phi_\tau))$$
can be recovered from the spectrum of $(U,g)$.
The last statement of the proposition follows immediately.
\end{proof}

\begin{dfn}\label{dfn:DegSing}
For any $v \in \germ{a}$, its \emph{degree of singularity} is defined to be the nonnegative integer
\begin{eqnarray*}
\operatorname{deg}_{\rm{sing}}(v) &=& \sum_{\stackrel{\beta > 0}{ \beta(v) = 0}} n_\alpha \\
&=&  \dim \germ{K}_{v} - \dim Z_{\germ{K}}(\germ{a}) \\
&=& \dim G - \dim Z_{\germ{K}}(\germ{a}) - \dim G \cdot v\\
&=& \widetilde{C}(M) - \dim G \cdot v,
\end{eqnarray*}
where $G \cdot v$ denotes the orbit of $v$ under the natural action of $G$ on $TM$
%When $v \neq 0$ and the corresponding geodesic $\gamma_v: [0,1] \to (M,g)$ is closed and of length $\tau$,
%$G\cdot v$ is diffeomorphic to $\Fix_v(\Phi_\tau) \equiv G\cdot \frac{v}{\| v\|\}$, the component of $\Fix(\Phi_\tau)$ 
%containing the unit vector $\frac{v}{\|v \|}$.
and $\widetilde{C}(M) \equiv \dim G - \dim Z_{\germ{K}}(\germ{a})$ is a constant depending only on $M$.
\end{dfn}

%%%%%%%%%%%%%%%%%%%%%%%%%%%%%%
\subsection{The Morse Index of a Closed Geodesic}\label{SubSec:MorseIndex}

As we noted in the introduction, in order to understand whether the $0$-th wave invariants 
associated to a clean length $\tau$ vanish 
one needs to be able to compute the Morse index of a closed geodesic. 
In this subsection we recall a result of Ziller which states that in a compact symmetric space, 
the Morse index of a closed geodesic can be computed in terms of the roots of the symmetric space. 

Given a Riemannian manifold $(M,g)$ we let $\mathcal{P}$ denote the space of piecewise smooth curves $c: [0,T] \to M$
and we recall that the energy functional $E: \mathcal{P} \to \R$ is given by 
$$E(c) = \int_{o}^T \| c'(t) \|^2 \; dt.$$
As usual we consider the critical points of $E$ subject to some boundary conditions.
Indeed, for any smooth manifold $B \subset M \times M$ we let $\mathcal{P}_B = \{ c \in \mathcal{P} : (c(0), c(T)) \in B \}$.
Then $c \in \mathcal{P}_B$ is said to be a \emph{stationary curve} of $E$ restricted to $\mathcal{P}_B$, 
if $DE_c(X) = 0$ for any $X \in T_c \mathcal{P}_B$ 
and its \emph{Morse index} (with respect to the boundary condition $B$), denoted by $\sigma_B(c)$, is defined to be 
$$\sigma_{B}(c) \equiv \sup \{ \dim L : L \leq T_c\mathcal{P}_B \mbox{ and } D^2E_c \upharpoonright L \mbox{ is negative deinite} \}.$$
In the case where $B = \{(p,q)\}$ for some $p, q \in M$
or $B = \Delta M \equiv  \{ (p,p) : p \in M\}$, it is known that 
$\sigma_B(c)$ is finite.

In the introduction we saw that in order to compute the $0$-th wave invariant of a clean length $\tau$
in the length spectrum of $(M,g)$, one must be able to compute $\sigma_{\Delta M}(\gamma)$ for any closed geodesic 
$\gamma$ of length $\tau$. In general, one has 
$$ \sigma_{\Delta M}(\gamma) = \sigma_{(\gamma(0), \gamma(T))}(\gamma) + \operatorname{Conv}(\gamma),$$
where $\operatorname{Conv}(\gamma)$ is an integer between $0$ and $\dim M -1$ known as the \emph{concavity}
of $\gamma$ (cf. \cite[Eq. $1.4'$]{BTZ}). However, in the case where our manifold is homogeneous, Ziller has shown that the concavity vanishes \cite[Theorem 1]{Ziller1}. Hence, it follows from the Morse index theorem that we may compute $\sigma_{\Delta M}(\gamma)$ in terms of conjugate points. And, in the case of a \emph{compact} symmetric space this can be realized in terms of the restricted roots of $M$.

%\begin{thm}[Theorem 1 \cite{Ziller1}]\label{thm:MorseIndex}
%Let $(M,g)$ be a Riemannian homogeneous space and $\gamma: [0,T] \to M$ 
%a closed geodesic with $\gamma(0) = \gamma(T) = p$, then $ i_{\Delta M}(\gamma) = i_{(p,p)}(\gamma)$.
%That is $i_{\Delta M}(\gamma)$ equals the number of conjugate points (counting multiplicities) of $\gamma(0) = p$ 
%along $\gamma$ restricted to $0 < t < T$. 
%\end{thm}

%\noindent
%In fact, in the case of a \emph{compact} symmetric space $M$ we may compute the index in terms of its roots.

\begin{thm}[\cite{Ziller3}, p. 11-12]\label{thm:MorseIndexRoots}
With the notation as above, assume that $M = G/K$ is a compact symmetric space.
Now, let $v \in \germ{a} \subset \germ{p}$ be such that $\gamma(t) = \exp_G(tv) \cdot p_0$, $0 \leq t \leq 1$, 
is a closed geodesic. Then the Morse index of $\gamma$ in the space of closed geodesics in $M$ is given by 
\begin{eqnarray}\label{eqn:ZillerMorseIndex}
\sigma_{\Delta M}(\gamma) &=& \sum_{\beta \in R^+ \cup \{0\}} n_\beta | \beta(v) | - 
\sum_{\stackrel{\beta \in R^+}{\beta(v) \neq 0}} n_\beta. 
%\\ &=& \sum_{\beta \in R^+} n_\beta | \beta(v) | - (\dim(M) - \Rank(M) - \DegSing(v))
\end{eqnarray}
%and $\beta(v) \in \Z_{\geq 0}$ for all $\beta \in R^{+}$.
\end{thm}

%\noindent 
If $\gamma_v$ in the above is a closed geodesic of length $\tau = \|v\|$, then in terms of the degree of singularity,
Equation~\ref{eqn:ZillerMorseIndex} becomes 
%the expression in Theorem~\ref{thm:MorseIndexRoots} becomes 
$\sigma_{\Delta M}(\gamma_v) = \sum_{\beta \in R^+} n_\beta | \beta(v) | - \dim(M) + \Rank(M) + \DegSing(v)$, 
%which we may write as 
%\begin{eqnarray}\label{eqn:MorseIndex}
%\sigma_{\Delta M}(\gamma_v) &=& \sum_{\beta \in R^+} n_\beta | \beta(v) | - \dim G\cdot v + \dim M.
%%\sigma_{\Delta M}(\gamma_v) &=& \sum_{\beta \in R^+} n_\beta | \beta(v) | - \dim(\Fix_{v}(\Phi_\tau)) + C(M),
%\end{eqnarray}
%It follows that 
and, since $\germ{K} = Z_{\germ{K}}(\germ{a}) \oplus (\oplus_{\beta \in R^+} \germ{K}^\beta)$, it follows that 
\begin{eqnarray}\label{eqn:MorseIndex}
\sigma_{\Delta M}(\gamma_v) = \left\{ \begin{array}{ll}
0 & \mbox{for $v = 0$} \\
\sum_{\beta \in R^+} n_\beta | \beta(v) | - \dim \Fix_v(\Phi_\tau) + \dim M & \mbox{for $v \neq 0$},
\end{array}
\right.
\end{eqnarray}
where $\Fix_v(\Phi_\tau) \equiv G \cdot \frac{v}{\| v \|}$ is the component of $\Fix(\Phi_\tau) \subseteq SM$
containing $\frac{v}{\|v\|}$. 
We will say that a vector $v \in \germ{a}$ is \emph{Euclidean} if $\beta(v) = 0$ for all $\beta \in R^+$ or, 
equivalently, $G \cdot v$ is diffeomorphic to $M$.
Since $\Fix_v(\Phi_\tau)$ is diffeomorphic to $G\cdot v$ for $v \neq 0$, we conclude
that when $v$ is Euclidean $\sigma_{\Delta M}(\gamma_v)$ vanishes, as expected. 
%When $v \neq 0$, $G\cdot v$ is diffeomorphic to $\Fix_v(\Phi_\tau) \equiv G\cdot \frac{v}{\| v\|}$
%is the component of $\Fix(\Phi_\tau) \subseteq SM$ containing the unit vector $\frac{v}{\|v\|}$.
%where $\Fix_v(\Phi_\tau) \equiv G \cdot \frac{v}{\|v\|}$ denotes the component of $\Fix(\Phi_\tau)$ containing 
%the unit vector $\frac{v}{\| v \|}$ and $C(M)$ is a constant. 
Now, as we noted in Section~\ref{SubSec:GeodSymmSpaces}, all geodesics in $M$ can be conjugated via the isometry group
to a geodesic having initial conditions in $\germ{a}$. Therefore, we have an effective means for computing the Morse Index of any closed 
geodesic in a compact symmetric space.

%We recall that the diagram of the root system $R \subset \germ{a}^*$ is given by  
%$D = \cup_{\beta \in R^*} \beta^{-1}(\Z)$ it determines %the conjugate points of the metric. 
%Indeed, if $x \in \germ{a}$, then $\exp_G(x) \cdot p_0$ is conjugate to 
%$p_0$ along $\exp_{G}(tx)$ if and only if $\beta(x) \in \Z - \{0\}$ for some $\beta \in R^{*}$ and 
%the number of such positive roots $\beta$ counted with their multiplicity (i.e., $n_\beta$) 
%is the multiplicity of the conjugate point (cf. \cite[Sec. 2]{Ziller3}). 
%{\bf Maybe I should review the solution of the Jacobi equation to justify this remark.}

 %%%%%%%%%%%%%%%%%%%%%%%%%%%%%%
\subsection{The Center of a Symmetric Space and Some Useful Lattices}\label{SubSec:TheCenter}
We recall that an isometry $f$ of a Riemannian manifold is said to be a \emph{transvection}
if there is a point $q$ at which $df$ is equal to the parallel transport 
along a piecewise smooth curve joining $q$ to $f(q)$.
If we let $\mathscr{T}(M)$ denote the group of transvections of a symmetric space $M$, then 
the \emph{center} of $M$, denoted by $Z(M)$, is defined to be the centralizer of $\mathscr{T}(M)$ 
in $\Isom(M)$. We then have the following well-known result concerning Riemannian coverings.

\begin{prop}[\cite{Wolf}, Theorem 8.3.11] Let $M$ be a symmetric space with center $Z(M)$.

\begin{enumerate}
\item Let $\pi: M \to N$ be a Riemannian covering with deck transformation group $\Gamma$.
If $N$ is a symmetric space, then $\Gamma$ is a discrete subgroup of $Z(M)$.

\item If $\Gamma \leq Z(M)$ is a discrete subgroup, then $N = \Gamma \backslash M$ is a symmetric space
and the natural projection $\pi : M \to \Gamma \backslash M$ is a Riemannian covering 
with deck transformation group $\Gamma$.
\end{enumerate}

\end{prop}
\noindent
Consequently, if $\widetilde{M}$ is a simply-connected symmetric space, then any symmetric space covered by $\widetilde{M}$ 
is of the form $\Gamma \backslash \widetilde{M}$ for some discrete subgroup of $Z(\widetilde{M})$.
%It follows that any compact symmetric space $M$ has homogeneity type given by  
%$$M= \Gamma \backslash (M_0 \times \widetilde{M}_1 \times \cdots \times \widetilde{M}_q),$$
%where $M_0$ is a torus of dimension $d$, $\widetilde{M}_j$ is a simply-connected irreducible symmetric space
%and $\Gamma \leq Z(M_0 \times \widetilde{M}_1 \times \cdots \times \widetilde{M}_q) = 
%M_0 \times Z(\widetilde{M}_1) \times \cdots \times Z(\widetilde{M}_q)$ is a discrete group. 
%After replacing $M_0$ with $M_0 / \Gamma \cap M_0$, we may assume that 
%$\Gamma \cap M_0$ is trivial. 
%It then follows that in any dimension there are only finitely many homogeneity 
%types of compact symmetric spaces (cf. \cite[Lemma 2.8]{GorSut}). 

Let $R = R(M)$ be the restricted roots of a symmetric space $M$ with respect to some 
maxiaml abelian subspace $\germ{a} \subset \germ{p} \equiv T_{p_0}M$.
Then for each $\alpha \in R$ we define its associated \emph{co-root}
$\alpha\,\check{}$ to be the vector in $\germ{a}$ satisfying
\begin{enumerate}
\item $\alpha\,\check{}$ is orthogonal to $\ker \alpha$,
\item  $\alpha(\alpha\,\check{} \; ) = 2$.
\end{enumerate}
The co-root, central and integral lattices of $M$ 
are defined by 
\begin{eqnarray*}
\Lambda_{R\,\check{}}(M) &=& \langle \alpha\,\check{} : \alpha \in R \rangle\\
\Lambda_Z(M) &=& \{ v \in \germ{a} : \alpha(v) \in \Z \mbox{ for all } \alpha \in R \}\\
\Lambda_{I}(M) &=& \{ v \in \germ{a} : \exp_{G}(v) \cdot p_0 = p_0 \}.
\end{eqnarray*}
Clearly, if $M_1$ and $M_2$ are symmetric spaces, then 
$\Lambda_Y(M_1 \times M_2) = \Lambda_Y(M_1) \times \Lambda_Y(M_2)$, where $Y = R\,\check{}, Z, I$.
 
%Now, we recall that if $M$ is a symmetric space and $R$ is the restricted root system associated to some 
%maximal abelian subspace $\germ{a} \leq \germ{p} \equiv T_{p_0}(M)$, then the co-root, central and integral lattices of $M$ 
%are given by 
%\begin{eqnarray*}
%\Lambda_{R\,\check{}}(M) &=& \langle \alpha\,\check{} : \alpha \in R \rangle\\
%\Lambda_Z(M) &=& \{ v \in \germ{a} : \alpha(v) \in \Z \mbox{ for all } \alpha \in R \}\\
%\Lambda_{I}(M) &=& \{ v \in \germ{a} : \exp_{G}(v) \cdot p_0 = p_0 \},
%\end{eqnarray*}
%where in the case that $M$ is flat and of dimension $k$ we set $\Lambda_{R\,\check{}}(M) = \{0\}$ and 
%$\Lambda_Z(M) = \R^k$. 
%It is clear that the co-root and central lattices only depend on the isomorphism type of the 
%restricted root system of $M$, while the integral lattice is a geometric construct.
%We then have the following useful relations.

\begin{prop}[cf. \cite{Loos} Theorems VI.2.4 and VI.3.6]\label{prop:TheLattices} %\textrm{}\\
Let $M = \Gamma \backslash (M_0 \times \widetilde{M}_1 \times \cdots \times \widetilde{M}_q)$ 
be the homogeneity type of a compact symmetric space and set $d = \dim M_0$. 
Then the following hold.
\begin{enumerate}
\item $\Lambda_I(M_0 \times \widetilde{M}_{\rm{cpt}}) \simeq \Z^d \times \Lambda_I(\widetilde{M}_{\rm{cpt}}) =
\Z^d \times \Lambda_I(\widetilde{M}_1) \times \cdots \times \Lambda_I(\widetilde{M}_q)$

\item $\Lambda_{R\,\check{}}(M) = \Lambda_{R\,\check{}}(M_0 \times \widetilde{M}_{\rm{cpt}}) = 
0 \times \Lambda_{R\,\check{}}(\widetilde{M}_1) \times \cdots \times \Lambda_{R\,\check{}}(\widetilde{M}_q)$.

\item $\Lambda_Z(M) = \Lambda_Z(M_0 \times \widetilde{M}_{\rm{cpt}}) = \R^d \times 
\Lambda_Z(\widetilde{M}_1) \times \cdots \times \Lambda_Z(\widetilde{M}_q)$.

\item $\Lambda_{R\,\check{}}(M) \leq \Lambda_{I}(M) \leq \Lambda_{Z}(M)$.

\item $Z(M) \simeq \Lambda_{Z}(M) / \Lambda_{I}(M)$

\item Let  $\pi: \Lambda_Z(M_0 \times \widetilde{M}_{\rm{cpt}}) \to 
Z(M_0 \times \widetilde{M}_{\rm{cpt}}) \equiv 
\Lambda_Z(M_0 \times \widetilde{M}_{\rm{cpt}}) /\Lambda_I(M_0 \times \widetilde{M}_{\rm{cpt}})$
be the natural projection, then 
$$\Lambda_I(M) = \pi^{-1}(\Gamma)$$

\item $\pi_1(M) \simeq \Lambda_{I}(M) / \Lambda_{R\,\check{}}(M)$; in particular, 
$\Lambda_I(\widetilde{M}_{\rm{cpt}}) =  \Lambda_{R\,\check{}}(\widetilde{M}_{\rm{cpt}})$.

\end{enumerate}

\end{prop}

%%%%%%%%%%%%%%%%%%%%%%%%%%%%
%%%%%%%%%%%%%%%%%%%%%%%%%%%%
%%%%%%%%%%%%%%%%%%%%%%%%%%%%
\section{Hearing the Length Spectrum of a Split-Rank Symmetric Space}\label{sec:HearingLengthSpectrum}

In this section we complete the proof of Theorem~\ref{thm:MainResult} 
by proving Proposition~\ref{prop:Component Dependence}, Lemma~\ref{lem:CLUResidual}
 and Theorem~\ref{thm:IntroMod4}.
We also provide an infinite family of examples demonstrating that Theorem~\ref{thm:IntroMod4}
cannot hold for an arbitrary (irreducible) symmetric space, 
which suggests that exploring the conjecture through the trace formula of Duistermaat and Guillemin
will prove to be more delicate in general.

\subsection{Component Lengths and the Morse Index}\label{SubSec:PropComponentDependence}
In this section we provide the argument in support of Proposition~\ref{prop:Component Dependence}
which tells us that in a compact symmetric space $M$ for which the non-Euclidean part of its universal cover
is split-rank, the Morse index modulo $4$ of a closed geodesic $\gamma$ of length $\tau$ is determined by 
the length of the components of $\gamma'(t)$ and the dimension of $\Fix_\gamma (\Phi_\tau)$, the component of $\Fix(\Phi_\tau)$ corresponding to $\gamma$.

\begin{proof}[Proof of Proposition~\ref{prop:Component Dependence}]
Let $M = \Gamma \backslash (M_0 \times \widetilde{M}_1 \times \cdots \times \widetilde{M}_q)$ be 
the homogeneity type of a compact symmetric space, 
where $\widetilde{M}_1 \times \cdots \times \widetilde{M}_q$ is split-rank
or, equivalently, $\widetilde{M}_j$ is split rank for $j = 1, \ldots , q$,
and let $g$ be the symmetric metric on $M$ induced by the 
symmetric metric $h \times g_1 \times \cdots \times g_q$ on  
$M_0 \times \widetilde{M}_1 \times \cdots \times \widetilde{M}_q$.
Furthermore, for $j = 0, 1, \ldots , q$, $\Gamma_j$ will denote the projection 
of $\Gamma \leq Z(M_0 \times \widetilde{M}_1 \times \cdots \times \widetilde{M}_q) = M_0 \times Z(\widetilde{M}_1) \times \cdots \times Z(\widetilde{M}_q)$ onto the $j$-th factor and $M_j \equiv \Gamma_j \backslash \widetilde{M}_j$.

Now, for each $j =1, \ldots , q$, let $R_j$ 
denote the restricted root system of $\widetilde{M}_j = G_j / K_j$
with respect to some choice of maximal abelian subspace 
$\germ{a}_j \subset \germ{p}_j \equiv T_o \widetilde{M}_j$
and let $R_j^+$ denote the positive roots with respect to some choice of 
Weyl chamber $\mathcal{C}_j \subset \germ{a}_j$.
Then, letting $\germ{a}_0$ denote the Lie algebra of the torus $M_0$,
the restricted root system for $M$ with respect to 
the maximal abelian subspace 
$\germ{a} = \germ{a}_0 \times \germ{a}_1 \times \cdots \times \germ{a}_q \subset \germ{p} \equiv T_oM$ 
is given by 
$$R = \{ \widetilde{\beta}_j = \beta_j \circ \pi_j : \beta_j \in R_j \mbox{ for } j = 1, \ldots , q \},$$
where $\pi_j: \germ{a}_1 \times \cdots \times \germ{a}_q \to \germ{a}_j$ is given by 
$v = (v_1, \ldots , v_q) \mapsto v_j$, and the positive roots with respect to the Weyl chamber 
$\mathcal{C} = \oplus_{j=1}^{q} \mathcal{C}_j$ are given by 
$$R^+ = \{ \widetilde{\beta}_j = \beta_j \circ \pi_j : \beta_j \in R_j^+ \mbox{ for } j = 1, \ldots , q \}.$$
It follows from Equation~\ref{eqn:MorseIndex} that for any $v = (v_0, v_1, \ldots , v_q) \neq 0 \in \overline{\mathcal{C}} \cap \Lambda_I(M) \subset \overline{\mathcal{C}} \cap \Lambda_Z(M)$
the Morse index of the associated closed geodesic $\gamma_v$ is given by 
\begin{eqnarray}\label{Eq:MorseIndexFactors}
%\sigma_{\Delta M}(\gamma_v) &=& \sum_{\beta \in R^+} n_{\beta} | \beta(v) | - \dim(\Fix_{\gamma_v}(\Phi_\tau)) + C(M) \\
%&=& \sum_{j =1}^{q} \sum_{\beta_j \in R_j^+} n_{\beta_j} |\beta_j(v)| - \dim(\Fix_{\gamma_v}(\Phi_\tau)) + C(M) \\
\sigma_{\Delta M}(\gamma_v) &=& \sum_{\beta \in R^+} n_{\beta} \beta(v) - \dim(\Fix_{\gamma_v}(\Phi_\tau)) + C(M)\\
&=& \sum_{j =1}^{q} \sum_{\beta_j \in R_j^+} n_{\beta_j} \beta_j(v_j) - \dim(\Fix_{\gamma_v}(\Phi_\tau)) + C(M), \nonumber
\end{eqnarray}
where, as in Section~\ref{Sec:MorseIndex}, $n_\beta$ and $n_{\beta_j}$ denote the multiplicities of the roots.
Therefore, each non-Euclidean factor $\widetilde{M}_j$ of the universal cover of $M$ makes a contribution of 
$\sum_{\beta_j \in R_j^+} n_{\beta_j} \beta_j(v_j)$ to the Morse index of $\gamma_{v}$.
We will show that for each $j=1, \ldots , q$ there is a function 
$f_{(M_j, g_j)}$ such that 
%\begin{equation}\label{eqn:ComponentLengthEquation}
 $\sum_{\beta_j \in R_j^+} n_{\beta_j} \beta_j(v_j) \equiv f_{(M_j, g_j)} (\| v_j \|) \mod 4$,
 as $v_j$ ranges over $\Lambda_I(M_j)$.
%\end{equation} 
Then, since $\Lambda_I(M) \leq \Lambda_I(M_0) \times \Lambda_I(M_1) \times \cdots \times \Lambda_I(M_q)$,
the theorem follows by taking $f_{(M,g)} \equiv \sum_{j=1}^{q} f_{(M_j, g_j)}$ 
and recalling that every closed geodesic is conjugate via the isometry group 
to a closed geodesic in our chosen maximal flat.

In light of the previous discussion, henceforth, we will assume that $(M,g)$ 
is an irreducible split-rank symmetric space and we set out to prove
there is a function $f_{(M,g)}$ such that 
\begin{equation}\label{eqn:ComponentLengthEquation}
 \sum_{\beta \in R+} n_{\beta} \beta(v) \equiv f_{(\widetilde{M}, g)} (\| v \|) \mod 4,
 \end{equation}  
for any $v \in \Lambda_I(M) \cap \overline{\mathcal{C}}$.
At times the following observation will be useful. 

\begin{lem}\label{lem:UsefulLemma}
Let $v = (c_1, \ldots , c_n), w = (d_1, \ldots , d_n) \neq 0 \in \Z^n$ be such that 
$\sum_{j=1}^{n} c_j^2 = \sum_{i =1}^{n} d_j^2 = \tau^2$, then 
$\sum_{j=1}^{n} c_j$ and  $\sum_{j=1}^{n} d_j$ are congruent modulo $2$.
That is, $\|v\|^2 = \|w\|^2$, implies 
the number of odd $c_n$'s is even if and only if the number of odd $d_n$'s is even.
\end{lem}

\begin{proof}
To the contrary, lets assume that $\sum_{j=1}^{n} c_j = 2k+1$ is odd and 
$\sum_{j=1}^{n} d_j = 2m$ is even.
Then $(2k+1)^2 = (\sum_{j=1}^{n} c_j)^2 = \tau^2 + 2 \sum_{i < j} c_ic_j$
and $(2m)^2 = (\sum_{j=1}^{n} d_j)^2 = \tau^2 + 2\sum_{i<j}d_id_j$. 
This implies that $(2k+1)^2 - (2m)^2$ is even, which is a contradiction.
\end{proof}

%Henceforth, we may now assume that $(M,g)$ is an irreducible split-rank symmetric space. 
%Since $(\widetilde{M}_j, g_j)$ is an irreducible split-rank space we know there exists an even integer 
%$N_j$ such that for each $\beta_j \in R_j^{+}$, we have $n_{\beta_j} = N_j$.
%From the classification of irreducible symmetric spaces one can see that $N_j = 2, 4$ or $8$, and 
%$N_j =2$ if and only if $(\widetilde{M}_j, g_j)$ is a simple Lie group equipped with a bi-invariant metric
%(i.e., $(\widetilde{M}_j, g_j)$ is Type II).
%Therefore, since roots are integer valued on any 
%$v_j \in \overline{\mathcal{C}}_j \cap \Lambda_{Z}(\widetilde{M}_j)$ 
%we conclude that when $(\widetilde{M}_j, g_j)$ is Type I, 
%the quantity $\sum_{\beta_j \in R_{j}^{+}} n_{\beta_j} \beta_j (v_j)$ 
%is divisible by four and we may take 
%$$f_{(\widetilde{M}_j, g_j)} \equiv 0.$$
%It will be useful to note that this conclusion also holds for 
%any simply-connected irreducible split-rank space.

\begin{lem}\label{lem:SimplyConnectedSplitRank}
Let $(M,g)$ be an irreducible simply-connected split rank symmetric space.
Then $f_{(M,g)} \equiv 0$.
\end{lem}

\begin{proof}
Since $(M, g)$ is an irreducible split-rank space we know there is an integer 
$N$ such that for each $\beta \in R^{+}$, we have $n_{\beta} = 2N$ 
with $N =1$ if and only if $M$ is a simple Lie group \cite[Theorems~VI.4.3 and VI.4.4]{Loos}.
(In fact, from the classification of irreducible symmetric spaces, 
one can see that $2N$ equals $2$, $4$ or $8$.)
Since $M$ is simply-connected, we know that the co-root lattice and integral lattice coincide: 
$\Lambda_{R^{\, \check{}}}(M) = \Lambda_I(M)$. 
Therefore, since the lowest strongly dominant form $\rho = \frac{1}{2} \sum_{\beta \in \R^+} \beta$ 
is integer-valued on the co-root lattice (see Lemma~\ref{lem:Rho}), 
we see that for any $v \in \overline{\mathcal{C}}\cap \Lambda_I(M)$ , 
we have $\sum_{\beta \in R^+} n_\beta \beta(v) = 4N \rho(v) \equiv 0 \mod 4$.
Therefore, we may take $f_{(M,g)}$ to be identically zero,
which concludes the proof of the lemma.
\end{proof}

\begin{lem}
Let $(M,g)$ be an irreducible split-rank symmetric space that is not a simple Lie group; 
that is, $(M,g)$ is an irreducible split-rank symmetric space of Type I.
Then $f_{(M,g)} \equiv 0$ satisfies Equation~\ref{eqn:ComponentLengthEquation}.
\end{lem}

\begin{proof}
%Since $(M, g)$ is an irreducible split-rank space we know there is an integer 
%$N$ such that for each $\beta \in R^{+}$, we have $n_{\beta} = 2N$ 
%with $N =1$ if and only if $M$ is a simple Lie group \cite[Theorems?? and ??]{Loos}.
%(In fact, from the classification of irreducible symmetric spaces, 
%one can see that $2N$ equals $2$, $4$ or $8$.)
As we noted in the proof of Lemma~\ref{lem:SimplyConnectedSplitRank}, 
since $(M,g)$ is split-rank and not a simple Lie group,
there is an integer $N \geq 2$, such that for each $\beta \in R^{+}$, we have $n_{\beta} = 2N$. 
Therefore, since the roots are integer-valued on any 
$v \in \overline{\mathcal{C}} \cap \Lambda_{Z}(\widetilde{M})$ 
we conclude that when $(M, g)$ is Type I, 
the quantity $\sum_{\beta \in R^{+}} n_{\beta} \beta (v)$ 
is divisible by four and we may take $f_{(M, g)} \equiv 0$.
\end{proof}

We now consider the case where $(M,g)$ is a simple Lie group $U$
equipped with a bi-invariant metric.
As we noted in the proof of Lemma~\ref{lem:SimplyConnectedSplitRank}, 
$n_\beta = 2$ for each $\beta \in R^{+}$.
By examining each type of simple Lie group we will produce a function $f_{(M,g)}$ 
such that 
$$\sum_{\beta \in R^+} n_\beta \beta (v) = 4 \rho (v)  \equiv f_{(M,g)} (\|v \|) \mod 4.$$
In many of the cases considered this function will be identically zero.

\medskip
\noindent
{\bf Standing Assumption:} 
Letting $B: \germ{u} \times \germ{u} \to \R$ be the Killing form on 
the Lie algebra $\germ{u}$ of $U$, 
for the remainder of the argument we will let $\bar{g}$ be the metric on 
$M = U \times U / \Delta U$ induced by $-r(B\oplus B)$, 
where $r >0$ is chosen to agree with the inner-product structure in Section~\ref{SubSec:Classification}.
With this choice of $r$, $\bar{g}$ is the bi-invariant metric on 
$M \simeq U$ induced by $-2r B$ on $\germ{u}$. 
Then, the metric $g$ equals $c \bar{g}$ for some $c >0$.

%We now consider the case where $\widetilde{M}_j$ is a simple Lie group $U$ 
%equipped with a bi-invariant metric, so that the multiplicities of the restricted roots are all equal to $2$.  
%By examining each type of simple Lie group we will demonstrate the existence of a function
%$f_{(M,g)}$ such that 
%$$\sum_{\beta \in R^+} n_\beta \beta (v) \equiv f_{(M,g)} (\|v \|) \mod 4.$$
%We note that in the case where $(M,g)$ is of type $A_n$ ($n \equiv 0 \mod 2$), 
%$C_n$ ($n\equiv 0,3 \mod 4$), $D_n$ ($n \equiv 0,1 \mod 4$), $F_2$, $G_2$, $E_6$ or $E_8$, we may take 
%this function to be identically zero.

%\medskip
%\noindent
%{\bf Standing Assumption:} 
%Since $M = U \times U / \Delta U$ is irreducible, throughout the argument we may take the symmetric metric 
%$g$ on $M$ to be the bi-invaraint metric induced by $-c(B \oplus B)$ on $U \times U$, 
%where $B : \germ{u} \times \germ{u} \to \R$ 
%is the Killing form and $c$ is chosen to agree with the inner product structures 
%in Section~\ref{SubSec:Classification}.
%That is, $g$ is the bi-invariant metric on $M \simeq U$ induced by $-cB$ on $\germ{u}$.

\begin{lem}\label{lem:TheFunctionTypeA_n}
Let $(M,g)$ be a simple group of type $A_n$ $(n \geq 1)$ equipped with a bi-invariant metric $g$.
That is, $M = \SU(n+1) / \Gamma$, where $\Gamma \leq \Z_{n+1}$, with $g = c \bar{g}$ for some 
$c >0$.
\begin{enumerate}
\item If $n$ is even or $\Gamma \neq \Z_{n+1}$, then $f_{(M,g)} \equiv 0$
satisfies Equation~\ref{eqn:ComponentLengthEquation}.
\item If $n$ is odd and $\Gamma = \Z_{n+1}$, then 
 $$f_{(M,g)}(\| v \|) = \left\{ \begin{array}{ll} 
0 & \mbox{when $(n+1)\frac{\|v\|^2}{c}$ is an even integer} \\
2 & \mbox{when $(n+1)\frac{\|v\|^2}{c}$ is an odd integer}. \\
\end{array}
\right.
$$ 
satisfies Equation~\ref{eqn:ComponentLengthEquation}.
\end{enumerate}
\end{lem}

\begin{proof}
Using the notation of Section~\ref{sec:TypeAn}, we recall that the central lattice is given by
$ \Lambda_{Z} = \langle L_j = \frac{n}{n+1} e_j - \frac{1}{n+1} \sum_{\stackrel{k =1}{k\neq j}}^{n+1} e_{k} : j = 1, \ldots, n\rangle$. 
%has $\Z$-basis $L_j = \frac{n}{n+1} e_j - \frac{1}{n+1} \sum_{\stackrel{k =1}{k\neq j}}^{n+1} e_{k}$, $j =1, \ldots , n$.
Now, let $v = \sum_{j=1}^{n} k_j L_j \in \Lambda_{I} \cap \overline{\mathcal{C}} \subseteq \Lambda_{Z} \cap \overline{\mathcal{C}}$, then 
\begin{eqnarray*}
2\sum_{\alpha \in R^+} \alpha (v) &=& \sum_{j=1}^{n} 2k_j \sum_{\alpha \in \Delta^+} \alpha(L_j) \\
&=&\sum_{j=1}^{n} 2k_j \sum_{\mu =1}^{n} 2(n-\mu +1) \epsilon_\mu (L_j) \\
&=&\sum_{j=1}^{n} 4k_j \sum_{\mu =1}^{n} (n-\mu +1) \epsilon_\mu (L_j) \\
&=&\sum_{j=1}^{n} 4k_j ((n-j+1)\frac{n}{n+1} - \frac{1}{n+1} \sum_{\stackrel{\mu =1}{\mu \neq j}}^{n} (n-\mu +1)) \\
&=&\sum_{j=1}^{n} 4k_j \left((n-j+1)\frac{n}{n+1} - \frac{1}{n+1} \left(\frac{n(n+1)}{2} - (n-j+1)\right)\right)\\ 
&=&\sum_{j=1}^{n} 4k_j \left((n-j+1) - \frac{n}{2} \right)\\
&=&\sum_{j=1}^{n} 4k_j \left( \frac{n-2j +2}{2} \right)\\
&=&\sum_{j=1}^{n} 2k_j (n - 2j +2)\\
&\equiv& 2n(\sum_{j=1}^{n} k_j) \mod 4
\end{eqnarray*}
Therefore, in the case where $n$ is even, this quantity is always congruent to zero modulo $4$, 
so we may take $f_{(M,g)} \equiv 0$.

In the case that $n$ is odd, 
we see that $2 \sum_{\alpha \in \Delta^+} \alpha (v)$ is congruent to $0$ (resp. $2$) modulo $4$
if the number of $k_j$'s that are odd is even (resp. odd). However,
the parity of the number of odd $k_j$'s is determined by the length of the vector $v$.
Indeed, $\| v\|^2 = \frac{n}{n+1} \sum_{j=1}^{n} k_j^2 - \frac{2}{n+1} \sum_{i < j} k_i k_j$, so that 
$(n+1) \|v \|^2$ is an integer and this integer is even 
if and only if the number of $k_j$'s that are odd is even.
To see which quotients $\SU(2n+1)/\Gamma$ admit $v \in \Lambda_I(\SU(2n+1)/\Gamma)$ 
with $(n+1)\|v\|^2$ odd,
recall that $\Lambda_Z$ is also generated by $\langle L_1, e_\mu - e_\nu: 1 \leq \mu < \nu \leq n+1 \rangle$
(See Section~\ref{sec:TypeAn}).
Therefore, we may express $v \in \Lambda_{I} \cap \overline{\mathcal{C}}$ 
as $v = \hat{k}_1 L_1 + \sum_{i<j} c_{ij} (e_i - e_j) = \hat{k}_1 L_1 + \sum_{j=1}^{n+1} C_j e_j$. 
It follows that when $n$ is odd, 
$(n+1) \| v\|^2$ is an odd integer if and only if $\hat{k}_1$ is odd. 
By Equation~\ref{eqn:IntegralLatticeA_n} the integral lattice of $M = \SU(n+1) /\Gamma$ is 
$\Lambda_I(M) = \langle k L_1 \rangle + \Lambda_{R^{\; \check{}}}$,
where $k = 0, 1, \ldots , n$ is the smallest generator of $\Gamma \leq \Z_{n+1}$.
Therefore, in the event that $n$ is odd, $\hat{k}_1$ can be odd only when 
$\Gamma = \Z_{n+1}$. And, the lemma follows.
\end{proof}

\begin{lem}\label{lem:TheFunctionTypeB_n}
Let $(M,g)$ be a simple group of type $B_n$ $(n \geq 2)$ equipped with a bi-invariant metric $g$.
That is, $M = \Spin(2n+1) / \Gamma$, where $\Gamma \leq Z(\Spin(2n+1) \simeq \Z_{2}$,
with $g = c\bar{g}$ for some $c >0$.
\begin{enumerate}
\item If $\Gamma$ is trivial, then $f_{(M,g)} \equiv 0$
satisfies Equation~\ref{eqn:ComponentLengthEquation}.
\item If $\Gamma = \Z_2$, then 
 $$f_{(M,g)}(\| v \|) = \left\{ \begin{array}{ll} 
0 & \mbox{when $\frac{\|v\|^2}{c}$ is an even integer} \\
2 & \mbox{when $\frac{\|v\|^2}{c}$ is an odd integer}. \\
\end{array}
\right.
$$ 
satisfies Equation~\ref{eqn:ComponentLengthEquation}.
\end{enumerate}
\end{lem} 

\begin{proof}
The first statement follows from Lemma~\ref{lem:SimplyConnectedSplitRank}.
Now, assume that $\Gamma = \Z_2$, so that the integral lattice $\Lambda_I(M)$ equals the 
central lattice $\Lambda_Z$.
As is noted in Section~\ref{sec:TypeBn} the central lattice is given by 
$\Lambda_Z = \langle e_1, e_1 - e_2, \ldots , e_1 - e_n \rangle = \langle e_1, \ldots , e_n \rangle$
and the sum of the positive roots is $\sum_{\alpha \in R^+} \alpha = \sum_{j =1}^{n} (2n -2j +1) \epsilon_j$.
Then, for $v = \sum_{j =1}^{n} k_j e_j \in \Lambda_{I} \cap \overline{\mathcal{C}} =  \Lambda_Z \cap \overline{\mathcal{C}}$
we have

\begin{eqnarray*}
2 \sum_{\alpha \in R^+} \alpha (v) &=& \sum_{j = 1}^{n} 4(n- j)\epsilon_j(v) + 2 \sum_{j =1}^{n} \epsilon_j(v) \\
&\equiv& 2 \sum_{j =1}^{n} \epsilon_j(v) \mod 4 \\
&=& 2\sum_{j =1}^{n} k_j \mod 4.
\end{eqnarray*}
\noindent
Then, by Lemma~\ref{lem:UsefulLemma}, 
if $w = \sum_{j=1}^{n} \hat{k}_j e_j \in \Lambda_{I} \cap \overline{\mathcal{C}}$ is such that $\frac{\| w \|^2}{c} = 
\frac{\| v\|^2}{c}$,
we have $\sum_{j=1}^{n} k_j$ is congruent to $\sum_{j=1}^{n} \hat{k}_j$ modulo $2$, which implies that 
$$2 \sum_{j=1}^{n} k_j \equiv 2 \sum_{j=1}^{n} \hat{k}_j \mod 4.$$
In the case where $\frac{\| w \|^2}{c} = 
\frac{\| v\|^2}{c}$ is an even integer, this quantity is zero modulo $4$. In the case where 
$\frac{\| w \|^2}{c} = 
\frac{\| v\|^2}{c}$ is odd, this quantity is $2$ modulo $4$.
\end{proof}

\begin{lem}\label{lem:TheFunctionTypeC_n}
Let $(M,g)$ be a simple group of type $C_n$ $(n \geq 3)$ equipped with a bi-invariant metric $g$.
That is, $M = \Sp(n) / \Gamma$, where $\Gamma \leq Z(\Sp(n)) \simeq \Z_{2}$, with $g = c \bar{g}$ for some $c >0$.
\begin{enumerate}
\item If $n \equiv 0, 3 \mod 4$ or $\Gamma$ is trivial, then $f_{(M,g)} \equiv 0$
satisfies Equation~\ref{eqn:ComponentLengthEquation}.
\item If $n \equiv 1,2 \mod 4$ and $\Gamma = Z (\Sp(n)) \simeq \Z_2$, then 
 $$f_{(M,g)}(\| v \|) = \left\{ \begin{array}{ll} 
0 & \mbox{when $\frac{\|v\|^2}{c}$ is an integer} \\
2 & \mbox{when $\frac{\|v\|^2}{c}$ is not an integer}. \\
\end{array}
\right.
$$ 
satisfies Equation~\ref{eqn:ComponentLengthEquation}.
\end{enumerate}
\end{lem}

\begin{proof}
If we let $F = \frac{1}{2} \sum_{j=1}^{n} e_j$, then according to Section~\ref{sec:TypeCn} 
the central lattice is given by $\Lambda_Z = \langle e_1, \ldots, e_{n-1} , F \rangle$
and the sum of the positive roots is 
$\sum_{\alpha \in R+} \alpha = \sum_{j =1}^{n} 2(n-j +1) \epsilon_j$.
Now, let $v = \sum_{j=1}^{n-1} k_j e_j + k_n F \in \Lambda_{I} \cap \overline{\mathcal{C}} \subseteq \Lambda_Z \cap \overline{\mathcal{C}}$,
then we have 
\begin{eqnarray*}
2 \sum_{\alpha \in R^+} \alpha (v) &=& \sum_{j=1}^{n} 4(n-j+1) \epsilon_j(v) \\
&=& \sum_{j=1}^{n-1} 4(n-j+1) k_j + k_n \sum_{j = 1}^{n} 4(n-j+1)\epsilon_{j}(F) \\
&\equiv& 4k_n \sum_{j = 1}^{n} (n-j+1)\epsilon_{j}(F)  \mod 4\\
&\equiv& 4 k_n \sum_{j=1}^{n}(n-j+1)\frac{1}{2} \mod 4 \\
&\equiv& 2 k_n \sum_{j=1}^{n}(n-j+1) \mod 4 \\
&\equiv& 2 k_n \left( \frac{n(n+1)}{2} \right) \\
&\equiv& k_n n(n+1)\\
&\equiv& 
\left\{ \begin{array}{ll}
0 \mod 4, & \mbox{for $k_n$ even} \\
0 \mod 4, & \mbox{for $k_n$ odd and $n\equiv 0,3 \mod 4$}\\
2 \mod 4, & \mbox{for $k_n$ odd and $n\equiv 1,2 \mod 4$}.\\
\end{array}
\right.
\end{eqnarray*}

The calculation above and and Lemma~\ref{lem:SimplyConnectedSplitRank} 
allow us to deduce that for $n\equiv 0,3 \mod 4$ or $\Gamma$ trivial, 
we can take $f_{(M,g)}(\| v \|) \equiv 0$.

Now, for any $v = sum_{j=1}^{n-1} k_j e_j + k_n F \in  \Lambda_Z \cap \overline{\mathcal{C}}$, 
we see that $\|v \|^2/c$ is an integer if and only if $k_n$ is even. This observation plus the calculation 
above establishes the second statement of the lemma. 
\end{proof}

\begin{lem}\label{lem:TheFunctionTypeD_n}
Let $(M,g)$ be a simple group of type $D_n$ $(n \geq 4)$ equipped with a bi-invariant metric $g$.
That is, $M = \Spin(2n) / \Gamma$, where 
$$\Gamma \leq Z(\Sp(n)) \simeq \left\{ \begin{array}{ll}
 \Z_{4} & n \mbox{ even} \\
 \Z_2 \oplus \Z_2 & n \mbox{ odd, }\\
 \end{array} 
 \right.$$
 with $g = c \bar{g}$ for some $c >0$.
Then Equation~\ref{eqn:ComponentLengthEquation} has a solution $f_{(M,g)}$
having the following form.
\begin{enumerate}
\item If $n \equiv 0, 1 \mod 4$, then $f_{(M,g)} \equiv 0$.
%satisfies Equation~\ref{eqn:ComponentLengthEquation}.

\item Let $n \equiv 3 \mod 4$. 
\begin{enumerate}
\item If $\Gamma$ trivial or $\Z_2 \oplus 1$, then $f_{(M,g)} \equiv 0$;
%satisfies Equation~\ref{eqn:ComponentLengthEquation};
\item If $\Gamma$ is $1 \oplus \Z_2$, $\langle (1,1) \rangle \simeq \Z_2$, or $Z(\Spin(2n)) \simeq \Z_2 \oplus \Z_2$, then 
$$f_{(M,g)}(\| v \|) = \left\{ \begin{array}{ll} 
0 & \mbox{when $\frac{\|v\|^2}{c}$ is an integer} \\
2 & \mbox{when $\frac{\|v\|^2}{c}$ is not an integer}. \\
\end{array}
\right.$$
%satisfies Equation~\ref{eqn:ComponentLengthEquation}.
\end{enumerate} 

\item Let $n \equiv 2 \mod 4$.
\begin{enumerate}
\item If $\Gamma$ is trivial or $\Z_2$, then $f_{(M,g)} \equiv 0$;
%satisfies Equation~\ref{eqn:ComponentLengthEquation};
\item If $\Gamma = Z(\Spin(2n)) \simeq \Z_4$, then  
 $$f_{(M,g)}(\| v \|) = \left\{ \begin{array}{ll} 
0 & \mbox{when $\frac{\|v\|^2}{c}$ is an integer} \\
2 & \mbox{when $\frac{\|v\|^2}{c}$ is not an integer}. \\
\end{array}
\right.
$$ 
%satisfies Equation~\ref{eqn:ComponentLengthEquation}.
\end{enumerate}

\end{enumerate}
\end{lem} 

\begin{proof}
As is noted in Section~\ref{sec:TypeDn}, the central lattice is given by 
$\Lambda_Z = \langle e_1, \ldots, e_{n-1} , F \rangle$,
where $F = \frac{1}{2} \sum_{j=1}^{n} e_j$,
and the sum of the positive roots is 
$\sum_{\alpha \in R+} \alpha = \sum_{j =1}^{n} 2(n-j) \epsilon_j = \sum_{j =1}^{n-1} 2(n - j) \epsilon_j$.
Now, let $v = \sum_{j=1}^{n-1} k_j e_j + k_n F \in \Lambda_{I} \cap \overline{\mathcal{C}} \subseteq \Lambda_Z \cap \overline{\mathcal{C}}$,
then we have 
\begin{eqnarray*}
2 \sum_{\alpha \in R^+} \alpha (v) &=& \sum_{j=1}^{n} 4(n-j) \epsilon_j(v) \\
&=& \sum_{j=1}^{n-1} 4 k_j (n-j) + k_n \sum_{j=1}^{n} 4(n-j)\epsilon_j(F) \\
&\equiv& 2 k_n \sum_{j=1}^{n}(n-j) \mod 4 \\
&\equiv& 2 k_n \left( \frac{n(n-1)}{2} \right) \\
&\equiv& k_n n(n-1)\\
&\equiv& 
\left\{ \begin{array}{ll}
0 \mod 4, & \mbox{for $k_n$ even} \\
0 \mod 4, & \mbox{for $k_n$ odd and $n\equiv 0,1 \mod 4$}\\
2 \mod 4, & \mbox{for $k_n$ odd and $n\equiv 2, 3\mod 4$}.\\
\end{array}
\right.
\end{eqnarray*}

Consulting the Equations~\ref{eqn:IntegralLatticeD_nOdd} and \ref{eqn:IntegralLatticeD_nEven}, 
we see that statements (1), (2a) and (3a) of the lemma follow from the computation above.
Now, notice that the parity of $k_n$ is determined by $\|v\|$ and 
$\|v\|^2/c$ is an integer if and only if $k_n$ is even.
Then, statements (2b) and (3b) follow from the computation above and 
Equations~\ref{eqn:IntegralLatticeD_nOdd} and \ref{eqn:IntegralLatticeD_nEven}.
\end{proof}

\begin{lem}\label{lem:TheFunctionTypeF_4}
Let $(M,g)$ be the simple Lie group of Type $F_4$ 
equipped with a bi-invaraint metric $g$, 
then $f_{(M,g)} \equiv 0$
satisfies Equation~\ref{eqn:ComponentLengthEquation}.
\end{lem}

\begin{proof}
Since $M$ is simply-connected the result follows from Lemma~\ref{lem:SimplyConnectedSplitRank}.
\end{proof}

\begin{lem}\label{lem:TheFunctionTypeG_2}
Let $(M,g)$ be the simple Lie group of Type $G_2$ equipped with a bi-invariant metric $g$, 
then $f_{(M,g)} \equiv 0$
satisfies Equation~\ref{eqn:ComponentLengthEquation}.
\end{lem}

\begin{proof}
Since $M$ is simply-connected the result follows from Lemma~\ref{lem:SimplyConnectedSplitRank}.
\end{proof}

\begin{lem}\label{lem:TheFunctionTypeE_8}
Let $(M,g)$ be the simple Lie group of Type $E_8$ 
equipped with a bi-invariant metric $g$, 
then $f_{(M,g)} \equiv 0$ satisfies Equation~\ref{eqn:ComponentLengthEquation}.
\end{lem}

\begin{proof}
Since $M$ is simply-connected the result follows from Lemma~\ref{lem:SimplyConnectedSplitRank}.
Alternatively, noting that $D_8 \subset E_8$ implies that 
 $\Lambda_{Z}^{E_8} \subset \Lambda_{Z}^{D_8}$, this lemma follows from 
 Lemma~\ref{lem:TheFunctionTypeD_n} in the case where $n = 4$. 
\end{proof}

\begin{lem}\label{lem:TheFunctionTypeE_7}
Let $(M,g)$ be a simple Lie group of type $E_7$ equipped with a bi-invariant metric $g$.
That is, letting $\widetilde{E}_7$ denote the 
unique compact simply-connected Lie group of type $E_7$, 
$M= \widetilde{E}_7/\Gamma$, where $\Gamma \leq Z(\widetilde{E}_7) \simeq \Z_2$, 
with $g = c \bar{g}$ for some $c >0$.

\begin{enumerate}
\item If $\Gamma$ is trivial, then $f_{(M,g)} \equiv 0$ satisfies Equation~\ref{eqn:ComponentLengthEquation}.
\item If $\Gamma = Z(\widetilde{E}_7) \simeq \Z_2$, then 
$$f_{(M,g)} = \left\{\begin{array}{ll}
0 & \frac{\|v\|^2}{c} \mbox{ is an integer} \\
2 & \frac{\|v\|^2}{c} \mbox{ is not an integer} \\
\end{array}
\right.
$$
satisfies Equation~\ref{eqn:ComponentLengthEquation}.
\end{enumerate}

\end{lem}

\begin{proof}
The first statement follows from Lemma~\ref{lem:SimplyConnectedSplitRank}.
So, we may assume that $M = \widetilde{E}_7/ Z(\widetilde{E}_7)$, in which case 
the integral lattice $\Lambda_I(M)$ coincides with the central lattice $\Lambda_Z$.

Now, let $v \in \Lambda_{I} \cap \overline{\mathcal{C}} = \Lambda_Z \cap \overline{\mathcal{C}}$, 
then following the notation and conventions of Section~\ref{sec:TypeE7}, 
one can see that $v = \frac{1}{2}\sum_{j=1}^{6} c_j e_j + \frac{1}{2}c_7(e_7 - e_8)$ 
for $c_1, \ldots, c_6, c_7 \in \Z$ satisfying $c_1 \equiv \cdots \equiv c_6 \mod 2$ and $c_7$ arbitrary.
Then we have 
\begin{eqnarray*}
2 \sum_{\alpha \in R^+} \alpha (v) &=& 4 \rho (v) \\
&=& 34 c_7 + 2\sum_{j=1}^{5}(6-j) c_j \\
&\equiv& 2c_7 + 2(5c_1 + 4c_2 + 3c_3 + 2c_4 + c_5) \mod 4 \\
&\equiv& 
\left\{ \begin{array}{ll}
2c_7 \mod 4, & \mbox{for $c_1, \ldots , c_6$ even} \\
2c_7 +2 \mod 4, & \mbox{for $c_1, \ldots , c_6$ odd}.\\
\end{array}
\right.
\end{eqnarray*}
Recalling that $c_1, \ldots , c_6$ have the same parity the previous equation becomes
\begin{eqnarray*}
2 \sum_{\alpha \in R^+} \alpha (v) &=&
\left\{ \begin{array}{ll}
0 \mod 4, & \mbox{for $c_7 \equiv c_1 \mod 2$} \\
2 \mod 4, & \mbox{for $c_7 \not\equiv c_1 \mod 2$}.\\
\end{array}
\right.
\end{eqnarray*}
%The result follows from the following paragraph which establishes that 
%$\|v \|^2/c$ determines whether $c_7$ has the same parity as $c_1, \ldots , c_6$.
However, it is easy to deduce that $\|v \|^2/c = \frac{1}{4} (\sum_{j=1}^{6} c_j^2 + 2c_7^2)$ is an integer 
if and only if $c_7 \equiv c_1 \mod 2$.
And the result follows.

%Since, $\|v \|^2/c = \frac{1}{4} (\sum_{j=1}^{6} c_j^2 + 2c_7^2)$, 
%it follows that in the case that $c_1, \ldots , c_6$ are even we find that 
%$\|v\|^2/c$ is an integer if and only if $c_7$ is even (i.e., $c_7 \equiv c_1 \mod 2$).
%When $c_1, \ldots , c_6$ are all odd it follows that $\| v\|^2/c$ is an integer if and only if $c_7$ is odd 
%(i.e., $c_7 \equiv c_1 \mod 2$).
%That is, $\|v\|^2/c$ is an integer if and only if $c_7 \equiv c_1 \mod 2$.
%The result now follows.
\end{proof}

\begin{lem}\label{lem:TheFunctionTypeE_6}
Let $(M,g)$ be a simple Lie group of Type $E_6$ equipped with a bi-invariant metric $g$.
That is, letting $\widetilde{E}_6$ denote the unique simply-connected compact Lie group of type $E_6$, 
 %that is, abusing notation, letting $E_6$ denote the simply-connected Lie group of type $E_6$,
$M = \widetilde{E}_6/\Gamma$, where $\Gamma \leq Z(\widetilde{E}_6) \simeq \Z_3$, 
with $g = c\bar{g}$ for some $c >0$.
Then $f_{(M,g)} \equiv 0$ satisfies Equation~\ref{eqn:ComponentLengthEquation}.
\end{lem}

\begin{proof}
Following the notation of Section~\ref{sec:TypeE6}, we recall that the central lattice of $E_6$ 
is generated by the vectors $v_1, \ldots, v_6$ and $F = \frac{2}{3}(e_6 + e_7 - e_8)$.
Now, since $v_1, \ldots , v_6 \in \Lambda_{R^{\, \check{}}}$, Lemma~\ref{lem:Rho} tells us that
$2\sum_{\alpha \in R^+} \alpha (v_j) = 4\rho(v_j)  \equiv 0 \mod 4$ for all $j = 1, \ldots , 6$. 
And, since $2\sum_{\alpha \in R^+} \alpha (F) = 32 \equiv 0 \mod 4$,
we find that $2\sum_{\alpha \in R^+} \alpha$ is congruent to zero on the entire central lattice
and the result follows.
%Therefore, we have $$f_{(M,g)} \equiv 0.$$
\end{proof}

This concludes the proof of the proposition.
\end{proof}

%%%%%%%%%%%%%%%%%%%%%%%%%%%%
\subsection{CLU Length Spectra are Generic among Symmetric Spaces}\label{SubSec:CLUResidual}
In this section we prove Lemma~\ref{lem:CLUResidual} which states that a generic compact symmetric space 
has CLU length spectrum.

\begin{proof}[Proof of Lemma~\ref{lem:CLUResidual}] 
Let $\Lambda_I(M) \leq \germ{a} = \germ{a}_0 \times \germ{a}_1 \times \cdots \times \germ{a}_q \subseteq \germ{p}$ denote the integral lattice of $M$
and let $h_A$ denote the flat metric on $M_0$ corresponding 
to the positive definite real symmetric matrix $A \in \mathcal{S}^+(d)$.
If $g$ is the symmetric metric on $M$ induced by the metric 
$h_A \times c_1^{2}g_0^{1} \times \cdots \times c_q^2 g_0^q$ on 
$M_0\times \widetilde{M}_1 \times \cdots \times \widetilde{M}_q$, then for each 
$v = (v_0, v_1, \ldots , v_q) \in \Lambda_I(M)$
the corresponding closed geodesic $\gamma_v: [ 0, 1] \to (M,g)$ 
has length 
$$\tau = (v_0^t A v_0 + \sum_{j=1}^{q} c_j^2 \| v_j\|_{j,0}^2)^{1/2},$$
where $\| \cdot \|_{j,0}$ is the norm associated with $g_0^j$, for each $j = 1, \ldots , q$.
Then for any $v, w \in \Lambda_I(M)$ define the continuous map
$Y_{(v,w)} : \mathscr{R}_{\rm{symm}}(M) \equiv \mathcal{S}^+(d) \times \R_+ \times \cdots \times \R_+ \to \R$
via
$$(A, c_1, \ldots , c_q) \mapsto v_0^t A v_0 -w_0^t A w_0 + \sum_{j=1}^{q} c_j^2 (\|v_j\|_{j,0}^2 - \| w_j \|_{j,0}^2).$$
Now, let $\Delta = \{ (v,w) \in \Lambda_I(M) \times \Lambda_I(M) : \| v_j \|_{j,0} \neq \| w_j \|_{j,0} \mbox{ for some } j \geq 1\}$.
Then it is clear that a metric $g$ has CLU length spectrum 
if and only if $g \in \cap_{\delta \in \Delta} Y_{\delta}^{-1}(\R - \{0\})$.
\end{proof}

%%%%%%%%%%%%%%%%%%%%%%%%%%%
\subsection{Determining the Morse Index through Length and Dimension}\label{SubSec:ThmIntroMod4}
We present the proof of Theorem~\ref{thm:IntroMod4} which establishes that for certain compact symmetric spaces $(M,g)$ for which the non-Euclidean part of its cover is trivial or split-rank, 
the Morse index modulo $4$ of a closed geodesic $\gamma$ of length $\tau$ 
is determined by $\tau$ and $\dim \Fix_\gamma(\Phi_\tau)$.

\begin{proof}[Proof of Theorem~\ref{thm:IntroMod4}]

Let $(M,g)$ be a symmetric space for which the non-Euclidean part of its universal cover is 
split-rank and let $\tau \in \Spec_L(M,g)$.
Now let $\gamma$ be a closed geodesic in $(M,g)$ of length $\tau$ 
with $\gamma'(0) = (v_0, v_1, \ldots , v_\ell)$.
Then by Proposition~\ref{prop:Component Dependence}, the Morse index of $\gamma$ satisfies  
$$\sigma_{\Delta M}(\gamma) \equiv h_{(M,g)}(\|v_0\|, \|v_1\|, \ldots , \|v_\ell \|, \dim \Fix_\gamma(\Phi_\tau)) \mod 4,$$
for some function $h_{(M,g)}(x_0, x_1,\ldots, x_\ell, y) = f_{(M,g)}(x_0, x_1, \ldots, x_\ell) - y + C(M)$.

Now, in the event that the metric $g$ has CLU length spectrum we see that for each $j = 0, 1, \ldots , \ell$
$\| v_j \| = c_j(\tau)$. Hence, we may replace the function $h_{(M,g)}$ with 
$H_{(M,g)} : \Spec_L(M,g) \times \{ 0, 1, \ldots , 2\dim M -1) \} \to \{0,1,2,3\}$.
In the case where $(M,g)$ satisfies any one of $(3)$- $(5)$, an examination of the proof of 
Proposition~\ref{prop:Component Dependence}
(and recalling Lemma~\ref{lem:SimplyConnectedSplitRank}) reveals that the function $f_{(M,g)}$ is identically zero and we find 
$$\sigma_{\Delta M}(\gamma) \equiv H_{(M,g)}( \dim \Fix_\gamma(\Phi_\tau)) \mod 4.$$

Finally, suppose $M = G/K$ is a split-rank symmetric space with the metric induced 
by $-cB \upharpoonright \germ{p} \times \germ{p}$ for some $c>0$, where $B$ is the Killing form on $\germ{g}$.
Given $v \in \germ{p} \equiv T_oM$, we have $\|v \|^2 = -cB(v,v) = -c \Tr(\ad(v) \circ \ad(v))$, 
so that in the case where $v \in \germ{a} \subset \germ{p}$ this becomes:
$$\|v \|^2 = \sum_{\beta \in R^+} \dim \germ{g}^{\beta} (\pi \beta(v)^2) = 
-2\pi \sum_{\beta \in R^{+}} n_{\beta} \beta(x)^2.$$
%where as in Section~\ref{} $n_\beta \equiv \dim \germ{p}^\beta$. 

%\begin{lem}\label{lem:UsefulLemma}
%Let $v = (c_1, \ldots , c_n), w = (d_1, \ldots , d_n) \neq 0 \in \Z^n$ be such that 
%$\sum_{j=1}^{n} c_j^2 = \sum_{i =1}^{n} d_j^2 = \tau^2$, then 
%$\sum_{j=1}^{n} c_j$ and  $\sum_{j=1}^{n} d_j$ are congruent modulo $2$.
%\end{lem}

%\begin{proof}
%To the contrary, lets assume that $\sum_{j=1}^{n} c_j = 2k+1$ is odd and 
%$\sum_{j=1}^{n} d_j = 2m$ is even.
%Then $(2k+1)^2 = (\sum_{j=1}^{n} c_j)^2 = \tau^2 + 2 \sum_{i < j} c_ic_j$
%and $(2m)^2 = (\sum_{j=1}^{n} d_j)^2 = \tau^2 + 2\sum_{i<j}d_id_j$. 
%This implies that $(2k+1)^2 - (2m)^2$ is even, which is a contradiction.
%\end{proof}

Hence, for any $v , w \in \Lambda_{Z}(M) \subset \germ{a}$ we see
$\|v\|^2 = \|w\|^2$ if and only if $\sum_{\beta \in R^+} n_\beta \beta(v)^2 = 
\sum_{\beta \in R^+} n_\beta \beta(w)^2$, which implies by the lemma that 
$\sum_{\beta \in R^+} n_\beta \beta(v) \equiv \sum_{\beta \in R^+} n_\beta \beta(w) \mod 2.$
Now, $M$ is split-rank if and only if $n_\beta = 2 \widehat{n}_\beta \in 2 \Z$ for all $\beta \in R^+$.
Then, for any $v, w \in \Lambda_Z(M) \subset \germ{a}$ we see that 
$\|v\|^2 = \|w\|^2$ if and only if $\sum_{\beta \in R^+} \widehat{n}_\beta \beta(v)^2 = \sum_{\beta \in R^+} \widehat{n}_\beta \beta(w)^2$, which implies that 
$\sum_{\beta \in R^+} \widehat{n}_\beta \beta(v) = \sum_{\beta \in R^+} \widehat{n}_\beta \beta(w) \mod 2.$ Therefore, we conclude that $\|v \| = \|w \|$ implies that 
$\sum_{\beta \in R^{+}} n_\beta \beta(v) \equiv \sum_{\beta \in R^{+}} n_\beta \beta(w) \mod 4$.
That is, there is a function $f_{(M,g)}$ such that for any 
$v \in \overline{\mathcal{C}} \cap \Lambda_I(M) \subset \overline{\mathcal{C}} \cap \Lambda_Z(M)$ we have
$$\sum_{\beta \in R^{+}} n_\beta \beta(v) \equiv f_{(M,g)}(\tau) \mod 4,$$
where $\tau = \|v\|$ is the length of the geodesic $\gamma$. 
\end{proof}

The following examples show that Theorem~\ref{thm:IntroMod4} fails 
for arbitrary symmetric spaces.

\begin{exa}[A semi-simple Lie group for which Theorem~\ref{thm:IntroMod4} fails]\label{exa:Mod4Counterexample1}
%$\SU(2)$ and $\SO(3) = \SU(2) / Z(\SU(2))$ are the two simple Lie groups of type $A_1$.
%Let $\bar{g}_1$ (respectively, $\bar{g}_2$) denote the bi-invariant metric 
For any simple Lie group $H$ with Lie algebra $\germ{h}$,
let $\bar{g}^H$ denote the bi-invariant metric induced by $-r B$, 
where $B$ is the Killing form on $\germ{h}$
and $r>0$ is chosen to agree with the inner product structure used in Appendix~\ref{SubSec:Classification}.
Now, consider the semi-simple Lie group $U = \SU(2) \times \SO(3)$ equipped with the 
bi-invariant metric $g = \frac{1}{4}\bar{g}^{\SU(2)} \times \bar{g}^{\SO(3)}$. 
Since $\SU(2)$ is type $A_1$, 
we can deduce from Equation~\ref{eqn:IntegralLatticeA_n} that the integral lattice 
of $U = \SU(2) \times \SU(2)/ 1 \times Z(\SU(2))$ is given by 
$$\Lambda_I(U) = \Lambda_I(\SU(2)) \oplus \Lambda_I(\SO(3)) = \langle (2L_1, 0), (0, L_1) \rangle.$$
The vectors $v = (2L_1, 2L_1), w = (4L_1, L_1) \in \Lambda_I(U)$ belong to the same Weyl chamber $\mathcal{C}$
and are both of length $\tau \equiv \sqrt{5}$ with respect to the metric $g$.
%In fact, these are the only vectors of length $\tau = \sqrt{5}$ in $\Lambda_I(U) \cap \mathcal{C}$.
As they are regular vectors in $\germ{u}$ we see that 
$\dim \Fix_{v}(\Phi_{\sqrt{5}}) = \dim \Fix_{w}(\Phi_{\sqrt{5}}) \equiv D$ is of maximal dimension.
And, since $v$ and $w$ are the only vectors of length $\sqrt{5}$ in $\Lambda_I(U) \cap \overline{\mathcal{C}}$,
Equation~\ref{eqn:WaveInvariants1} yields 
$$\Wave_0^{\bullet}(\sqrt{5}) = \left( \frac{1}{2\pi i} \right)^{\frac{D-1}{2}}
\left( i^{-\sigma_v} \int_{\Fix_v(\Phi_{\sqrt{5}})} d \mu_v + i^{-\sigma_w} \int_{\Fix_w(\Phi_{\sqrt{5}})} d\mu_w \right),$$
where $\sigma_v$ (resp. $\sigma_w$) is the Morse index of the geodesic $\gamma_v$ (resp. $\gamma_w$)
and $\mu_v$ (resp. $\mu_w$) is the canonical Duistermaat-Guillemin measure on 
$\Fix_v(\Phi_\tau)$ (resp. $\Fix_w(\Phi_\tau)$).
Now, using Proposition~\ref{prop:Component Dependence} and Lemma~\ref{lem:TheFunctionTypeA_n}, 
we find that $\sigma_v$ fulfills 
\begin{eqnarray*}
\sigma_v &\equiv& f_{(\SU(2), \frac{1}{4} \bar{g}^{\SU(2)})}(\|2L_1\|) + f_{(\SO(3), \bar{g}^{\SO(3)})}(\|2 L_1 \|) + D \mod 4\\
&\equiv& 0 + 0 + D \mod 4 \\
&\equiv& D \mod 4.
\end{eqnarray*}
On the other hand, we see that $\sigma_w$ satisfies
\begin{eqnarray*}
\sigma_w &\equiv& f_{(\SU(2), \frac{1}{4} \bar{g}^{\SU(2)})}(\|4L_1\|) + f_{(\SO(3), \bar{g}^{\SO(3)})}(\| L_1 \|) + D \mod 4\\
&\equiv& 0 + 2 + D \mod 4 \\
&\equiv& D +2 \mod 4.
\end{eqnarray*}

\noindent
%Therefore, $\sigma_v \not\equiv \sigma_w \mod 4$ and 
Therefore, $\sigma_v - \sigma_w \equiv 2 \mod 4$ and
we conclude that Theorem~\ref{thm:IntroMod4} does not hold.
We also see that in the event that $\Fix_{v}(\Phi_{\sqrt{5}})$ and $\Fix_w(\Phi_{\sqrt{5}})$ have the same 
volume with respect to their Duistermaat-Guillemin measures, the wave invariant $\Wave_0^{\bullet}(\sqrt{5})$ is zero.
Clearly, similar examples can be constructed for certain reducible homogeneity types 
$M = \Gamma \backslash (M_0 \times \widetilde{M}_1 \times \cdots \times \widetilde{M}_q) \not\in \mathscr{H}$, 
%where $\widetilde{M}_{\rm{cpt}}$ is split-rank and  $ M \not\in \mathscr{H}$, 
where $\mathscr{H}$ is as defined in Theorem~\ref{thm:MainResult}. 
%there is a factor $\widetilde{M}_j$ not satisfying condition (3) of Theorem~\ref{thm:IntroMod4}.
\end{exa}

%\begin{rem}\label{rem:IntroMod4}
%{\bf Not exactly what you mean to say!!}
%It is clear that if $M = \Gamma \backslash (M_0 \times \widetilde{M}_1 \times \cdots \times \widetilde{M}_q)$
%is the homogeneity type of a non-reducible symmetric space with $\widetilde{M}_{\rm{cpt}}$ split-rank 
%and containing a $\widetilde{M}_j$ not satisfying conditions 2(a) and 2(b) of Theorem~\ref{}, 
%then $M$ admits 
%From the proof, we see that working modulo $4$, the Morse index of a closed geodesic $\gamma$ in any irreducible 
%split-rank symmetric space is $0$ or $2$, where this quantity is determined by $\tau$, the length of $\gamma$,
%and $\dim \Fix_{\gamma}(\Phi_\tau)$. In fact, the value $2$ can only occur for simple Lie groups of type 
%$A_{2n}$, $B_{n}$, $C_{4n+1}$, $C_{4n+2}$, $D_{4n+2}$, $D_{4n+3}$ or $E_7$.
%From this it follows that if $M$ is any reducible split-rank symmetric space whose universal cover has 
%at least one irreducible factor of the type in the preceding sentence, then $M$ admits symmetric metrics for which 
%Theorem~\ref{thm:IntroMod4} fails. 
%\end{rem}

\begin{exa}[A family of irreducible symmetric spaces for which Theorem~\ref{thm:IntroMod4} fails]\label{exa:Mod4Counterexample2}
A symmetric space $M = G/K$ is said to be of \emph{maximal rank} if 
$\operatorname{rank}(M) = \operatorname{rank}(G)$ or, equivalently, 
its restricted roots occur with multiplicity one \cite[Proposition VI.4.1]{Loos}.
Now, for any compact semi-simple Lie group $G$ with maximal torus $T$ there is 
an involution $\sigma \in \operatorname{Aut}(G)$ such that 
for any $x \in T$ we have $\sigma(x) = x^{-1}$ \cite[Theorem V.4.2]{Loos}
and the associated \emph{space of symmetric elements} 
$G_{\sigma} \equiv \{ x \sigma(x)^{-1} : x \in G \} \subset G$
is a symmetric space of maximal rank with restricted root system isomorphic 
to that of $G$; however, the roots occur with multiplicity \emph{one} instead of two.

Now, let $\sigma \in \operatorname{Aut}(\SO(2n+1))$, $n \geq 2$, be an automorphism of the type discussed 
in the previous paragraph and let $\SO(2n+1)_\sigma$ denote the corresponding space of symmetric elements.
The restricted root system $R$ of $\SO(2n+1)_\sigma$ is of type $B_n$ and the integral lattice
is given by $\Lambda_I = \Lambda_Z = \langle e_1, \ldots , e_n \rangle$.
Consider the vectors $v = (7, 6, 0, \ldots , 0), w = (9, 2, 0, \ldots , 0) \in \overline{\mathcal{C}} \cap \Lambda_I$ 
of length $\tau = \sqrt{85}$ and let $\gamma_v$ and $\gamma_w$ denote the corresponding closed geodesics of 
length $\tau$ in $\SO(2n+1)_\sigma$.
Then, it is clear that $\{\alpha \in R^{+} : \alpha (v) = 0\} = \{ \alpha \in R^{+}: \alpha(w) = 0\}
= \{ \epsilon_j : 3 \leq j \leq n \} \cup \{\epsilon_\mu \pm \epsilon_\nu : 3\leq \mu < \nu \leq n\}$, which is a set 
of order $(n-2)^2$, so that $\operatorname{deg}_{\operatorname{sing}}(v) = \operatorname{deg}_{\operatorname{sing}}(w) = (n-2)^2$ or, equivalently,
$$\dim \Fix_{\gamma_v}(\Phi_\tau) = \dim \Fix_{\gamma_w}(\Phi_\tau) \equiv D,$$
where, as before, for any closed geodesic $\gamma$ of length $\tau$, $\Fix_{\gamma}(\Phi_\tau)$ 
denotes the component of $\Fix(\Phi_\tau)$ containing $\gamma'(t)/ \| \gamma'(t) \|$.
Using the formula for the sum of the positive roots provided in Section~\ref{sec:TypeBn} 
we find that  
$$\sum_{\alpha \in R^{+}} \alpha(v) \equiv \left\{ \begin{array}{ll}
3 \mod 4 & \mbox{for $n$ even} \\
1 \mod 4 & \mbox{for $n$ odd}
\end{array}
\right.
$$

and 
$$\sum_{\alpha \in R^{+}} \alpha(w) \equiv \left\{ \begin{array}{ll}
1 \mod 4 & \mbox{for $n$ even} \\
3 \mod 4 & \mbox{for $n$ odd}
\end{array}
\right.
$$
It then follows from Equation~\ref{eqn:MorseIndex} that 
the Morse indexes of $\gamma_v$ and $\gamma_w$ are not congruent modulo $4$ 
(but they do have the same parity).
Therefore, it is not the case that the Morse index modulo $4$ of a closed geodesic $\gamma$ 
in $\SO(2n+1)_\sigma$ is a function of its length $\tau$ and the dimension of the corresponding 
component of $\Fix(\Phi_\tau)$.
\end{exa}

%\begin{rem}
%In the example above if one takes $v = (8,1, 0, \ldots , 0), w = (7,4, 0, \ldots , 0) \in \overline{\mathcal{C}} 
%\cap \Lambda_{I}$ of length $\tau = \sqrt{65}$, everything is the same with the exception that 
%$\sum_{\alpha \in R^{+}} \alpha (v) \equiv \sum_{\alpha \in R^{+}} \alpha (w) \equiv 2n+1 \mod 4$. 
%\end{rem}

\begin{rem}
It would appear to be of interest to explore whether cancellations can occur in the trace formula for the examples presented above and other symmetric spaces for which Theorem~\ref{thm:IntroMod4} fails.
This will be taken up in a subsequent article.
\end{rem}

%\begin{rem}
%In light of the introduction to this article, it would appear to be of interest to examine 
%whether the trace of the wave group of $\SO(2n+1)_\sigma$ is smooth at the length $\tau = \sqrt{85}$,
%for instance.
%\end{rem}

%%%%%%%%%%%%%%%%%%%%%%
%%%%%%%%%%%%%%%%%%%%%%
%%%%%%%%%%%%%%%%%%%%%%
\section{Unclean Left-Invariant Metrics on $\SO(3)$ and $S^3$}\label{Sec:UncleanMetrics}

Theorem~\ref{thm:SymmSpcsCIH} states that every compact globally symmetric space is clean. In this section we provide a proof of Theorem~\ref{thm:UncleanMetrics} which states that there are homogeneous metrics that fail to be clean. Indeed, within the class of naturally reductive left-invariant metrics on $\SO(3)$, we give an explicit description of the metrics that are clean and unclean. We find that the clean metrics contain a residual set, while the unclean metrics form a dense subset. 
%we show that within the set of left-invariant naturally reductive metrics on $\SO(3)$ the unclean metrics are dense. 
Therefore, even in the setting of homogeneous Riemannian metrics, a new technique for analyzing the singularities of the wave trace is needed. The unclean metrics we observe satisfy condition (1) of Definition~\ref{dfn:Clean}, but possess periods $\pm \tau$ for which condition (2) is not met. We are able to verify condition (1) by explicitly computing the closed geodesics for these metrics. We are not aware of whether condition (1) is always satisfied for homogeneous metrics. Although we discuss left-invariant naturally reductive metrics on $\SO(3)$, similar statements and arguments apply to the left-invariant naturally reductive metrics on $S^3 \simeq \SU(2)$; i.e., the Berger spheres. 
%We begin with a review of the notion of naturally reductive metrics. %$\germ{l}$ $\germ{d}$ $\germ{w}$ $\germ{m}$

%In this appendix we demonstrate that there are homogeneous metrics that fail to be clean.

%\begin{thm}\label{thm:UncleanMetrics}
%Within the class of left-invariant naturally reductive metrics on $\SO(3)$, the clean metrics contain a residual set; however, the collection of unclean metrics is non-empty and contains certain normal homogeneous metrics.
%\end{thm}

%\noindent
%Therefore, even in the setting of homogeneous Riemannian metrics, a new technique for analyzing the singularities of the wave trace is needed.

%%%%%%%%%%%%%%%%%%%%%%%%%%%%%%%%%%%%%%
\subsection{Classification of Naturally Reductive Metrics on Lie Groups}

Let $(M,g)$ be a connected homogeneous Riemannian manifold.   Choose a base
point $p_0 \in M$.  Let $G$ be a transitive group of isometries of $(M,g)$, and let
$K$ be the isotropy group of $p_0$.  Now, suppose the Lie algebra $\germ{g}$ of $G$  decomposes into a direct sum $\germ{g}=\germ{K}+\germ{p}$, where $\germ{K}$ is the Lie algebra of $K$ and $\germ{p}$ is an $\Ad(K)$-invariant complement of $\germ{K}$.  Given a vector $X \in\germ{g}$ we obtain a Killing field $X^*$ on $M$ by $X_{p}^*  \equiv \frac{d}{dt}|_{t=0} \exp_{H}{tX}\cdot p$ for $p\in M$. The map $X \mapsto X^*$ is an antihomomorphism of Lie algebras. We may identify $\germ{p}$ with $T_{p_0}M$ by the linear map $X \mapsto X_{p_0}^*$.   Thus, the
homogeneous Riemannian metric $g$ on $M$  corresponds to an inner product
$\langle \cdot , \cdot \rangle$ on $\germ{p}$.   For $X\in\mg$, write $X=X_\germ{K}+ X_\germ{p} $ with $X_\germ{K}\in \germ{K}$ and $X_\germ{p}\in \germ{p}$.   Recall that for $X,Y\in \germ{p}$,
\begin{eqnarray}\label{eqn:NatRed} 
(\nabla_{X^*}Y^*)_{p_0} = -\frac{1}{2}([X,Y]_{\germ{p}}^*)_{p_0} + W(X,Y)^*_{p_0},
\end{eqnarray}
where $W: \germ{p} \times \germ{p} \to \germ{p}$ is the symmetric bilinear map defined by 
$$2 \langle W(X,Y), Z \rangle = \langle [Z, X]_{\germ{p}}, Y \rangle + \langle X, [ Z,
Y]_{\germ{p}} \rangle.$$

\begin{dfn}\label{def:NatRed}
Let $(M,g)$ be a Riemannian homogeneous space and let $G$ be a transitive group
of isometries of $(M,g)$, so that $M = G/K$.
\begin{enumerate}
\item $(M,g)$ is said to be \emph{reductive} (with respect to $G$), if there is an $\Ad(K)$-invariant complement 
$\germ{p}$ of $\germ{K}$ in $\germ{g}$.

\item $(M,g)$  is said to be \emph{naturally reductive} (with respect to $G$) or \emph{$G$-naturally reductive}, if there exists an $\Ad(K)$-invariant complement $\germ{p}$ of $\germ{K}$ (as above)
such that 
$$\langle [Z, X]_\germ{p}, Y \rangle + \langle X, [Z,Y]_\germ{p} \rangle =0,$$
or equivalently $W \equiv 0$. That is, for any $Z \in \germ{p}$  the map $[Z, \cdot ]_{\germ{p}} : \germ{p} \to \germ{p}$ is skew symmetric with respect to $\langle \cdot , \cdot \rangle$.

\item $(M,g)$ is said to be \emph{normal homogeneous}, if there is an $\Ad(G)$-invariant inner product $Q$ on $\germ{g}$ such that $$Q(\germ{p}, \germ{K})= 0 \mbox{ and } Q \upharpoonright \germ{p} = \langle \cdot , \cdot \rangle.$$
\end{enumerate}
\end{dfn}

%At our preferred point $p_0$, the Levi-Civita connection $\nabla$ of a naturally reductive space $(M,g)$ is given by
%$$(\nabla_{v} X^*)(e) = \left\{\begin{array}{ll} [X,v] & \mbox{ if } X \in \germ{K} \\ \frac{1}{2}[X,v]_{\germ{p}} & \mbox{ if } X \in \germ{p}. \end{array} \right. $$ 
%In subsequent sections of this article, it will be useful to recall that there is a metric connection $\widetilde{\nabla}$ whose geodesics coincide with those of $\nabla$, but whose torsion tensor $T^{\widetilde{\nabla}}$ and curvature tensor $R^{\widetilde{\nabla}}$ are both $\widetilde{\nabla}$-parallel. The relationship between $\nabla$ and $\widetilde{\nabla}$ is given by $\nabla_{X}Y = \widetilde{\nabla}_{X}Y - \frac{1}{2}T^{\widetilde{\nabla}}(X,Y)$. At $p_0$, we notice that $\widetilde{\nabla}$, $T^{\widetilde{\nabla}}$ and $R^{\widetilde{\nabla}}$ can be expressed in terms of the Lie bracket. Indeed, for $v \in T_{p_{0}}M \equiv \germ{p}$ and $X \in \germ{h}$ we have 
%$$(\widetilde{\nabla}_{v} X^*)(e) = \left\{\begin{array}{ll} [X,v] & \mbox{ if } X \in \germ{K}\\ \mbox{}[X,v]_{\germ{p}} & \mbox{ if } X \in \germ{p} \end{array} \right. $$ 
%and for $X, Y, Z \in \germ{p}$ we have $T^{\widetilde{\nabla}}(X,Y) = -[X,Y]_{\germ{p}}$ and $R^{\widetilde{\nabla}}(X,Y)Z = -[[X,Y]_{\germ{K}}, Z]$.

\begin{rem}\label{rem:NatRed}
If $G$ is a connected group of isometries acting transitively on $(M,g)$, then $(M,g)$ is reductive with respect to $G$ \cite{KS}. %This is essentially a consequence of the fact that for any Riemannian manifold $(M,g)$ (not necessarily homogeneous) and $p \in M$ the subgroup of the full isometry group of $(M,g)$ fixing $p$ is compact in the compact open topology \cite[Theorem IV.2.5]{Helgason}. The reader can consult \cite{KS} for a complete proof.
\end{rem}

In \cite{DZ}, D'Atri and Ziller addressed the problem of classifying the naturally reductive left-invariant metrics on compact Lie groups. Recalling that for any subgroup $K$ of the Lie group $U$ the natural action of $G \equiv U\times K$ on $U$ is defined by $(g, k) \cdot x = g x k^{-1}$, D'Atri and Ziller's classification of such metrics is as follows.
%\end{rem}

\begin{thm}[\cite{DZ} Theorems 3 and 7]\label{thm:NatRed}
Consider a connected compact simple Lie group $U$ and let $g_0$ be the bi-invariant Riemannian metric on $U$ induced by the negative of the Killing form $B$. Let $K \leq U$ be a connected
subgroup with Lie algebra $\germ{K} = \germ{K}_0 \oplus \germ{K}_1 \oplus \cdots \oplus \germ{K}_r$,  where $\germ{K}_0 = Z(\germ{K})$ is the center of $\germ{K}$ and $\germ{K}_1, \ldots , \germ{K}_r$ are the simple ideals in $\germ{K}$.
Let $\germ{m}$ be a 
$g_{0}$-orthogonal complement of $\germ{K}$ in $\germ{u}$.  Given any 
$\alpha, \alpha_1, \ldots, \alpha_r > 0$ and an arbitrary inner product $h$ on $\germ{K}_0$,
let $g_{\alpha, \alpha_1, \ldots, \alpha_r, h}$ denote the metric on $U$ induced by the $\Ad(K)$-invariant inner product on $\germ{u}$ given by 
\begin{eqnarray}\label{Eq:NatRedMetric}
\alpha g_0 \upharpoonright \germ{m} \oplus h \upharpoonright \germ{K}_0 \oplus 
\alpha_1 g_0 \upharpoonright \germ{K}_1 \oplus \cdots \oplus \alpha_r g_0 \upharpoonright \germ{K}_r.
\end{eqnarray}
%induces a left-invariant metric $g_{\alpha,\alpha_1, \ldots, \alpha_r, h}$ on $U$. 
%we may define a left-invariant metric $g_{\alpha, \phi}$ on $G$ by the orthogonal direct sum
%\begin{eqnarray}\label{Eq:NatRed}
%g_{\alpha, \phi} &=& \alpha g_{0} \upharpoonright \germ{u} \oplus \phi \upharpoonright
%\germ{K}.\end{eqnarray}
 Then:
\begin{enumerate}
\item $g_{\alpha, \alpha_1, \ldots , \alpha_r, h}$ is naturally reductive with respect to 
the natural action of $G \equiv U\times K$ on $U$; 
%, where $G$ acts by left translations and $K$ by right translations on $G$.

\item every left-invariant naturally reductive metric on $U$ arises in this fashion;

\item $g_{\alpha, \alpha_1, \ldots , \alpha_r, h}$ is normal homogeneous if and only if 
$h \leq \alpha g_0 \upharpoonright \germ{K}$.

\item $\Isom(g_{\alpha, \alpha_1, \ldots , \alpha_r, h})^0$, the connected isometry group, 
is given by $U \times N_U(K)^0$, where $N_U(K)$ denotes the normalizer of $K$ in $U$.

\end{enumerate}
 \end{thm}

Naturally reductive spaces generalize the notion of a symmetric space and, as is the case with symmetric spaces, the geodesics in a naturally reductive space $(M = G/K, g)$ are of the form $\Exp_G(tX)\cdot p$, where $X \in \germ{p}$. That is, the geodesics in a naturally geodesic space are precisely the integral curves of Killing vector fields. %In the case of left-invariant naturally reductive metrics we obtain the following characterization of the closed geodesics.

\begin{prop}\label{prop:NatRedGeodesics}
Let $U$ be a simple Lie group and $K$ a connected subgroup. Now, let $g_{\alpha, \alpha_1, \ldots , \alpha_r, h}$ be a $U \times K$ naturally reductive metric on $U$ and $\germ{p} \leq \germ{u} \times \germ{k}$ the $\Ad(\Delta K)$-invariant complement constructed in Appendix~\ref{sec:AdK}. 
%\ref{rem:QOrthogonal}.
Then, the geodesics through $g \in U$ with respect to $g_{\alpha, \alpha_1, \ldots , \alpha_r, h}$ are of the form 
$$\exp_{U \times K}(t\Ad(g)X, Y) \cdot g = g \exp_{U}(tX)\exp_{U}(-tY),$$
where $(X,Y) \in \germ{p}$, and such a geodesic is smoothly closed if and only if 
$\exp_{U}(tX) = \exp_{U}(tY)$ for some $t >0$.
\end{prop}

\noindent
This will prove useful in our proof of Theorem~\ref{thm:UncleanMetrics}, where we will need an explicit description of the closed geodesics with respect to a left-invariant naturally reductive metric on $\SO(3)$.

%\begin{rem}\label{rem:NoLassos}
%Since, the geodesics in a naturally reductive space are integral curves of Killing fields, there are no geodesic lassos in a naturally reductive space (i.e., all self-intersections of a geodesic are smooth). Although it is not needed elsewhere in the paper, we note that since every homogeneous Riemannian space is reductive \cite[Proposition 1]{KS}, an application of Noether's theorem (cf. \cite[Theorem 1.3]{Tak}) shows there are no geodesic lassos in an arbitrary homogeneous space.
%\end{rem}

%Since, the geodesics in a naturally reductive space are integral curves of Killing fields, we see there are no geodesic lassos in a naturally reductive space (i.e., all self-intersections of a geodesic are smooth). Although it is not needed elsewhere in the paper, we observe that every homogeneous Riemannian manifold has this property.

Since, the geodesics in a naturally reductive space are integral curves of Killing fields, there are no geodesic lassos in a naturally reductive space (i.e., all self-intersections of a geodesic are smooth). Although it is not needed elsewhere in the paper, we note that since every homogeneous Riemannian space can be expressed as a reductive space \cite[Proposition 1]{KS}, an application of Noether's theorem (cf. \cite[Theorem 1.3]{Tak}) shows there are no geodesic lassos in an arbitrary homogeneous space.

\begin{prop}\label{prop:NoLassos}
Let $(M, g)$ be a homogeneous Riemannian manifold and $\gamma: \R \to M$ a geodesic. If $\gamma(t_0) = \gamma(t_1)$, then $\gamma'(t_0) = \gamma'(t_1)$. That is, any self-intersection of a geodesic in a homogeneous space is smooth.
\end{prop}

\subsection{The Poincar\'{e} map of naturally reductive metrics}\label{sec:PoincareMap}

We recall that given a Riemannian manifold $(M,g)$ the geodesic flow is the map
$\Phi : \R \times TM \to TM$ given by 
$$\Phi(t, v) = \frac{d}{dt}\gamma_{v}(t),$$
where $\gamma_v$ is the unique geodesic with $\gamma_v'(0) = v$.
Throughout we will set $\Phi_t(v) = \Phi(t,v)$.
Of particular interest to us is the derivative of $\Phi_\tau$.
If for each $v \in TM$ we let $T_v TM = \mathcal{H}_v \oplus \mathcal{V}_v$ be the decomposition into the 
horizontal and vertical spaces, then for any $(A,B) \in T_v TM$ we have
$$\Phi_{t*}(A,B) = (Y(t), \nabla Y(t)),$$
where $Y(t)$ is the Jacobi field along $\gamma_v $ such that $Y(0) = A$ and 
$\nabla Y(0) = B$ (see \cite[p. 56]{Sakai}).
If the geodesic $\gamma_v$ is periodic of period $\tau$, then we set 
$$P = \Phi_{\tau*} : T_{v}TM \to T_{v}TM.$$
Since $\gamma_{v}'(t)$ and $t\gamma_{v}'(t)$ are Jacobi fields along $\gamma_v$ we see 
that 
$$P(v,0) = (v,0) \mbox{ and } P(0,v) = (\tau v, v).$$
Hence, in order to understand $P$ we must analyze how it behaves on the orthogonal complement 
of $(v,0)$ and $(0,v)$; that is, we seek to understand 
$$P: E \oplus E \to E \oplus E,$$
where $E = \{ u \in T_p M : \langle u , v \rangle = 0\}$.
This map is called the (linearized) \emph{Poincar\'{e} map} and from the above if $Y$ is a Jacobi field with 
initial data $(Y(0), \nabla Y(0)) \in E \oplus E$, then 
$$P(Y(0), \nabla Y(0) ) = (Y(\tau), \nabla Y(\tau)).$$
In the case of (compact) naturally reductive manifolds the Poincar\'{e} map 
has been completely determined by Ziller as follows.

Let $M = G/K$ be a naturally reductive space and as before let $\germ{p} \leq \germ{g}$ be an $\Ad(K)$-invariant complement. For any unit vector $v \in  \germ{p} \equiv T_{p_0}M$ we let $\gamma_v (t)$ be the unit speed geodesic given by $\exp_{G}(tv) \cdot p_0$. Now, let $v \in \germ{p}$ be a unit vector such that the geodesic $\gamma_v(t)$ is closed and set $E = \{ u \in \germ{p} : \langle u , v \rangle =0\}$. Then the restriction of the maps $B(\cdot ) = -[v, [v, \cdot]_\germ{K}]$ and $T(\cdot ) = -[v, \cdot]_\germ{p}$ to $E$ are symmetric and skew-symmetric, respectively.
Now let $E_0$ denote the $0$-eigenspace of $B: E \to E$ and $E_1$ be the sum of 
its non-zero eigenspaces, and we express $E_0$ as the orthogonal direct sum 
$E_0 = E_2 \oplus E_3$, where $E_2 = \{ X \in E_0 : T(X) \in E_1 \}$. 
Then as in \cite[p. 579]{Ziller2} we define the following subspaces of $E \oplus E$:

\begin{enumerate}
\item $V_1 = \{ (X, \frac{1}{2} [ X, v]_\germ{p}) : X \in E_1 \oplus E_3 \}$
\item $V_2 = \{ (0 , X) : X \in E_1 \}$
\item $V_3 = \{ (X, \frac{1}{2} [v, X]_\germ{p}) : X \in E_2 \}$
\item $V_4 = \{ (X, \frac{1}{2}[v,X]_\germ{p}) : X \in E_3 \} = \{ (X,- \frac{1}{2}T(X)) : X \in E_3 \}$
\item $V_5 = \{ (Z, X + \frac{1}{2}[v,Z]_\germ{p}) : X \in E_2, Z \in E_1 \mbox{ and } B(Z) = T(X) \equiv [X, v]_\germ{p}  \}$
\end{enumerate}

\begin{rem}\label{rem:Eigenspaces}\text{}\\
\begin{enumerate}
\item In \cite{Ziller2} there is an omission in the definition of $V_5$ (cf. \cite[p. 73]{Ziller1}).
\item We note that since $B: E_1 \to E_1$ is an isomorphism, $V_5$ is non-trivial if and only if $E_2$ is non-trivial. In particular, for each $X \in E_2$, there exists a unique $Z \in E_1$ such that $B(Z) = T(X)$.
\item It will be useful later to notice that $E_1 \leq [\germ{K}, v]$. Indeed, following \cite[p. 72]{Ziller1}, we recall that $B: E \to E$ is a self-adjoint map. Let $X_1, \ldots , X_q$ be an orthonormal basis of eigenvectors 
with eigenvalues $\lambda_1, \ldots , \lambda_q$, and set $Z_i \equiv [v, X_i]_\germ{K} \in \germ{K}$.
Then $\lambda_i X_i = B(X_i) = [Z_i, v]$ and for $\lambda_i \neq 0$ we get $X_i = \frac{1}{\lambda_i}[Z_i, v] \in [\germ{K}, v]$, which establishes the claim.
\end{enumerate}
\end{rem}

With the notation as above we have the following theorem due to Ziller. 

\begin{thm}\label{thm:PoincareMap}
Let $(M = G/K, g)$ be a (compact) naturally reductive space and let $\gamma_v (t) = \exp_{G}(tv) \cdot p_0$ be a smoothly closed unit speed geodesic in $M$ of length $\tau$ with $\gamma_{v}'(0) = v \in \germ{p} \equiv T_{p_0} M$. Then 

\begin{enumerate}
\item (\cite[Theorem 1]{Ziller2}) $E \oplus E = V_1 \oplus V_2 \oplus V_3 \oplus V_4 \oplus V_5$

\item (\cite[Theorem 1]{Ziller1}) The Poincar\'{e} map $P : E \oplus E \to E \oplus E$ along $\gamma_v$ is described as follows:

\begin{enumerate}
\item $P \upharpoonright V_1 \oplus V_2 \oplus V_3 = \operatorname{Id}$;
\item $P (X, \frac{1}{2}[v,X]_\germ{p})= (\Psi(X) , \Psi( \frac{1}{2}[v,X]_\germ{p} )) = 
(\Psi(X) , \frac{1}{2}[v,\Psi(X)]_\germ{p} )$, for $(X, \frac{1}{2}[v,X]_\germ{p}) \in V_4$,
where $\Psi$ is the isometry $e^{\ad(\tau v)} = \Ad (\exp_{H}(\tau v))$, 
we recall that because $\gamma_v$ is a geodesic it is given by $\exp_{H}(t v) 
\cdot p_0$ and since it is closed of length $\tau$ we have that $\exp_{H}(\tau v) \in K$;
\item $P (Z, X + \frac{1}{2}[v,Z]_\germ{p}) = \tau (X, \frac{1}{2}[v,X]_\germ{p}) + 
(Z, \frac{1}{2}[v,Z]_\germ{p})$, for $ (Z, X + \frac{1}{2}[v,Z]_\germ{p}) \in V_5$.
\end{enumerate}

\end{enumerate}

\end{thm}

\begin{rem}
The compactness condition in the above was used by Ziller to establish that a Jacobi filed $J(t)$ along $\gamma_v$ with $J(0) \in V_5$ 
must have unbounded length, which is used to show that $V_5 \cap (V_1\oplus V_2\oplus V_3 \oplus V_4)$ is trivial \cite[p. 579-80]{Ziller2}. 
However, this argument only really requires completeness, which is enjoyed by all naturally reductive spaces 
since geodesics are precisely the orbits of one-parameter groups of isometries. Therefore, the above is true for all naturally reductive 
manifolds.
\end{rem}

The following observation is an immediate consequence of the previous proposition.

\begin{cor}\label{cor:PeriodicFields}
Let $\gamma_v(t)$ be a closed unit speed geodesic as above and 
let $Y(t)$ be a Jacobi field along $\gamma_v$.
Then $Y(t)$ is periodic if and only if 
$Y(t)$ has the following initial conditions:
$$(Y(0), \nabla Y(0) ) \in V_1 \oplus V_2 \oplus V_3 \oplus V_4^{\rm{per}} \oplus \Span_{\R}\{(v,0)\}, $$
%$$(Y(0), \nabla Y(0) ) \in V_1 \oplus V_2 \oplus V_3 \oplus V_5', $$
where $V_4^{\rm{per}} \equiv \{ (X, \frac{1}{2}[v,X]_\germ{p}) : X \in E_3 \mbox{ and } \psi(X) = X \} \leq V_4$.
%$V_5' = \{ (X, \frac{1}{2}[v,X]_\germ{p}) : X \in E_3 \mbox{ and } \psi(X) = X \}$.
\end{cor}

%%%%%%%%%%%%%%%%%%%%%%%%%%%%%%
%%%%%%%%%%%%%%%%%%%%%%%%%%%%%%
\subsection{Proof of Theorem~\ref{thm:UncleanMetrics}}\label{Sec:ProofUnclean}

Let $U$ be an arbitrary compact semi-simple Lie group with 
bi-invariant metric $g_0$ induced by the negative of the Killing form $B$ on $T_eU$.
Now for any left-invariant metric $g$ on $U$ there is a linear transformation 
$\Omega : T_e U \to T_eU$ that is self-adjoint with respect to $-B$ and such that 
for any $v, w \in T_e U$ we have $\langle v, w \rangle = -B( \Omega(v), w)$, 
where $\langle \cdot , \cdot \rangle$ is the restriction of $g$ to $T_e U$.

\begin{dfn}\label{dfn:MetricEigenvalues}
With the notation as above, the eigenvalues $0 < \mu_1 \leq \mu_2 \leq \cdots \leq \mu_n$ of $\Omega$ 
are called the \emph{eigenvalues of the metric} $g$.
\end{dfn}

\begin{prop}[\cite{BFSTW} Proposition 3.2]\label{prop:MetricEigenvalues}
Two left-invariant metrics $g_1$ and $g_2$ on $\SO(3)$ are isometric if and only if 
$g_1$ and $g_2$ have the same eigenvalues counting multiplicities. 
\end{prop}

\begin{notarem}\label{rem:Notation}%\text{}\\
We will now establish notation and collect some facts that will prove useful throughout the remainder of this section.
\begin{enumerate}
\item For the remainder of this section we will let $U$ denote the Lie group $\SO(3)$, 
$\germ{u}$ denote its Lie algebra $\germ{so}(3)$, and $g_0$ will denote the bi-invariant metric 
on $\SO(3)$ induced by $-B$, where $B$ denotes the Killing form. 
Additionally, we will let $\exp$ denote the exponential map $\exp_{U}: \germ{u} \to U$.

\item With Proposition~\ref{prop:MetricEigenvalues} in mind we let 
$$\Theta_1 = \frac{1}{2 \sqrt{2}}\left(\begin{array}{cc}-i & 0 \\0 & i\end{array}\right) , 
\; \Theta_2 = \frac{1}{2 \sqrt{2}} \left(\begin{array}{cc}0 & -i \\-i & 0\end{array}\right), \; 
\Theta_3 = \frac{1}{2 \sqrt{2}}  \left(\begin{array}{cc}0 & -1 \\1 & 0\end{array}\right) $$
denote the standard $g_0$-orthonormal basis of $\germ{so}(3) \simeq \germ{su}(2)$.
Then for any choice of positive constants $c_1, c_2$ and $c_3$  
the self-adjoint map $\Omega: (\germ{so}(3) , -B) \to (\germ{so}(3), -B)$ 
given by $\Omega(\Theta_{j}) = c_j \Theta_j$ defines a left-invariant metric
$g_{(c_1, c_2, c_3)}$ on $\SO(3)$ and, by Proposition~\ref{prop:MetricEigenvalues}, 
these account for all of the left-invariant metrics on $\SO(3)$ up to isometry. 
Now, since $\SO(2)$ is the only non-trivial connected proper subgroup of $\SO(3)$ it follows from Theorem~\ref{thm:NatRed} that up to isometry the left-invariant naturally reductive metrics on $\SO(3)$ are the metrics $g_{(\alpha, \alpha, A)}$ given by:
\begin{eqnarray}\label{eqn:SO3Metrics}
g_{(\alpha, \alpha, A)} = \alpha g_0 \upharpoonright \germ{m} \oplus A g_0 \upharpoonright \germ{K},
\end{eqnarray}
where $\germ{K} = \germ{so}(2) = \Span( \Theta_3)$ and 
$\germ{m} = \germ{K}^{\perp_0} = \Span \{\Theta_1, \Theta_2 \}$ is the orthogonal complement of $\germ{K}$ with respect to $g_0$. 
We set $K = \exp_U(\germ{K})$.

%inner product $\langle \Theta_i, \Theta_j \rangle = c_{i} \delta_{ij}$ 
%defines a left-invaraint metric $g_{(c_1, c_2, c_3)}$ on $\SO(3)$ and, 
%by Proposition~\ref{prop:MetricEigenvalues}, these account for all left-invariant metrics on $\SO(3)$ up to isometry.
%Now, since $\SO(2)$ is the only non-trivial connected proper subgroup of $\SO(3)$ 
%it follows from Theorem~\ref{thm:NatRed} that up to isometry the left-invariant naturally 
%reductive metrics on $\SO(3)$ are the metrics $g_{(\alpha, \alpha, A)}$ given by:
%\begin{eqnarray}\label{eqn:SO3Metrics}
%g_{(\alpha, \alpha, A)} = \alpha g_0 \upharpoonright \germ{u} \oplus A g_0 \upharpoonright \germ{K},
%\end{eqnarray}
%where $\germ{K} = \germ{so}(2) = \Span( \Theta_3)$ and 
%$\germ{u} = \germ{K}^{\perp_0} = \Span \{\Theta_1, \Theta_2 \}$ is the orthogonal complement of $\germ{K}$ with respect to $g_0$. 
%We set $K = \exp_H(\germ{K})$.

\item Let $\germ{p}$ denote the $\Ad(K)$-invariant complement of $\Delta \germ{K} \leq \germ{g} \times \germ{K}$ discussed in Appendix~\ref{sec:AdK}. Then we have the following.
\begin{enumerate}
\item[(a)] If $\alpha = A$, then by Equation~\ref{eqn:QOrthogonal2} we see $\germ{p} = \germ{g} \oplus 0$. In which case 
$$\germ{p} = \Span \{\frac{1}{\sqrt{\alpha}}(\Theta_1, 0),  \frac{1}{\sqrt{\alpha}}(\Theta_2, 0), \frac{1}{\sqrt{\alpha}}(\Theta_3, 0) \} $$
and 
$$ \Delta \germ{K} = \Span \{D = (\Theta_3, \Theta_3) \},$$
where by $\Span\{A_1, \ldots, A_k\}$ we denote the linear span of $A_1, \ldots, A_k$ over $\R$.
\item[(b)] If $\alpha \neq A$, then by Equation~\ref{eqn:QOrthogonal1} $\germ{p} = \germ{p}_1 \oplus \germ{q}_0$, where 
$\germ{p}_1 = \{(X, 0) : X \in \germ{u} = \germ{K}^{\perp_0} \}$ and
$\germ{q}_0 =\{ (\overline{A} Z, - \alpha Z) : Z \in \germ{K}\}$ for 
$\overline{A} = \frac{A \alpha}{\alpha - A}$.
In which case 
$$\germ{p} = \operatorname{Span}\{ Z_1 =(\frac{1}{\sqrt{\alpha}}\Theta_1, 0), 
Z_2 = (\frac{1}{\sqrt{\alpha}}\Theta_2, 0), 
Z_3 = \frac{1}{\sqrt{A}(\overline{A} +\alpha)}(\overline{A} \Theta_3, -\alpha \Theta_3) \}$$
and 
$$ \Delta \germ{K} = \Span \{D = (\Theta_3, \Theta_3) \}.$$
It is clear that the adjoint action of $\Delta K \leq G \times K$ on $\germ{p}$ fixes $Z_3$ 
and acts as the group of rotations on $\operatorname{Span}_{\R}\{Z_1, Z_2\} = \germ{p}_1$.
\end{enumerate}

\item For any $(V, W) \in \germ{p}$, where $\germ{p}$ is as above, the geodesic $\gamma_{(V,W)}(t)$ 
with $\gamma_{(V,W)}(0) = e$ and $\gamma_{(V,W)}'(0) = V-W$ is given by 
$$\gamma_{(V,W)}(t) = \exp(tV)\exp(-tW).$$
The geodesic $\gamma_{(V,W)}$ is a one-parameter subgroup of $\SO(3)$ 
if and only if $V, W \in \germ{so}(3)$ are linearly dependent.

\item For any compact Lie group endowed with a bi-invariant metric the sectional curvature of a $2$-plane $\sigma$ in the Lie algebra spanned by two orthonormal vectors $X$ and $Y$ is given by $\operatorname{Sec}(\sigma) = \frac{1}{4} \| [X,Y]\|^2$. Consequently, with respect to the metric $g_0$, the Lie group $\SO(3)$ has constant sectional curvature $\frac{1}{8}$ and is double covered by $S^3(2\sqrt{2})$, the round $3$-sphere of radius $2\sqrt{2}$. It follows that the geodesics in $(\SO(3), g_0)$ are all closed, have a common (primitive) length $\ell_0 \equiv 2 \sqrt{2} \pi$. %$\ell_0 \equiv \pi \sqrt{2}$, 

\item It follows from the previous remark that any two primitive geodesics through a given point 
of $\SO(3)$ with respect to $g_0 = g_{(1,1,1)}$ have only one point in common or 
have exactly the same image. Furthermore, since $g_0$ 
is bi-invariant, its geodesics through $e$ coincide with the one-parameter subgroups of $\SO(3)$. 
Given a vector $X \in \germ{u} = \germ{so}(3)$ we then define its \emph{period} to be 
$\operatorname{Per}(X) =  \frac{\ell_0}{\|X \|_0}$, so $\operatorname{Per}(X)$ is 
the amount of time it takes for the one-parameter subgroup $\exp(tX)$ to return to the identity 
element for the first time.

\item It will be useful to observe that 
$\vol(g_{(\alpha, \alpha, A)}) = \alpha \sqrt{A} V_0$, 
where $V_0 \equiv \vol(g_{(1,1,1)}) = \frac{1}{2} \vol(S^3(2 \sqrt{2})) = 16\sqrt{2} \pi^2 $.

\end{enumerate}

\end{notarem}

We now describe the closed geodesics of an arbitrary naturally reductive metric on $\SO(3)$
and compute the length spectrum.

\begin{thm}\label{thm:ClosedGeodesics}
Consider the naturally reductive metric $g_{(\alpha, \alpha, A)}$ on $\SO(3)$ and let $\ell_0$ be as in \ref{rem:Notation}(5).

\begin{enumerate}
\item If $\alpha = A$, then the closed geodesics through the identity are precisely the 
one-parameter subgroups of $\SO(3)$ and the non-trivial primitive geodesics are all 
of length $\sqrt{A}\ell_0$.

\item If $A \neq \alpha$, then the geodesic $\gamma_{(V,W)}$ is closed if and only if
one of the following holds:
\begin{enumerate}
\item $(V,W) \in \germ{p}_1$, in which case $\gamma_{(V,W)}$ is a one-parameter 
subgroup of $\SO(3)$ with primitive length $\sqrt{\alpha}\ell_0$.
%Geodesics of this form or translates thereof will be said to be of Type I;
\item $(V,W) \in \germ{q}_0$,  in which case $\gamma_{(V,W)}$ is a one-parameter 
subgroup of $\SO(3)$ with primitive length $\sqrt{A}\ell_0$. 
%Geodesics of this form or translates thereof will be said to be of Type II;
\item $(V,W) = (X + \overline{A}Z, -\alpha Z) \in \germ{p}$, where $X \neq 0 \in \germ{u}$ and 
$Z \neq 0  \in \germ{K}$ and there exist $p, q \in \N$ relatively prime integers such that:
\begin{enumerate}
\item $\frac{q^2}{p^2} > \frac{A^2}{(A-\alpha)^2}$
\item $\|X\|_0^2 = \sigma(p,q, \alpha, A) \|Z\|_0^2,$ where $\sigma(p,q, \alpha, A) \equiv \frac{q^2 \alpha^2}{p^2} - \frac{A^2 \alpha^2}{(\alpha - A)^2}$.

\end{enumerate}

In this case we see that the closed geodesic $\gamma_{(V,W)}$ is not a one-parameter subgroup  
and its primitive length is given by $\sqrt{\alpha}\ell_0[q^2 + p^2\frac{A}{\alpha - A}]^{\frac{1}{2}}$,
which is always strictly larger than $\sqrt{\alpha}\ell_0$.
%Geodesics of this form or translates thereof will be said to be of Type III.

\end{enumerate}

\end{enumerate}

Consequently, the length spectrum of $g_{(\alpha, \alpha, A)}$ is given by 

$$\Spec_L(g_{(\alpha, \alpha,A)}) = \left\{ 
\begin{array}{ll}
\{0 \} \cup \{k \sqrt{\alpha}\ell_0 : k \in \N \}  & \alpha =  A\\
\{0\} \cup \{k \sqrt{\alpha}\ell_0, k \sqrt{A}\ell_0, k \tau : k \in \N \mbox{ and } \tau > 0 \mbox{ with } 
\mathcal{E}_{\tau, \alpha, A} \neq \emptyset \} & A \neq \alpha,
\end{array}
\right. $$
where for each $\tau >0$ we let $\mathcal{E}_{\tau, \alpha, A}$ denote the finite collection of relatively prime 
ordered pairs $(p,q) \in \N \times \N$ satisfying $\frac{q}{p} > |\frac{A}{A-\alpha}|$ and 
$\sqrt{\alpha}\ell_0[q^2 + p^2\frac{A}{\alpha - A}]^{\frac{1}{2}} = \tau$.
\end{thm}

\begin{dfn}\label{dfn:GeodesicTypes}
Let $g_{(\alpha, \alpha, A)}$ be a naturally reductive metric on $\SO(3)$ with $\alpha \neq A$.

\begin{enumerate}
\item A geodesic of the form given in Theorem~\ref{thm:ClosedGeodesics}(2a) or 
a translate thereof is said to be of \emph{Type I}.
\item A geodesic of the form given in Theorem~\ref{thm:ClosedGeodesics}(2b) or 
a translate thereof is said to be of \emph{Type II}.
\item A geodesic of the form given in Theorem~\ref{thm:ClosedGeodesics}(2c) or 
a translate thereof is said to be of \emph{Type III}.
\end{enumerate}
A periodic orbit of the geodesic flow of $g_{(\alpha, \alpha, A)}$ will be called Type I, II or III according to whether its corresponding closed geodesic is.
\end{dfn}

\begin{rem}
The systole of a closed Riemannian manifold $(M,g)$, denoted $\operatorname{Syst}(M,g)$, is the minimum length of a non-contracible closed geodesic and is necessarily at least as large as $\tau_{\rm{min}}(M,g)$, the minimum length of a non-trivial closed geodesic. Theorem~\ref{thm:ClosedGeodesics} shows us that if $\alpha \neq A$, then the shortest non-trivial closed geodesic with respect to $g_{(\alpha, \alpha, A)}$ is always of Type I or Type II. Therefore, since (primitive) one-parameter subgroups of $\SO(3)$ are homotopically non-trivial, it follows that for any $\alpha, A >0$, we have $\operatorname{Syst}(g_{(\alpha, \alpha, A)}) = \tau_{\rm{min}}(g_{(\alpha, \alpha, A)})$.
%the systole with respect to any metric in $\mathcal{M}_{\rm{Nat}}(\SO(3))$ coincides with the length of the shortest non-trivial closed geodesic. 
We also note that it is easy to show that a prime geodesic of Type III is homotopically trivial if and only if $p+q$ is even.
\end{rem}

\begin{rem}\label{rem:Systole}
In the case where $A \leq \alpha$ the primitive geodesics of Type I and II are shorter 
than the primitive geodesics of Type III.
However, when $A > \alpha$, this need not be the case. 
For example, if we let $\alpha =1$ and  $A = 10$, then $(p,q) = (1,2)$
gives rise to a primitive geodesic that is not a one-parameter subgroup and is of length $\ell_0 \sqrt{4 + \frac{10}{9}}$.
However, if $A > \alpha$ and $(A - \alpha)^2 < \alpha$, then the prime geodesics of 
Type I and II will still be shorter than the prime geodesics of Type III. 
\end{rem}

\begin{proof}[Proof of Theorem~\ref{thm:ClosedGeodesics}]
For any vector $U \in TG$ we will let $\|U\|_0$ (respectively $\|U\|$) denote 
its length with respect to the metric $g_0$ (respectively $g_{(\alpha, \alpha, A)}$).

In the case where $\alpha = A$ we recall from \ref{rem:Notation} that
$\germ{p} = \germ{p}_1' = \germ{g} \oplus 0$.
Hence, the geodesics $\gamma_{(V,0)}(t) = \exp(tV)$ are one-parameter subgroups of $G$ and 
the primitive non-trivial geodesics are of length $\sqrt{A}\ell_0 = \sqrt{\alpha}\ell_0$ with respect to $g_{(A,A,A)}$. Thus establishing $(1)$.
 
In the case where $\alpha \neq A$ we recall that 
$\germ{p} = \germ{p}_1 \oplus \germ{q}_0$, where
$\germ{p}_1 = \{(X, 0) : X \in \germ{K}^{\perp_0} \}$ and $\germ{q}_0 = \{ (\overline{A}Z, -\alpha Z) : Z \in \germ{K} \}$
(since $\germ{K}$ is abelian). To find the closed geodesics and their lengths we consider the following three cases.

\medskip

\noindent
{\bf Case I:} $(V,W) = (X, 0) \in \germ{p}_1$ for some $X \neq 0 \in \germ{K}^{\perp_0}$.

\medskip
In this case the geodesic $\gamma_{(V,W)}(t) = \exp(tX)$ is a non-trivial one-parameter subgroup of $\SO(3)$. Consequently, it is closed and has primitive length 
$L( \gamma_{(V,W)}) = \operatorname{Per}(X) \cdot \|X\| = \operatorname{Per}(X) \sqrt{\alpha} \|X\|_0 = \sqrt{\alpha} \ell_0$.

%\begin{eqnarray*}
%L( \gamma_{(V,W)}) &=& \operatorname{Per}(X) \cdot \|X\| \\ &=& \operatorname{Per}(X) \sqrt{\alpha} \|X\|_0 \\  &=& \sqrt{\alpha} \ell_0.\\
%\end{eqnarray*}

\medskip
\noindent
{\bf Case II:} $(V,W) = (\overline{A} Z, -\alpha Z) \in \germ{q}_0$ for some $Z \neq 0 \in \germ{K}$

\medskip
In this case the geodesic $\gamma_{(V,W)}(t) = \exp_G(t(\overline{A} + \alpha)Z)$ 
is a non-trivial one-parameter subgroup of $\SO(3)$. 
Consequently, it is closed and has primitive length $L( \gamma_{(V,W)}) = \operatorname{Per}((\overline{A} + \alpha)Z) \cdot \|(\overline{A} + \alpha)Z\| = \operatorname{Per}((\overline{A} + \alpha)Z) \sqrt{A} \|(\overline{A} + \alpha)Z\|_0 = \sqrt{A} \ell_0$.

%\begin{eqnarray*}
%L( \gamma_{(V,W)}) &=& \operatorname{Per}((\overline{A} + \alpha)Z) \cdot \|(\overline{A} + \alpha)Z\| \\ &=& \operatorname{Per}((\overline{A} + \alpha)Z) \sqrt{A} \|(\overline{A} + \alpha)Z\|_0 \\ &=& \sqrt{A} \ell_0.\\
%\end{eqnarray*}

\medskip
\noindent
{\bf Case III:} $(V,W) = (X + \overline{A} Z, -\alpha Z )$, where 
$X \neq 0 \in \germ{K}^{\perp_0}$ and  $Z \neq 0 \in \germ{K}$.

\medskip
The geodesic $\gamma_{(V,W)}(t) = \exp(t(X+ \overline{A}Z)) \exp(t \alpha Z)$ is clearly not a one-parameter 
subgroup of $\SO(3)$, and it is closed if and only if there is a $t_0 > 0 $ such that 
\begin{eqnarray}\label{eqn:Intersection1}
\exp(t_0(X+ \overline{A}Z)) = \exp(- t_0 \alpha Z).
\end{eqnarray}
As noted in \ref{rem:Notation}(5), the images of two non-trivial one-parameter subgroups $\exp(tX_1)$ and $\exp(tX_2)$  in $\SO(3)$ 
either have only the identity element in common or are identical, and the latter occurs if and only if 
$X_1$ and $X_2$ are linearly dependent. Therefore, since $X+ \overline{A}Z$ and $\alpha Z$ are linearly 
independent we see that Equation~\ref{eqn:Intersection1} holds if and only if there is a $t_0 > 0$ such that 
\begin{eqnarray}\label{eqn:Intersection2}
e^{t_0(X+ \overline{A}Z)} = e^{- t_0 \alpha Z} = e,
\end{eqnarray}
which is equivalent to the existence of relatively prime integers $p, q \in \N$ such that 
$p \operatorname{Per}(\alpha Z) = q \operatorname{Per}(X + \overline{A}Z)$. 
Writing out the period of $\alpha Z$ and $X + \overline{A}Z$ explicitly 
we find that Equation~\ref{eqn:Intersection2} holds if and only if there exist relatively prime $p, q \in \N$ such that 
\begin{enumerate}
\item $\sigma(p,q, \alpha, A) \equiv \alpha^2(\frac{q^2}{p^2} - \frac{A^2}{(\alpha -A)^2}) > 0 $;
%\frac{\alpha^2 q^2}{p^2} - \overline{A}^2 = \frac{\alpha^2(q^2(\alpha-A)^2 - p^2A^2)}{p^2(\alpha -A)^2}> 0$
\item $\|X\|_0^2 = \sigma(p, q, \alpha, A) \|Z\|_0^2$.
\end{enumerate}
The function $\sigma$ has the property that 
$\sigma(p, q, \alpha, A) = \sigma(\tilde{p}, \tilde{q}, \alpha, A)$
if and only if $\frac{q}{p} = \frac{\tilde{q}}{\tilde{p}}$ and clearly $\sigma(p,q,\alpha, A) >0$ 
is equivalent to $\frac{q^2}{p^2} > \frac{A^2}{(A-\alpha)^2}$.

Now, let $X \neq 0 \in \germ{u}$, $Z \neq 0 \in \germ{K}$ and let $p, q \in \N$ be relatively prime integers 
such that Equation~\ref{eqn:Intersection2} holds. Then $\gamma_{(V,W)}$ is closed and its primitive length is given by

\allowdisplaybreaks[3]

\begin{eqnarray*}
L( \gamma_{(X + \overline{A}Z, -\alpha Z)})^2 &=& [q \operatorname{Per}(X + \overline{A}Z) \| X + (\overline{A} + \alpha)Z \|]^2 \\
&=& [\frac{q \ell_0}{\|X + \overline{A}Z \|_0} \| X + (\overline{A} + \alpha)Z \|]^2 \\
&=& \frac{q^2 \ell_0^2}{ \|X\|_0^2 + \overline{A}^2 \|Z\|_0^2}  \| X + (\overline{A} + \alpha)Z \|^2\\
&=& \frac{q^2 \ell_0^2}{ \|X\|_0^2 + \overline{A}^2 \|Z\|_0^2} \left( \| X\|^2 + (\overline{A} + \alpha)^2 \|Z \|^2\right)\\
&=& \frac{q^2 \ell_0^2}{ \|X\|_0^2 + \overline{A}^2 \|Z\|_0^2} \left( \alpha \| X\|_0^2 + (\overline{A} + \alpha)^2 A\|Z \|_0^2 \right)\\
&=& \frac{q^2 \ell_0^2}{ \|X\|_0^2 + \overline{A}^2 \|Z\|_0^2} \left( \alpha \| X\|_0^2 + \left( \frac{\alpha^2}{\alpha -A} \right)^2 A\|Z \|_0^2 \right)\\
&=& q^2 \ell_0^2\frac{  (\alpha \| X\|_0^2 + (\frac{\alpha^2}{\alpha -A})^2 A\|Z \|_0^2)}{ \|X\|_0^2 + \overline{A}^2 \|Z\|_0^2} \\
&=& q^2 \ell_0^2\frac{  (\alpha \| X\|_0^2 + \frac{A\alpha^4}{(\alpha -A)^2}\|Z \|_0^2)}{ \|X\|_0^2 + \overline{A}^2 \|Z\|_0^2} \\
&=& \alpha q^2 \ell_0^2 \cdot \frac{ \| X\|_0^2 + \frac{A\alpha^3}{(\alpha -A)^2}\|Z \|_0^2}{ \|X\|_0^2 + \frac{A^2 \alpha^2}{(\alpha - A)^2} \|Z\|_0^2}\\
&=& \alpha q^2 \ell_0^2 \cdot \frac{   \| X\|_0^2 + \frac{A\alpha^3}{(\alpha -A)^2\sigma(p,q, \alpha, A)}\|X \|_0^2}{ \|X\|_0^2 + \frac{A^2 \alpha^2}{(\alpha - A)^2\sigma(p,q,\alpha,A)} \|X\|_0^2} \\
&=& \alpha q^2 \ell_0^2 \cdot \frac{ \| X\|_0^2 + \frac{A\alpha^3}{(\alpha -A)^2}\|Z \|_0^2}{ \|X\|_0^2 + \frac{A^2 \alpha^2}{(\alpha - A)^2} \|Z\|_0^2} \\
&=& \alpha q^2 \ell_0^2 \cdot \frac{   \| X\|_0^2 + \frac{A\alpha^3}{(\alpha -A)^2\sigma(p,q, \alpha, A)}\|X \|_0^2}{ \|X\|_0^2 + \frac{A^2 \alpha^2}{(\alpha - A)^2\sigma(p,q,\alpha,A)} \|X\|_0^2} \\
&=& \alpha q^2 \ell_0^2 \cdot \frac{  1 + \frac{A\alpha p^2}{q^2(\alpha -A)^2 - p^2A^2}}{1 + \frac{A^2 p^2}{q^2(\alpha -A)^2 - p^2A^2}}\\
&=& \alpha q^2 \ell_0^2 \cdot (1 + \frac{p^2A(\alpha -A)}{q^2(\alpha -A)^2})\\
&=& \alpha \ell_0^2 \cdot (q^2 + p^2\frac{A}{(\alpha -A)}).
\end{eqnarray*}

In the event that $\alpha > A$ it is clear that this geodesic will have length strictly greater than $\sqrt{\alpha} \ell_0$.
To handle the case where $\alpha < A$ we note that $\alpha \ell_0^2 \cdot (q^2 + p^2\frac{A}{(\alpha -A)})$ 
is greater than $\alpha \ell_0^2$ if and only if  $\frac{A}{A-\alpha}= | \frac{A}{A-\alpha} | < \frac{q^2 -1}{p^2}$. 
But we recall that $p, q \in \N$ were chosen so that $\frac{q}{p} > |\frac{A}{A-\alpha}| = \frac{A}{A-\alpha} >1$,  
and notice that for $q > p$ we have $\frac{q^2 -1}{p^2} > \frac{q}{p}$.
Hence, for $A\neq \alpha$, we see that $L(\gamma_{(V,W)}) > \sqrt{\alpha}\ell_0$.

Cases I-III establish statement (2) of the theorem and the statement concerning 
the length spectrum of an arbitrary naturally reductive 
metric $g_{(\alpha, \alpha, A)}$ is now immediate. 
We conclude the proof by showing that the set $\mathcal{E}_{\tau, \alpha, A}$ is finite.

Indeed, in the case where $A < \alpha$, we see that $\mathcal{E}_{\tau, \alpha, A}$ 
is a subset of the intersection of an ellipse with the integer lattice in $\R^2$, which implies it is finite.
In the event that $A > \alpha$, the points $(p,q) \in \mathcal{E}_{\tau, \alpha, A}$ are a subset 
of the intersection of the integral lattice with the hyperbola
\begin{eqnarray*}\label{eqn:Hyperbola}
\frac{y^2}{\tau^2/\alpha \ell_0^2} - \frac{x^2}{\tau^2(A-\alpha)/\alpha \ell_0^2 A} = 1
\end{eqnarray*}
having asymptotes $y = \pm \sqrt{\frac{A}{A-\alpha}} x$.
Now, suppose $\mathcal{E}_{\tau, \alpha, A}$ is infinite, then, since 
$\frac{q}{p} > |\frac{A}{A-\alpha}| = \frac{A}{A-\alpha} > 1$, 
we see that $q$ must become arbitrarily large.
Then, since the hyperbola is asymptotic to $y = \sqrt{\frac{A}{A-\alpha}} x$,
we see that the expression $|p - \sqrt{\frac{A-\alpha}{A}} q|$ can be made arbitrarily small in 
$\mathcal{E}_{\tau, \alpha, A}$.
However, $\frac{q}{p} > \frac{A}{A-\alpha} > 1$ implies
$$p < \frac{A- \alpha}{A} q < \sqrt{\frac{A- \alpha}{A}} q$$
for any $(p,q) \in \mathcal{E}_{\tau, \alpha, A}$,
which implies the quantity $|p - \sqrt{\frac{A-\alpha}{A}} q|$ cannot be made arbitrarily small.
So, we see $\mathcal{E}_{\tau, \alpha, A}$ is finite.
%In the event that $A > \alpha$, then $\mathcal{E}_{\tau, \alpha, A}$ is a subset of the hyperbola 
%\begin{eqnarray*}\label{eqn:Hyperbola}
%\frac{x^2}{\tau^2/\alpha \ell_0^2} - \frac{y^2}{\tau^2(A-\alpha)/\alpha \ell_0^2 A} = 1
%\end{eqnarray*}
%having vertices $(\pm \frac{\tau}{\sqrt{\alpha}\ell_0})$ and asymptotes $y = \pm \sqrt{\frac{A - \alpha}{A}} x$.
%Therefore, for any $\epsilon >0$ there is an $N >0$ such that if $(q, p) \in \mathcal{E}_{\tau, \alpha, A}$ with $q > N$, then
%$ q\sqrt{\frac{A - \alpha}{A}}  - \epsilon < p < q\sqrt{\frac{A - \alpha}{A}} $. 
%On the other hand, since $(q,p) \in \mathcal{E}_{\tau, \alpha, A}$
%we know that $\frac{q}{p} > | \frac{A}{\alpha -A}| > 1$. Therefore, since $\frac{A-\alpha}{A} <1$,
%for each $(q,p) \in \mathcal{E}_{\tau, \alpha, A}$ we have 
%$$p < \frac{A-\alpha}{A} q < \sqrt{\frac{A-\alpha}{A}}q.$$
%It follows that the set $\mathcal{E}_{\tau, \alpha, A}$ must be finite.
\end{proof}

For any period $\tau$ of the geodesic flow of a symmetric metric $g_{(\alpha,\alpha,\alpha)}$ on $\SO(3)$, we see that $\Fix(\Phi_\tau)$ is the entire unit tangent bundle and it follows that such metrics are clean. 
%Of course, this also follows from Theorem~\ref{thm:SymmSpcsCIH}.
We now wish to examine the ``cleanliness'' of the other naturally reductive metrics on $\SO(3)$.
Towards this end we begin by examining the fixed point sets of the geodesic flow for 
naturally reductive metrics that are not symmetric.
%metrics $g_{(\alpha, \alpha, A)}$ where $A \neq \alpha$. 

\begin{lem}\label{lem:Submanifolds}
Consider the naturally reductive metric $g_{(\alpha, \alpha, A)}$ on $U = \SO(3)$, where $\alpha \neq A$, and let $U \times K = \SO(3) \times \SO(2)$ be the connected component of the identity in the isometry group of $g_{(\alpha, \alpha, A)}$.
%of the identity in $\Isom(g_{(\alpha, \alpha, A)})$.%%, act on the unit tangent bundle in the natural fashion.
We let $v = c_1Z_1 + c_2 Z_2 + c_3 Z_3 \in \germ{p} \equiv T_e U$ be a unit vector where $Z_1, Z_2, Z_3 \in T_eU$ 
is the orthonormal basis given in \ref{rem:Notation}(3).% and $c_1^2 + c_2^2 + c_3^2 =1$.

\begin{enumerate}
\item If $c_1^2 + c_2^2 = 1$, then $(U\times K) \cdot v \simeq \SO(3) \times S^1$ and this 
$4$-dimensional submanifold of $T^1\SO(3)$ accounts for all the unit speed primitive geodesics of 
Type I, all of which have length $\sqrt{\alpha} \ell_0$.
The manifold $(U\times K) \cdot v$ is said to be a Type I component.

\item If $c_3 = \pm 1$, then $(U \times K) \cdot v \simeq \SO(3)$ and the $3$-dimensional submanifold 
$(U\times K) \cdot v  \cup (U\times K) \cdot (-v)$ of $T^1\SO(3)$ accounts for all the 
unit speed primitive geodesics of Type II, all of which have length $\sqrt{A}\ell_0$.
The manifold $(U \times K) \cdot v$ is said to be a Type II component.

\item Let $\tau > 0$ be such that $\mathcal{E}_{\tau, \alpha, A}$ is non-empty. 
For each $(p,q) \in \mathcal{E}_{\tau, \alpha, A}$ 
fix a unit vector $v_{(p,q)} = c_1 Z_1 + c_2 Z_2 + c_3 Z_3$, where 
$c_1^2 + c_2^2 = \frac{\sigma(p,q,\alpha, A)}{\sigma(p,q,\alpha,A) + 1}$ and 
$c_3^2= \frac{1}{\sigma(p,q, \alpha, A) + 1}$. 
Then $(U\times K) \cdot v_{(p,q)} \simeq \SO(3) \times S^1_{(p,q)}$, 
where $S^1_{(p,q)}= \{ xZ_1 + yZ_2 + zZ_3: x^2 + y^2 = \frac{\sigma}{\sigma +1} z = c_3\}$,
and the $4$-dimensional submanifold
$\cup_{(p,q) \in \mathcal{E}_{\tau, \alpha, A}} (U \times K) \cdot (\pm v_{(p,q)})$ accounts 
for the unit speed primitive geodesics of Type III having length $\tau$.
The manifold $(U \times K) \cdot v_{(p,q)}$ is said to be a Type III component.
\end{enumerate}

\end{lem}

\begin{proof}
We recall that the isotropy group of the identity element corresponding to the 
natural action of $U \times K$ on $\SO(3)$ is $\Delta K = \SO(3)$, and as we noted in \ref{rem:Notation}(3) the isotropy 
action of $\Delta K$ on $\germ{p} \equiv T_e G$ acts via rotations on $\germ{p}_1 = \Span_{\R}\{Z_1, Z_2\}$ and 
fixes $\germ{q}_0 = \Span_{\R}\{Z_3\}$. 
The lemma now follows from Theorem~\ref{thm:ClosedGeodesics}.
\end{proof}

\begin{lem}\label{lem:Discrete}
For any $B > 0$, there are finitely many $0 < \tau < B$ such that $\mathcal{E}_{\tau, \alpha, A}$ is non-empty.
\end{lem}

\begin{proof}
This follows immediately from the fact that a Type III geodesic has length of the form 
$\sqrt{\alpha}\ell_0 [q^2 + p^2 \frac{A}{\alpha - A}]^{\frac{1}{2}}$, where $p, q, \in \N$, 
and the values of this function form a discrete subset of $\R$. 
%This also follows from the fact that the length spectrum of an 
%analytic manifold must be discrete \cite{SS}. 
\end{proof}

Using Theorem~\ref{thm:ClosedGeodesics} and Lemmas~\ref{lem:Submanifolds} and \ref{lem:Discrete} the following is immediate.

\begin{cor}\label{cor:Submanifolds}
Let $g_{(\alpha, \alpha, A)}$ be a naturally reductive metric on $\SO(3)$ with 
unit tangent bundle $T^1 \SO(3)$ and corresponding geodesic flow 
$\Phi_t : T^1 \SO(3) \to T^1 \SO(3)$, $t \in \R$. %Also, let $\ell_0 = \pi \sqrt{2}$. 
Then, for each period $\tau$ of the geodesic flow, $\Fix (\Phi_\tau)$ is a union of finitely many (homogeneous) submanifolds of $T^1\SO(3)$ and for each $u \in \Fix(\Phi_\tau)$ the connected component of $\Fix(\Phi_\tau)$ containing $u$ is given by $\operatorname{Isom}(g_{(\alpha, \alpha, A)})^0\cdot u$, where $\operatorname{Isom}(g_{(\alpha, \alpha, A)})^0$ denotes the connected component of the identity in the isometry group. In particular, we have the following:

\begin{enumerate}
\item $\alpha = A$ if and only if $|\tau| = \sqrt{\alpha}\ell_0$ is the length of the shortest non-trivial closed geoedesic 
and $\Fix(\Phi_\tau) = T^1 \SO(3)$ is $5$-dimensional.

\item $A < \alpha$ if and only if $|\tau| = \sqrt{A}\ell_0$ is the length of the shortest non-trivial closed geodesic and  $\Fix(\Phi_\tau) \simeq \SO(3) \cup \SO(3)$ is $3$-dimensional. In which case all geodesics of length $|\tau| = \sqrt{A}\ell_0$ are of Type II.

\item $A > \alpha$ if and only if $|\tau| = \sqrt{\alpha}\ell_0$ is the length of the shortest non-trivial closed geodesic and $\Fix(\Phi_\tau) \simeq \SO(3) \times S^1$ is $4$-dimensional. In which case all geodesics of length $|\tau| = \sqrt{\alpha}\ell_0$ are of Type I.
\end{enumerate}
\end{cor}

We now give an explicit description of the naturally reductive metrics on $\SO(3)$ 
which fail to be clean.

\begin{thm}\label{thm:CleanMetrics}
The naturally reductive metric $g_{(\alpha, \alpha, A)}$ is unclean
if and only if $A \in \alpha\Q_{+} - \{ \alpha \}$, where $\Q_+$ 
denotes the positive rational numbers.
Moreover, if we express $A  \in \alpha\Q_{+} - \{ \alpha \}$ as $A = \frac{2\alpha j}{k}$,
where $k, j \in \N$ are relatively prime, then a period $\tau$ of the geodesic flow of $g_{(\alpha, \alpha,A)}$ is unclean if and only if $|\tau|= mk \sqrt{A} \ell_0$ for some $m \in \N$. 
%Consequently, if $\tau \in \Spec_{L}(g_{\alpha, \alpha, A})$ is unclean, then there Type II geodesics of length $\tau$.
%Consequently, the length of the shortest non-trivial closed geodesic---which in this case coincides with the systole---is always 
%clean and as such is ``audible'' within the class of naturally reductive metrics on $\SO(3)$.
\end{thm}

%\begin{rem}
%We recall that the naturally reductive metric $g_{(\alpha, \alpha, A)}$ is normal homogeneous 
%if and only if $A \leq \alpha$. Although the symmetric metric $g_{(A,A ,A)}$ is clean, the above theorem 
%tells us that certain normal homogeneous metrics are unclean. 
%\end{rem}

\begin{cor}\label{cor:CleanMetrics}
Let $|\tau|$ be the length of the shortest non-trivial closed geodesic with respect to 
a left-invariant naturally reductive metric on $\SO(3)$. Then, $\tau$ is clean.
\end{cor}

\begin{proof} Without loss of generality, we may assume $\tau \geq 0$. Let $g_{(\alpha, \alpha, A)} \in \mathcal{M}_{\rm{Nat}}(\SO(3))$ and $\tau_{\rm{min}}$ denote the length of its shortest non-trivial closed geodesic. 
If $A \leq \alpha$, then $\tau_{\rm{min}} = \sqrt{A} \ell_0$ and in the event that $A > \alpha$ we see that $\tau_{\rm{min}} = \sqrt{\alpha}\ell_0$. Now, let $\tau \in \Spec_{L}^{\pm}(g_{(\alpha, \alpha, A)})$ be an unclean period of the geodesic flow. Then, by Theorem~\ref{thm:CleanMetrics}, we have that $A \in \alpha \Q_{+} - \{ \alpha \}$ and, if we express $A$ as $\frac{2j}{k}\alpha$, where $j, k$ are relatively prime, 
then $|\tau| = mk \sqrt{A}\ell_0$ for some positive integer $m$.
It follows that if $A < \alpha$, then $k \geq 3$ and, therefore, $|\tau| = mk \sqrt{A}\ell_0 > \tau_{\rm{min}} = \sqrt{A} \ell_0$.
Similarly, if $A > \alpha$, then $|\tau| = mk \sqrt{A}\ell_0 > \tau_{\rm{min}} = \sqrt{\alpha} \ell_0$.
Therefore, $\tau = \pm \taumin$ is always clean.
\end{proof}

\begin{proof}[Proof of Theorem~\ref{thm:CleanMetrics}]
In Corollary~\ref{cor:Submanifolds} we have already established that for each $\tau$ in the length spectrum of $g_{(\alpha, \alpha, A)}$, the fixed point set $\Fix(\Phi_\tau)$ is the disjoint union of finitely many \emph{homogeneous} submanifolds $\Theta_1, \ldots , \Theta_q$. 
Hence, our objective is to show that for each $\tau \in \Spec_L(g_{(\alpha, \alpha, A)})$, each $j = 1, \ldots, q \equiv q(\tau)$ 
and each $u \in \Theta_j$ we have
\begin{eqnarray}\label{eqn:Clean}
\ker(D_u \Phi_{\tau} - \Id_u) = T_u(\Theta_j).
\end{eqnarray}
That is, we must show that the periodic Jaocbi fields $Y(t)$ along the geodesic $\gamma_v(t)$ 
are precisely those whose initial conditions satisfy $(Y(0), \nabla Y(0)) \in T_v(\Theta_j)$. 
%So we need to confirm that the fixed point set of the Poincar\'{e} map $P \equiv P_v$ 
%along $\gamma_v$ is precisely $T_v(\Fix(\Phi_\tau)) = T_v(N_j)$.
Since $g_{(\alpha, \alpha, A)}$ is a homogeneous metric, 
it is enough to verify this for some $v \in T_e G \cap \Fix (\Phi_\tau)$.
And, since the connected components are homogeneous, Corollary~\ref{cor:PeriodicFields} informs us that 
$\ker(D_v \Phi_{\tau} - \Id_v) = T_v(\Fix (\Phi_\tau))$ if and only if $V_{4}^{\rm{per}} = V_4^{\rm{iso}}$. 

In the case where $A = \alpha$, it is clear that the metric is clean 
since all geodesics are closed and have the same primitive length $\ell_0$. 
%This also follows from Theorem~\ref{thm:SymmSpcsCIH} since $g_{(A,A,A)}$ is symmetric.
Therefore, the remainder of our discussion will focus on the case where $A \neq \alpha$.

For any left-invariant naturally reductive metric $g= g_{(\alpha,\alpha,A)}$, where $\alpha \neq A$, and $\tau \in \Spec_{L}(g_{(\alpha, \alpha,A)}$, if $\Theta \subseteq \Fix(\Phi_\tau)$ is of Type I or III the computations below will show that Equation~\ref{eqn:Clean} holds for any $u \in \Theta$. However, if $\Theta$ is Type II, we will find that Equation~\ref{eqn:Clean} fails to hold precisely when 
\begin{enumerate}
\item $A = \alpha \frac{2j}{k} \neq \alpha$ with $\gcd(j,k) =1$, and 
\item $\tau = mk\sqrt{A}\ell_0$.
\end{enumerate}
We now take up the details.

Suppose that $A \neq \alpha$.
Now, let $\germ{p} \equiv T_eG$ denote the
%Now, let $\germ{p}$ denote the 
$\Ad(\Delta K)$-invariant complement of $\Delta \germ{K} = \Span \{D\}$ in $\germ{g} \times \germ{K}$. 
%and we identify $\germ{p}$ with $T_eG$ in the usual fashion. 
Then, following \ref{rem:Notation}(3), the collection $\{Z_1, Z_2, Z_3\}$ 
forms a $g$-orthonormal basis for $\germ{p}$. 
Hence, any unit vector $v \in \germ{p} \equiv T_e G$ is of the form 
$c_1 Z_1 + c_2 Z_2 + c_3 Z_3$, where $c_1^2 + c_2^2 + c_3^2 = 1$. 
By Theorem~\ref{thm:ClosedGeodesics} the geodesic 
$\gamma_v (t) = \exp_{G \times K}(tv) \cdot e$ is closed if and only if 
one of the following hold:
\begin{enumerate}
\item $c_1^2 + c_2^2 = 1$ (i.e., $\gamma_v$ is of Type I);
\item $c_3 = \pm 1$ (i.e., $\gamma_v$ is of Type II);
\item $c_1^2 + c_2^2 = \frac{\sigma(p,q, \alpha, A)}{\sigma(p,q, \alpha, A) + 1}$ and $c_3 = \pm \sqrt{\frac{1}{\sigma(p,q, \alpha, A) + 1}}$ for some choice of $p, q \in \N$ relatively prime with $\frac{q^2}{p^2} > (\frac{A}{A- \alpha})^2$ (i.e., $\gamma_v$ is of Type III).
\end{enumerate}
In the case where $\gamma_v$ is closed we must determine 
the fixed point set of the associated Poincar\'{e} map $P : E \oplus E \to E \oplus E$, 
where (as in Section~\ref{sec:PoincareMap}) $E = \{ u \in \germ{p} : \langle u, v\rangle  = 0 \}$. 
By Corollary~\ref{cor:PeriodicFields}, this means we must determine the subspaces 
$V_1, \ldots,V_4^{\rm{per}}, V_5 \leq E \oplus E$. 
In particular, as noted above, we want to determine whether $V_4^{\rm{per}} = V_4^{\rm{iso}}$.
Towards this end, in Figure~\ref{fig:Bracket} we have collected information concerning Lie brackets in 
$\germ{g} \times \germ{K} = \germ{p} \oplus \Delta \germ{K}$ that will be useful in our computations.
We now examine the behavior of the Poincar\'{e} map associated to the three types of closed geodesics 
listed above.

%\begin{center}
\begin{figure}
\begin{tabular}{| c | c || c | c |} \hline
$A$ & $B$ & $[A,B]_{\Delta \germ{K}}$ & $[A,B]_{\germ{p}}$ \\ \hline\hline 
$Z_1$ & $Z_2$ & $\frac{1}{\sqrt{2}(\overline{A} + \alpha)}D$ & $\frac{\sqrt{A}}{\alpha \sqrt{2}}Z_3$ \\ \hline
$Z_1$ & $Z_3$ & $0$ & $-\frac{\sqrt{A}}{\alpha \sqrt{2}}Z_2$ \\ \hline
$Z_2$ & $Z_3$ & $0$ & $\frac{\sqrt{A}}{\alpha \sqrt{2}}Z_1$ \\ \hline
$Z_1$ & $D$ & $0$ & $-\frac{1}{\sqrt{2}}Z_2$ \\ \hline
$Z_2$ & $D$ & $0$ & $\frac{1}{\sqrt{2}}Z_1$ \\ \hline
$Z_3$ & $D$ & $0$ & $0$ \\ \hline
\end{tabular}
\caption{The Lie Bracket in $\germ{g} \times \germ{K} = \germ{p} \oplus \Delta \germ{K}$}
\label{fig:Bracket}
\end{figure}
%\end{center}

\medskip
\noindent 
{\bf Case I:} $v = c_1 Z_1+ c_2 Z_2$ with $c_1^2 + c_2^2 = 1$.

By Theorem~\ref{thm:ClosedGeodesics} and Corollary~\ref{cor:Submanifolds}
we see that $v \in \Fix(\Phi_\tau)$ if and only if $\tau = k \sqrt{\alpha}\ell_0$ for $k \in \N$, 
in which case the connected component of $\Fix(\Phi_\tau)$ containing $v$ 
is the $4$-dimensional manifold $(G\times K) \cdot v  \simeq \SO(3) \times S^1$.

Fix $\tau = k \sqrt{\alpha}\ell_0$. Since $v = c_1 Z_1+ c_2 Z_2$ with $v_1^2 + c_2^2 = 1$ 
we see that $E = \Span\{ c_2 Z_1 - c_1Z_2, Z_3\}$. 
We now compute the eigenspaces of the self-adjoint map $B: E \to E$ given by 
$B(\cdot) = -[ v, [v, \cdot]_{\Delta \germ{K}}]$. We have 
\begin{eqnarray*}
B(Z_3) &=& -[c_1 Z_1 + c_2 Z_2, [c_1 Z_1 + c_2 Z_2, Z_3]_{\Delta \germ{K}}] \\
	     &=& -[c_1 Z_1 + c_2 Z_2, 0] \\
	     &=& 0 
\end{eqnarray*}
and 
\begin{eqnarray*}
B(c_2 Z_1 - c_1Z_2) &=& [c_1 Z_1 + c_2 Z_2, [Z_1, Z_2]_{\Delta \germ{K}}] \\
				      &=& [c_1 Z_1 + c_2 Z_2, \frac{1}{\sqrt{2}(\overline{A} + \alpha)}D ]\\
				      &=& \frac{c_1}{\sqrt{2}(\overline{A} + \alpha)}[Z_1, D] + \frac{c_2}{\sqrt{2}(\overline{A} + \alpha)}[Z_2, D]\\
				      &=& -\frac{c_1}{\sqrt{2}(\overline{A} + \alpha)}Z_2 + \frac{c_2}{\sqrt{2}(\overline{A} + \alpha)}Z_1 \\
				      &=& -\frac{1}{\sqrt{2}(\overline{A} + \alpha)}(c_2 Z_1 - c_1 Z_2). 
\end{eqnarray*}
Hence, $E_0 = \Span \{ Z_3 \}$ and $E_1 = \Span \{c_2 Z_1 - c_1 Z_2 \}$. 
Now, let $T : E \to E$ be the skew-symmetric map $T(\cdot) = -[v, \cdot]_\germ{p}$. Then 
\begin{eqnarray*}
T(Z_3) &=& -c_1 [Z_1, Z_3]_\germ{p} - c_2[Z_2, Z_3]_\germ{p} \\
            &=& c_1\frac{\sqrt{A}}{\alpha \sqrt{2}} Z_2 - c_2 \frac{\sqrt{A}}{\alpha \sqrt{2}} Z_1\\
            &=& -\frac{\sqrt{A}}{\alpha \sqrt{2}}(c_2 Z_1 - c_1 Z_2),
\end{eqnarray*}
which is an element of $E_1$, and by skew-adjointness we have $T(c_2 Z_1 - c_1 Z_3) =  \frac{\sqrt{A}}{\alpha \sqrt{2}}Z_3$ 
which is an element of $E_0$. Therefore, $E_2 = E_0$ and $E_3 = 0$ which implies $E = E_1 \oplus E_2$. We then find that 
$$E\oplus E = V_1 \oplus V_2 \oplus V_3 \oplus V_5.$$
In particular, $V_4 = 0$.
Consequently, we conclude that the fixed vectors of $P$ 
coincide with the isotropic Jacobi fields. It then follows that 
$$\ker (D_v \Phi_\tau - \Id_v) = T_v \Fix(\Phi_\tau).$$

\medskip

\noindent
{\bf Case II:} $v = \pm Z_3$ 

By Theorem~\ref{thm:ClosedGeodesics} and Corollary~\ref{cor:Submanifolds}
we see that $v \in \Fix(\Phi_\tau)$ if and only if $\tau = k \sqrt{A}\ell_0$ for $k \in \N$, 
in which case the connected component of $\Fix(\Phi_\tau)$ containing $v$ 
is the $3$-dimensional manifold $(G\times K) \cdot v  \simeq \SO(3)$.

Fix $\tau = k \sqrt{A}\ell_0$, form some $k \in \N$. Since $v = \pm Z_3$, we find that $E = \Span\{Z_1, Z_2\}$.
It is then clear that $B \equiv 0$, and we conclude that $E_0 = E$ and $E_1 = 0$.  
The skew-adjoint map $T: E \to E$ is given by the following:
$$T(Z_1) = -[Z_3, Z_1]_\germ{p} = -\frac{\sqrt{A}}{\alpha \sqrt{2}} Z_2$$
and 
$$T(Z_2) = -[Z_3, Z_2]_\germ{p} = \frac{\sqrt{A}}{\alpha \sqrt{2}} Z_1.$$
Hence, $E_2 \equiv \{ \Theta \in E_0 : T(\Theta) \in E_1 \} = 0$, $E_3 = E_0 = E$ and we conclude that 
$$E \oplus E = V_1 \oplus V_4.$$
Therefore, since $V_1$ is $2$-dimensional and the connected component of $\Fix(\Phi_\tau)$ containing $v$
is $3$-dimensional we see that $T_v(\Fix(\Phi_\tau)) = V_1 \oplus \Span\{ (v,0)\}$, which 
implies $V_4^{\rm{iso}} = 0$.
This last equality can also be seen by recalling that 
$(X, \frac{1}{2}[v, X]_\germ{p}) \in V_4$ gives rise to a non-trivial isotropic Jacobi field along 
$\gamma_v$ if and only if $X \neq 0 \in E_3$ is such that $T(X) \in [\Delta \germ{K}, v]$.
However, since $[\Delta \germ{K}, v] = 0$ and $T: E \to E$ is an isomorphism, no such vector exists
and we see that $V_4^{\rm{iso}} = 0$ .
Hence, if $P$ has non-trivial fixed vectors in $V_4$ (i.e., $V_4^{\rm{per}} \neq 0$),  
they will not lie in $T_v (\Fix(\Phi_\tau))$.

We now recall that $(X, \frac{1}{2} [v, X]_\germ{p}) \in V_4$ is fixed by $P$ if and only if 
$\Psi(X) = X$, where $\Psi : E \to E$ is given by $\Psi = e^{\ad(k \sqrt{A} \ell_0 v)}$.
Now, since $Z_1$ and $Z_2$ span $E$ and $v = Z_3$, it follows that $\ad v = - T$; therefore, 
$$\Psi = e^{-k \sqrt{A} \ell_0 T}.$$
With respect to the basis $\{ Z_1 , Z_2\}$ of $E$ we see that $-k \sqrt{A} \ell_0T$ is represented by the following matrix
$$\left(\begin{array}{cc}0 & -\theta (\alpha, A) \\\theta(\alpha, A) & 0\end{array}\right),$$
where $\theta (\alpha, A) = \frac{kA \ell_0}{\alpha \sqrt{2}} = \frac{k A \pi}{\alpha}$.
Hence, with respect to the basis $\{ Z_1 , Z_2\}$, $\Psi$ has the following matrix
$$ \left(\begin{array}{cc}\cos \theta(\alpha, A)  & \sin \theta(\alpha, A)  \\-\sin\theta(\alpha, A)  & \cos \theta(\alpha, A) \end{array}\right).$$
Therefore, $\Psi$ has a fixed vector if and only if $\theta (\alpha ,A) \in 2\pi \N$, which is equivalent to $A \in \frac{2\alpha}{k} \N$.
This implies that $\ker(D_v \Phi_\tau - \Id_v) \neq T_v(\Fix(\Phi_\tau))$ if and only if $A \in \frac{2\alpha}{k} \N$.
Since $k \in \N$ is arbitrary, we may conclude that in the case where $A \neq \alpha$ we have 
$\ker(D_v \Phi_\tau - \Id_v) \neq T_v(\Fix(\Phi_\tau))$ if and only if $A = \frac{2 \alpha j}{k}$, where $j$ and $k$ are relatively prime,
and $\tau = m (k \sqrt{A}\ell_0)$ for some $m \in \N$.

%We conclude that $A \not\in 2 \alpha Z$ if and only if 
%$$\ker (d_v \Phi_\tau - I_v) = T_v((G\times K) \cdot v).$$
%In particular, if $A < \alpha$ (i.e, the metric is normal homogeneous), this always holds.

\medskip

\noindent
{\bf Case III:} $v = c_1 Z_1 + c_2 Z_2 + c_3 Z_3$, where $c_1^2 + c_2^2 = \frac{\sigma(p,q, \alpha, A)}{1 + \sigma(p,q,\alpha, A)}$ and $c_3 = \pm \sqrt{\frac{1}{1 + \sigma(p,q,\alpha, A)}}$ 
for unique $p, q \in \N$ relatively prime such that $\frac{q^2}{p^2} > (\frac{A}{\alpha -A})^2$.

\medskip

By Theorem~\ref{thm:ClosedGeodesics} and Corollary~\ref{cor:Submanifolds}, 
we see that in this case $v \in \Fix(\Phi_\tau)$ if and only if 
$\tau = k \sqrt{\alpha} \ell_0(q^2 + p^2 \frac{ A}{(\alpha-A)})^\frac{1}{2}$ for $k \in \N$, 
in which case the connected component of $\Fix(\Phi_\tau)$ containing $v$ 
is the $4$-dimensional manifold $(G\times K) \cdot v  \simeq \SO(3) \times S^1$.

Fix $\tau = k \sqrt{\alpha} \ell_0(q^2 + p^2 \frac{ A}{(\alpha-A)})^\frac{1}{2}$ for some $k \in \N$ and notice that 
$E = \Span\{ c_2 Z_1 - c_1Z_2, c_1c_3Z_1 + c_2 c_3 Z_2 - (c_1^2 + c_2^2)Z_3 \}$.
To find the eigenspaces of $B : E \to E$ we observe that
\begin{eqnarray*}
B(c_2Z_1 - c_1 Z_2) &=& -[v, [v, c_2Z_1 - c_1 Z_2]_{\Delta \germ{K}}] \\
				      &=& -[c_1Z_1 + c_2Z_2+c_3Z_3, -(c_1^2 + c_2^2)[Z_1, Z_2]_{\Delta \germ{K}}] \\
				      &=& - [c_1 Z_1 c_2 Z_2 + c_3Z_3, -\frac{(c_1^2 + c_2^2)}{\sqrt{2}(\overline{A} + \alpha)}D] \\
				      &=& \frac{(c_1^2 + c_2^2)}{\sqrt{2}(\overline{A} + \alpha)} (c_1 [Z_1, D] + c_2[Z_2, D] + c_3 [Z_3, D]) \\
				      &=&  \frac{(c_1^2 + c_2^2)}{\sqrt{2}(\overline{A} + \alpha)} (c_2Z_1 - c_1Z_2)
\end{eqnarray*}
and 
\begin{eqnarray*}
B(c_1c_3Z_1 + c_2 c_3 Z_2 - (c_1^2 + c_2^2)Z_3) &=& -[v,[v, c_1c_3Z_1 + c_2 c_3 Z_2 - (c_1^2 + c_2^2)Z_3]_{\Delta \germ{K}}]\\
											     &=& -[c_1Z_1 + c_2Z_2 + c_3Z_3, c_1c_2c_3[Z_1, Z_2]_{\Delta \germ{K}} - c_1c_2c_3[Z_1, Z_2]_{\Delta \germ{K}}]\\
											     &=& 0.
\end{eqnarray*}
Hence, $E_0 = \Span \{ c_1 c_3 Z_1 + c_2 c_3 Z_2 - (c_1^2 + c_2^2)Z_3 \}$ and 
$E_1 = \Span \{c_2Z_1 - c_1Z_2\}$.
We now determine $E_2$ and $E_3$ by computing $T: E \to E$:
\begin{eqnarray*}
T(c_1 c_3 Z_1 + c_2 c_3 Z_2 - (c_1^2 + c_2^2) Z_3) &=& -[c_1 Z_1 + c_2 Z_2 + c_3Z_3, c_1c_3 Z_1 + c_2 c_3 Z_2 - (c_1^2 + c_2^2)Z_3]_{\germ{p}} \\
&=& c_1 [Z_1, Z_3]_\germ{p} + c_2[Z_2, Z_3]_\germ{p} \\
&=& \frac{\sqrt{A}}{\alpha \sqrt{2}}(c_2Z_1 - c_1Z_2)
\end{eqnarray*}
and we also see that 
$$T(c_2 Z_1 - c_1 Z_2) = - \frac{\sqrt{A}}{\alpha \sqrt{2}}(c_1c_3 Z_1 + c_2 c_3 Z_2 - (c_1^2 + c_2^2)Z_3).$$
It follows that $E_2 = \{ X \in E_0 : T(X) \in E_1\} = E_0$ and $E_3 = 0$, which allows us to see that $E = E_1 \oplus E_2$.
Therefore, $V_4 = 0$ and
$$E \oplus E = V_1 \oplus V_2 \oplus V_3 \oplus V_5.$$
%$$E \oplus E = V_1 \oplus V_2 \oplus V_3 \oplus V_5.$$
Hence, the only fixed vectors of $P$ come from isotropic Jacobi fields and we have 
$$\ker(D_v \Phi_\tau - I_v) = T_v (\Fix(\Phi_\tau)).$$

\medskip

Cases I - III now clearly imply the theorem. Indeed, when $\alpha \neq A$, we see that the cleanliness of $\tau \in \Spec_L(g_{(\alpha, \alpha, A)}$ hinges on the behavior of the Poincar\'{e} map along geodesics of length $\tau$ having Type II. The conclusion of Case II, then gives us the main statement of the theorem.
\end{proof}

%\begin{rem}
%From the previous proof we see that the failure of a naturally reductive metric on $\SO(3)$ to be 
%``clean'' hinges on the behavior of the Poincar\'{e} map on the component of $\Fix(\Phi_\tau)$ 
%containing the vector $Z_3 \in \germ{p} \equiv T_e \SO(3)$. 
%It is a natural and interesting question to determine whether there is a trace formula for those $\tau$ such that $Z_3 \in \Fix(\Phi_\tau)$.
%We plan to explore this as part of a forthcoming paper.
%\end{rem}

%The following is an immediate consequence of the Theorem~\ref{thm:CleanMetrics}.

%\begin{cor}
%The set $\mathcal{C}$ of naturally reductive left-invariant metrics satisfying the clean intersection hypothesis is open and dense within $\mathcal{M}_{\rm{Nat}}(\SO(3))$.  On $\mathcal{C}$ two metrics are isospectral if and only if they have identical length spectra.
%\end{cor}

\begin{rem}\label{rem:CleanBehavior}
It is clear from the proof of Theorem~\ref{thm:CleanMetrics} that the cleanliness of a metric 
is dictated by the behavior of the Poincar\'{e} map along Type II geodesics.
\end{rem}

\begin{proof}[Proof of Theorem~\ref{thm:UncleanMetrics}]
The space of naturally reductive left-invariant metrics on $\SO(3)$ is identified with 
$\mathcal{A} = \{(\alpha, A) : \alpha, A > 0 \} \subset \R^2$.
Now for each $r \in \Q_{+}$ let $\mathcal{A}_{r} \equiv \{ (\alpha, A) : A = r \alpha \}$.
Then it follows from Theorem~\ref{thm:CleanMetrics} that the class of clean metrics in $\mathcal{A}$ is given by $\mathcal{C} = \cap_{r \neq 1 \in \Q_{+}} (\mathcal{A} - \mathcal{A}_r)$, which is a residual set containing the bi-invariant metrics $\mathcal{A}_1$.

The final statement follows from the fact that the normal homogeneous metrics on $\SO(3)$ are identified with the set $\mathcal{N} = \{(\alpha, A) : \alpha \leq A \}$ and, by Theorem~\ref{thm:CleanMetrics}, we see that $\mathcal{N} \cap (\mathcal{A} - \mathcal{C}) \neq \emptyset$.
\end{proof}

%%%%%%%%%%%%%%%%%%%%%%%%%%%%%%
%%%%%%%%%%%%%%%%%%%%%%%%%%%%%%
%\subsection{An Application: Detecting the Moments of Inertia of a Molecule}\label{Sec:Application}
\subsection{The $0$-th Wave Invariant and the Poisson Relation}\label{Sec:Application}
Is the Poisson relation an equality for the clean left-invariant naturally reducitve metrics on $\SO(3)$? In light of the argument employed in Theorem~\ref{thm:MainResult}, it is reasonable to wonder whether one can show that for such metrics the leading term of the trace formula is non-zero for each period $\tau$. We will observe that although this is clear when $|\tau|$ is the length of the shortest non-trivial geodesic, this does not appear to be an easy matter to resolve, in general. We begin by reviewing the construction of the Duistermaat-Guillemin measure as discussed in the appendix of \cite{BPU}.
%In light of the method of proof employed for Theorem~\ref{}, it is natural to wonder whether it is possible to establish that for each \emph{clean} left-invariant naturally reductive metric $g$ of $\SO(3)$ and each $\tau \in \Spec_{L}(\SO(3),g)$ it is possible to resolve whether one of $\Wave_{0}^{\rm{even}}(\tau)$ or $\Wave_{0}^{\rm{odd}}(\tau)$ is non-zero. We show that with the exception of the systole, this is hard to resolve. Ultimately the issue will be the behavior of the Morse index.

%As is shown in \cite{DuGu}, if $\tau$ is a clean length, then each component $\Theta_j$ of $\Fix(\Phi_{\tau})$ admits a canonical positive measure $\mu_{j}^{\tau}$, which we will refer to as the \emph{Duistermaat-Guillemin Measure} (or \emph{density}). We will now review the construction of the Duistermaat-Guillemin measure as discussed in the appendix of \cite{BPU}. 

\begin{DGMeasure}
Let $\tau$ be a clean period of the geodesic flow. For simplicity we will assume that $\Theta = \Fix(\Phi_{\tau})$ is connected and we will let $\widetilde{\Theta} = \{ t X_{p} : X_{p} \in F \mbox{ and } t >0  \}$. We will exploit the symplectic structure of the tangent bundle to construct a canonical measure $\tilde{\mu}^{\tau}$ on $\widetilde{\Theta}$ and obtain a canonical measure on $\Theta$ be dividing by the measure $|dq|$ (in the transverse direction).

Indeed, one can check that $\widetilde{\Theta}$ is a clean fixed point set of $\widetilde{\Phi}_{\tau}$.
Now, let $z \in \widetilde{\Theta}$ and consider $T= Id_{z} - D_{z} \Phi_{\tau} : V \to V$, where 
$V \equiv T_{u} TM$. Following \cite[p. 524-525]{BPU} we can construct a density on $T_{z} \widetilde{\Theta}$ as follows.
\begin{itemize}
\item Let $\mathcal{E} = \{e_1, \ldots , e_{k}\}$ be a basis for $W \equiv T_{z} \widetilde{\Theta}$;
\item Let $W^{\Omega} = \{ v \in V : \Omega(w, v) = 0 \mbox{ for each } w \in W \}$ be the 
$\Omega$-orthogonal complement of $W$ in $V$.% with respect to $\Omega$.
\item Let $\mathcal{F} = \{f_1, f_2, \ldots , f_k \}$ be a basis for a complement of $W^\Omega$ satisfying
$$\Omega(e_{i}, f_{j}) = \delta_{ij}.$$
\item Let $\mathcal{V} = \{v_1, \ldots , v_{2n-k} \}$ be a basis for a complement of $W$ in $V$.
\end{itemize}

With the above notation we have the following lemma.

\begin{lem}[Lemma A.2 \cite{BPU}]\label{lem:BPU} \textrm{}\\

\begin{enumerate}
\item $\ker(T) = W$ and the image of $T$ equals $W^{\Omega}$, so that $T\mathcal{V} \cup \mathcal{F}$ is a basis for $V$. 
\item Let $\phi \in |V|^{1/2}$ be an arbitrary half-density on $V$. Then the $DG$-density $\tilde{\mu}^{\tau}$ on $W \equiv T_{z} \wtTheta$ is given by 
$$\tilde{\mu}^{\tau}(\mathcal{E}) = \frac{\phi(\mathcal{V} \wedge \mathcal{E})}{\phi(T\mathcal{V} \wedge \mathcal{F})} = \frac{1}{|\alpha(u)|^{1/2}},$$
where we abuse notation and have
 $\mathcal{E} = e_1 \wedge \cdots \wedge e_{k}$, $\mathcal{F} = f_1 \wedge \cdots \wedge f_k$, 
$\mathcal{V} = v_1 \wedge \cdots \wedge v-{2n-k}$, $T\mathcal{V} = Tv_1 \wedge \cdots \wedge T v_{2n-k}$ and 
$\alpha(u) \neq 0$ satisfies $T\mathcal{V} \wedge \mathcal{F} = \alpha(z) \mathcal{V} \wedge \mathcal{E}$.
%We note that $T\mathcal{V} \wedge \mathcal{F} \neq 0$.
\end{enumerate}

\end{lem}

It then follows that if we let $\nu_{\tilde{g}\upharpoonright \Theta}$ denote the Riemannian density on $\Theta$ induced by the Sasaki metric $\tilde{g}$ on $TM$ associated to $g$, then the \emph{Duistermaat-Guillemin measure} $\mu^{\tau}$ on $\Theta$ is given by 
$$\mu^{\tau} = \frac{1}{|\alpha|^{1/2}} \nu_{\tilde{g}\upharpoonright \Theta},$$
where for each $z \in \Theta$ the function $\alpha(z)$ is computed as in the preceding lemma.
\end{DGMeasure}

%As an application of Corollary~\ref{cor:CleanMetrics}, we show that the naturally reductive metrics on $\SO(3)$ can be mututally distinguished via their spectra.

%\begin{prop}\label{prop:MomentsOfInertia}
%The left-invariant naturally reductive metrics on $\SO(3)$ can be mutually distinguished via their spectra.
%\end{prop}

%\begin{rem}
%This is a special case of \cite[Theorem ??]{SS}, where we employ heat invariants to demonstrate that homogeneous three-manifolds can be mutually distinguished via their spectra.
%\end{rem}

%\begin{proof}[Sketch of Proof]
%Given a left-invariant naturally reductive metric $g$ on $\SO(3)$ we may use the procedure outlined above to compute the Duistermaat-Guillemin measure associated to $\tau \in \Spec_{L}(\SO(3)),g)$.

\medskip

Using the procedure outlined above we obtain the following.

\begin{lem}\label{lem:DGMeasure}
Let $g= g_{(\alpha, \alpha, A)}$ be a left-invariant naturally reductive metric on $\SO(3)$
with corresponding Sasaki metric $\tilde{g}$ on the tangent bundle.
Let $\tau$ be a clean period of the geodesic flow in the length spectrum of $g$ and $d\mu^{\tau}$ denote the corresponding DG-measure on $\Fix(\Phi_{\tau})$ (see p.~\pageref{DGMeasure}).
%as in Section~\ref{Sec:TraceFormula}.
And, the set $\mathE_{\tau, \alpha,A}$ and function $\sigma(p.q, \alpha, A)$ are as in Theorem~\ref{thm:ClosedGeodesics}.

\begin{enumerate}
\item If $\alpha = A$, then $\Fix(\Phi_{\taumin}) = T^1\SO(3) = \SO(3) \times S^2_1$ and $d\mu^\tau = d\nu_{\tilde{g} \upharpoonright \Fix(\Phi_\tau)}$. That is, $d \mu^{\tau}$ is the Riemannian density on $\Fix(\Phi_\tau)$ induced by the restriction of the Sasaki metric. 
And, we have  
$$\int_{\Fix(\Phi_\tau)} d\mu^{\tau} = \vol(g) \cdot \vol(S^2_1) = 4\pi \vol(g)$$

\item For $\alpha \neq A$ the components of $\Fix(\Phi_\tau)$ are of Type I, II or III (see Lemma~\ref{lem:Submanifolds}).

\begin{enumerate}
\item Suppose $\Theta \subset \Fix(\Phi_\tau)$ is a component of Type I. Then, 
the restriction of the DG-measure to $\Theta$ is given by  
$d\mu^\tau \upharpoonright \Theta \equiv \frac{1}{\sqrt{|\tau|}} d \nu_{\tilde{g} \upharpoonright \Theta}$ and $$\int_{\Theta} d \mu^\tau = \frac{1}{\sqrt{|\tau|}} \vol(g) \vol(S^1_1) = \frac{2\pi}{\sqrt{|\tau|}}\vol(g).$$

\item Suppose $\Theta \subset \Fix(\Phi_\tau)$ is a component of Type II. 
Then, the restriction of the DG-measure to $\Theta$ is given by  
$d\mu^\tau \upharpoonright \Theta \equiv \frac{1}{|\tau|} d \nu_{\tilde{g} \upharpoonright \Theta}$ 
and $$\int_{\Theta} d \mu^\tau = \frac{1}{|\tau|}\vol(g).$$

\item Suppose $\Theta \subset \Fix(\Phi_\tau)$ is a component of Type III, so that 
$\Theta = (G \times K) \cdot v_{(p,q)}$ for $(p,q) \in \mathcal{E}_{\tau, \alpha, A}$ and $v_{(p,q)} \in T_eG$ as in Lemma~\ref{lem:Submanifolds}(3).
Then, the restriction of the DG-measure to $\Theta$ is given by  
$d\mu^\tau \upharpoonright \Theta \equiv \frac{1}{\sqrt{|\tau|}} d \nu_{\tilde{g} \upharpoonright \Theta}$ and $$\int_{\Theta} d \mu^\tau = \frac{2\pi}{\sqrt{|\tau|}} \sqrt{\frac{\sigma(p,q, \alpha, A)}{\sigma(p,q, \alpha, A) + 1}}\vol(g).$$
%where $\sigma(p,q, \alpha, A)$ is as in \ref{}.

\end{enumerate}

\end{enumerate}

\end{lem}

%We may then compute the wave invariants associated to a left-invariant naturally reductive metric
As an application we obtain the wave invariant associated to the systole of a left-invariant naturally reductive metric on $\SO(3)$. The fact that these wave-invariants are non-zero, establishes the ``audibility'' of the systole in this setting.

\begin{lem}\label{lem:WaveInvariants}
Let $g = g_{(\alpha, \alpha, A)}$ be a left-invariant naturally reductive metric on $\SO(3)$. Let $\tau_{\rm{min}} = \tau_{\rm{min}}(g)$ denote the length of the shortest closed geodesic with respect to $g$ and $\sigma$ denote the Morse index of any smooth closed geodesic with respect to $g$ having length $\tau_{\rm{min}}$. %The length $\tau_{\rm{min}}$ is necessarily clean by   .

\begin{enumerate}
\item If $\alpha = A$, then $\taumin = \sqrt{\alpha} \ell_0$ is clean, $\Fix(\Phi_{\taumin})$ is a connected manifold of dimension 5, and all of the closed geodesics of length $\taumin$ have Morse index $\sigma = 0$.
Furthermore, 
$$\Wave_{0}^{\rm{odd}}(\pm \taumin) = -\frac{1}{\pi} \vol(g)$$
%$$\Wave_{0}^{\rm{odd}}(\taumin) = \left(\frac{1}{2\pi i}\right)^2  i^{-\sigma} \frac{4\pi}{\alpha^{1/4} \ell_0^{1/2}} \vol(g) = \frac{16 \alpha^{5/4}}{\ell_0}$$
and $\Wave_{k}^{\rm{even}}(\pm \taumin) = 0 $ for $k \geq 0$.

\item If $A < \alpha $, then $\taumin = \sqrt{A} \ell_0$ is clean, $\Fix(\Phi_{\taumin})$ is a a manifold of dimension 3 having two connected components, and the closed geodesics of length $\taumin$ all have the same Morse index $\sigma$.
Furthermore,
$$\Wave_{0}^{\rm{odd}}(\pm \taumin) = \frac{i^{-(\sigma +1)}}{\pi} \frac{\vol(g)}{\taumin}.$$ 
%$$\Wave_{0}^{\rm{odd}}(\taumin) = \frac{1}{2\pi i}  i^{-\sigma} \frac{4\pi}{\alpha^{1/4} \ell_0^{1/2}} \vol(g) = i^{-(\sigma +1)}\frac{16 \alpha \sqrt{2} \pi}{\ell_0}.$$
$\Wave_{k}^{\rm{even}}(\pm \taumin) = 0 $ for $k \geq 0$.

\item If $A > \alpha$, then $\taumin = \sqrt{\alpha} \ell_0$ is clean,
 $\Fix(\Phi_{\taumin})$ is a connected manifold of dimension 4, 
 and the closed geodesics of length $\taumin$ all have common Morse index $\sigma$.
Furthermore, 
$$\Wave_{0}^{\rm{even}}(\pm \taumin) = \left(\frac{1}{2\pi i}\right)^{3/2}  i^{-\sigma} \frac{2 \pi}{\sqrt{\taumin}} \vol(g) .$$ 
%$$\Wave_{0}^{\rm{even}}(\taumin) = \left(\frac{1}{2\pi i}\right)^{3/2}  i^{-\sigma} \frac{2 \pi}{\sqrt{\taumin}} \vol(g) .$$
$\Wave_{k}^{\rm{odd}}(\pm \taumin) = 0 $ for $k \geq 0$.
\end{enumerate}
Consequently, for each left-invariant naturally reductive metric on $\SO(3)$, $\taumin$ can be recovered from its spectrum.
\end{lem}

\begin{rem}
It is possible to use the $0$-th wave-invariant associated to $\taumin(g)$ to demonstrate that the left-invariant naturally reductive metrics on $\SO(3)$ can be mutually distinguished via their spectra. In terms of physical chemistry this means that for a spherical or symmetric molecule (e.g., methane, benzene or chloromethane) its moments of inertia can be recovered from its rotational spectrum. For a more general discussion see \cite{SS}, where we use the heat trace to prove the following more general statement: If $(M_1, g_1)$ and $(M_2, g_2)$ form a non-trivial isospectral pair of homogeneous three-manifolds, then $M_1$ and $M_2$ are both spherical three-manifolds with non-isomorphic fundamental groups and equipped with Type I metrics. Consequently, the moments of inertia of any molecule can be recovered from its rotational spectrum.
%In \cite{SS} we use the heat trace to prove the following more general statement: homogeneous three-manifolds can be mutually distinguished via their spectra.
\end{rem}

Let $g_{(\alpha, \alpha, A)}$ be a left-invariant naturally reducitve metric on $U = \SO(3)$ with $\alpha \neq A$. Fix an element $\tau \in \LSpec^{\pm}(g_{(\alpha, \alpha, A)})$ and recall that $\Fix(\Phi_\tau)$ consists of components of Type I, II and III. We observe that $\Fix(\Phi_\tau)$ cannot contain components of Type I and Type II simultaneously. Indeed, if this were the case, then we could find natural numbers $m$ and $n$ such that $|\tau| = m \sqrt{\alpha} \ell_0 = n \sqrt{A} \ell_0$, which would imply that $A \in \alpha \Q_{+} - \{\alpha\}$, contradicting the fact that the metric $g_{(\alpha, \alpha, A)}$ is clean (see Theorem~\ref{thm:CleanMetrics}). It is also the case that Type I components cannot occur along with Type III components. For otherwise, there exist natural numbers $m$ and $n$ such that $|\tau| = m \sqrt{\alpha} \ell_0 = n \sqrt{\alpha} \ell_0 (q^2 + p^2 \frac{A}{\alpha - A})^{1/2}$, which implies that $A \in \alpha \Q_{+} - \{\alpha \}$ and leads us to conclude that the metric $g_{(\alpha, \alpha, A)}$ is actually unclean, which is a contradiction. 

Now, let $|\tau| = n \sqrt{\alpha} \ell_0$ be the length of a Type I geodesic. Then, the previous paragraph dictates that $\Fix(\Phi_\tau)$ consists of the lone Type I component. Therefore, since the Type I component is of dimension $4$ we see that $\Wave_{0}^{\rm{even}}(\tau) \neq 0$. Therefore, the length of any Type I geodesic is contained in the singular support of the trace of the wave group.

To analyze periods arising from Type II and Type III geodesics, we recall that the Type II components are $3$-dimensional while the Type III components are $4$-dimensional. Hence, given a period $\tau$ not arising form a Type I geodesic, the Type II components determine $\Wave_{k}^{\rm{odd}}(\tau)$ and Type III components determine $\Wave_{k}^{\rm{even}}(\tau)$.

Let $\tau$ be the period of a Type II orbit of the geodesic flow, then the odd-dimensional part of $\Fix(\Phi_\tau)$ is of the form $(U\times K) \cdot v \cup (U\times K) \cdot -v$, where $v$ is the initial velocoty of some unit speed geodesic of Type II (see Lemma~\ref{lem:Submanifolds}). Then, since the Morse index associated to these components is clearly the same, we conclude that $\Wave_{0}^{\rm{odd}}(\tau)$ is non-zero and, therefore, $\tau$ is also in the singular support of the trace of the wave group.

Now, let $\tau$ be the period of a Type III orbit of the geodesic flow, then the even-dimensional part of $\Fix(\Phi_\tau)$ is a finite union of Type III components of dimension $4$: each component is of the form $(U\times K) \cdot v_{(p,q)} \simeq \SO(3) \times S^1$ (see Lemma~\ref{lem:Submanifolds}). Using Ziller's recasting of the Jacobi equation for naturally reductive metrics \cite{Ziller2} and our explicit understanding of the Poincar\'{e} map, one can show that the conjugate points along a Type III geodesic are as follows.

%If the only odd-dimensional components in $\Fix(\Phi_{\tau})$ are the Type II components $\Theta_{1} = (G \times K) \cdot v$ and $\Theta_{2} = (G \times K) \cdot -v$, where $v$ is the initial velocity of some unit speed Type II geodesic (see Lemma~\ref{lem:Submanifolds}), then since the Morse index associated to these components is clearly the same, we conclude that $\Wave_{0}^{\rm{odd}}(\tau)$ is non-zero and, therefore, $\tau$ is also in the singular support of the trace of the wave group.

%An issue arises in using the $0$-th wave invariants when Type III components occur in $\Fix(\Phi_{\tau})$. Indeed, one can show that the conjugate points along a Type III geodesics are as follows.

\begin{prop}\label{prop:ConjugatePointsTypeIII}
Suppose $\gamma_{v_{(p,q)}}$ is a Type III geodesic with $v_{(p,q)} = c_1 Z_1 + c_2Z_2 + c_3 Z_3 \in \germ{p} \equiv T_{e} U$, where $c_1^2 + c_2^2 = \frac{\sigma(p,q,\alpha,A)}{\sigma(p, q, \alpha, A) +1}$ and $c_3 = \pm \sqrt{\frac{1}{\sigma(p,q, \alpha, A) + 1}}$ for a unique $(p,q) \in \mathcal{E}_{\tau, \alpha, A}$, where $\sigma(p,q,\alpha,A)$ and $\mathcal{E}_{\tau, \alpha,A}$ are defined as in Theorem~\ref{thm:ClosedGeodesics}. And, let 
$$a(p,q, \alpha, A) = \sqrt{\phi^2 + \frac{\sigma(p,q,\alpha,A)}{\sigma(p,q,\alpha,A) + 1}\frac{1}{2(\overline{A} + \alpha)}},$$ where $\phi \equiv \frac{1}{\alpha}\sqrt{\frac{A}{2}}$ and $\overline{A} \equiv \frac{A\alpha}{\alpha - A}$.
\begin{enumerate}
\item If $\alpha > A$, then $\gamma_{v_{(p,q)}}(t_0)$, $t_0 > 0$, is conjugate to $e = \gamma_{v_{(p,q)}}(0)$ along 
$\gamma_{v_{(p,q)}}$ if and only if $t_0 \neq 0 \in \frac{2\pi}{a(p,q, \alpha, A)} \N$. 
And, in this case the conjugate point has multiplicity one.
\item If $\alpha < A$, then $\gamma_{v_{(p,q)}}(t_0)$, $t_0 > 0$, is conjugate to $e = \gamma_{v_{(p,q)}}(0)$ along 
$\gamma_{v_{(p,q)}}$ if and only if $t_0 \neq 0 \in \frac{2\pi}{a(p,q, \alpha, A)} \N$ or 
$t_0 = \frac{4 \alpha^2}{(A - \alpha)\frac{\sigma(p,q,\alpha,A)}{\sigma(p,q,\alpha,A) + 1}}$. 
And, in this case the conjugate point has multiplicity one.
\end{enumerate}
\end{prop}  

%\begin{proof}
%We omit the long computation, which makes use of Ziller's recasting of the Jacobi equation for 
%naturally reductive metrics \cite{Ziller2} and our explicit understanding of the Poincar\'{e} map.
%\end{proof}
\noindent
Using the previous proposition one can compute the Morse index associated to each 
Type III component $(U \times K) \cdot v_{(p,q)}$ contained in $\Fix(\Phi_{\tau})$. 
This, in conjunction with computations in the spirit of those used to establish 
Lemma~\ref{lem:DGMeasure}(2c), allows one to compute the contribution of each 
Type III component to the wave invariant $\Wave_{0}^{\rm{odd}}(\tau)$. 
However, some inspection will demonstrate that these contributions behave 
rather erratically making it difficult to rule out the possibility of cancellations, in general. 
Therefore, at the moment, the best we can say is the following.

\begin{prop}\label{prop:PoissonRelationNatReductive}
Let $g = g_{(\alpha, \alpha, A)}$ be a left-invariant naturally reductive metric on $\SO(3)$, then 
$$\{\pm n \sqrt{\alpha}\ell_0,  \pm n' \sqrt{\alpha}\ell_0 : n, n' \geq 0 \} \subseteq \SingSupp(\operatorname{Tr}(e^{-it\sqrt{\Delta_g}})) \subseteq \Spec_L^{\pm}(g).$$
In particular, $\taumin(g)$ can be recovered from the Laplace spectrum.
%In particular, the systole of $g$, $\tau_{\rm{min}}$, can be recovered from the Laplace spectrum.
\end{prop}
 
 \noindent
Consequently, it is unclear whether the wave group of a left-invariant naturally reductive metric on $\SO(3)$ is singular at periods of the geodesic flow that can only arise from Type III orbits.
 
%\begin{prop}
%Let $g_{(\alpha, \alpha, A)}$ be a clean left-invariant naturally reductive metric on $\SO(3)$ and $\tau$ an element in the length spectrum of $g_{(\alpha, \alpha, A)}$. If $\tau$ is a multiple of $\sqrt{\alpha}\ell_0$, or $\tau$ is a multiple of $\sqrt{A}$ for which $\Fix(\Phi_{\tau})$ contains no Type III components, then $\tau$ is in the singular support of the trace of the wave group of $g_{(\alpha, \alpha, A)}$. 
%\end{prop}

%%%%%%%%%%%%%%%%%%%%%%
%%%%%%%%%%%%%%%%%%%%%%
%%%%%%%%%%%%%%%%%%%%%%
\appendix

%%%%%%%%%%%%%%%%%%%%%%%%%%%%%%%%%%%%%%%%%%
%%%%%%%%%%%%%%%%%%%%%%%%%%%%%%%%%%%%%%%%%%
%%%%%%%%%%%%%%%%%%%%%%%%%%%%%%%%%%%%%%%%%%
%%%%%%%%%%%%%%%%%%%%%%%%%%%%%%%%%%%%%%%%%%
\section{Abstract Root Systems}\label{Sec:RootSystems}
 
We compute the co-root lattice, central lattice and lowest dominant forms associated to the indecomposable root systems. We also comopute the integral lattice associated to each simple Lie group. This data is used in the proof of Theorem~\ref{thm:IntroMod4} in Section~\ref{sec:HearingLengthSpectrum}. References for some aspects of this section include \cite[Chp. 10.3]{Helgason}, \cite[Chp. V.6]{BtD}, \cite[Chp. 7.3]{Loos}, \cite[Chp. 2.14]{Samelson}, \cite[Chp. III]{Humphreys} and \cite{Adams}.

%%%%%%%%%%%%%%%%%%%%%%%%%
\subsection{Basic Definitions}\label{SubSec:BasicDefinitions}
Let $V$ be a finite dimensional vector space over $\R$ and let $V^{*}$ denote its dual space.
A \emph{root  system} associated to $V$ is a finite subset $R \subset V^{*}$ having the following properties
\begin{enumerate}
\item $R$ spans $V^*$;
\item There is a map $\lambda \mapsto \lambda\check{}$ from $R$ to $V$ such that 
for any $\alpha, \beta \in R$ we have
\begin{enumerate}
\item $\alpha ( \alpha\,\check{}\, ) = 2$
\item $\beta(\alpha\,\check{}\,) \in \Z$;
\item $\beta - \beta(\alpha\,\check{}\,)\alpha \in R$
\end{enumerate}
\end{enumerate}
In the event that $c \alpha \in R$ implies $c = \pm 1$ for any choice of $c \in \R, \alpha \in R$ we say that 
the root system is \emph{reduced}; otherwise, we say the root system is \emph{non-reduced} 
in which case $c $ is restricted to the values $\pm \frac{1}{2}, \pm 1, \pm 2$.
The dimension of $V$ is the \emph{rank} of the root system.
For each $\alpha \in R$, $\alpha\,\check{}$ is called the \emph{co-root} of $\alpha$ and the set 
$R\,\check{} \subset V \simeq V^{**}$ consisting of the co-roots is a root system associated to $V^{*}$ known as the \emph{co-root system}. 
%(It can be verified that $R\,\check{} \subset V \simeq V^{**}$ is a root system associated to $V^{*}$.)

Associated to each $\alpha \in R$ we have the \emph{reflection through the root $\alpha$} 
which is defined to be the isomorphism $S_\alpha : V \to V$ given by $S_{\alpha}(X) = X - \alpha(X) \alpha\,\check{}$.
The map $S_\alpha$ fixes the hyperplane $\ker (\alpha)$ and sends $\alpha\,\check{}$ to $-\alpha\,\check{}$.
The \emph{Weyl group} associated to $R$ is the finite group $W$ generated by the $S_{\alpha}$'s. 
If for any $\lambda \in V^*$ we define $S_{\alpha} \cdot \lambda \equiv \lambda \circ S_\alpha^{-1} = \lambda \circ S_\alpha$,
then it follows from condition 2(c) that $W$ acts on $R$ via permutations. 
We also recall that if $\inner$ is a $W$-invariant inner product on $V$ and if for each $\lambda \in V^*$ we let 
$\lambda^* \in V$ be the unique vector such that $\langle \lambda^*, \cdot \rangle = \lambda( \cdot )$, 
then $\alpha\check{} = \frac{2 \alpha^*}{\langle \alpha^*, \alpha^* \rangle}$. 
It follows that for any $\beta, \alpha \in R$ we have $\beta(\alpha\check{} \,) = 0$ if and only if 
$\langle \beta^*, \alpha^* \rangle = 0$, we therefore agree to say that two roots $\alpha , \beta \in R$ 
are \emph{orthogonal} if $\beta(\alpha\check{}\,) = 0$. 

A vector $v \in V$ is said to be \emph{regular} if $\alpha(v) \neq 0$ for every $\alpha \in R$; otherwise,
we will say the vector is \emph{singular}.
The \emph{Weyl chambers} associated to the root system $R$ are the connected components of 
$V - \cup_{\alpha \in R} \ker(\alpha)$ and a choice of Weyl chamber $\mathcal{C}$ allows us to partition
$R$ into two disjoint sets:
$$R^{+} = R^{+}(\mathcal{C}) \equiv \{ \alpha \in R : \alpha \upharpoonright \mathcal{C} > 0 \}$$
and
$$R^{-} = R^{-}(\mathcal{C}) \equiv \{ \alpha \in R : \alpha \upharpoonright \mathcal{C} < 0 \}.$$
The roots in $R^{+}$ (resp. $R^{-}$) are known as the \emph{positive} (resp. \emph{negative}) \emph{roots}
of $R$ with respect to $\mathcal{C}$. An element $\alpha \in R^{+}$ is said to be \emph{decomposable} 
if there are $\alpha_1, \alpha_2 \in R^{+}$ such that $\alpha = \alpha_1 + \alpha_2$; otherwise, we say that $\alpha \in R^{+}$ 
is \emph{indecomposable}.

\begin{dfn}
Let $R$ be a root system and $\mathcal{C}$ an associated Weyl chamber. 
A subset $B \subset R^{+} = R^{+}(\mathcal{C})$ is said to be a (\emph{positive}) \emph{basis} for $R$
if the following conditions are met:
\begin{enumerate}
\item $B$ is a vector space basis of $V^*$;
\item The coordinates of each root $\alpha \in R$ with respect to the basis $B$ 
are all non-negative integers or all non-positive integers. 
\end{enumerate}
The elements of $B$ are called \emph{simple roots}.
\end{dfn}

\begin{thm}[\cite{Humphreys} Sec. 10]\label{thm:Basis}
Every root system has a positive basis. In fact, if $R$ is a root system and $\mathcal{C}$ is an associated Weyl chamber,
the set $B^+(\mathcal{C})$ consisting of the indecomposable elements in $R^{+}(\mathcal{C})$ 
forms a positive basis and all positive bases of $R$ arise in this manner.
\end{thm}

\begin{cor}\label{cor:Basis}
Let $B$ be a positive basis of the root system $R$, 
then the set $B\,\check{} = \{ \alpha\,\check{} : \alpha \in B \}$
is a positive basis for $R\,\check{}$.
\end{cor}

%And, we remind ourselves that a positive basis of $R$ is a subset $B \subset R^{+}$ satisfying the following conditions:
%\begin{enumerate}
%\item $B$ is a vector space basis of $V^*$;
%\item The coordinates of each root $\alpha \in R$ with respect to the basis $B$ 
%are all non-negative integers or all non-positive integers. 
%\end{enumerate}
%The elements of $B$ are called \emph{simple roots}.

Associated to any root system $R \subset V^*$ there are two important lattices in $V$.

\begin{dfn}
Let $V$ be a vector space and $R \subset V^*$ a root system with co-root system$R\,\check{} \subset V$.
\begin{enumerate}
\item The \emph{co-root lattice} associated to $R$ is the lattice 
$\Lambda_{R\,\check{}}$ generated by the co-roots.
\item The \emph{central lattice} associated to $R$ is the lattice
$\Lambda_Z  = \{ v \in V : \alpha(v) \in \Z \mbox{ for all } \alpha \in R \}$.
\end{enumerate} 
Clearly, $\Lambda_{R\,\check{}} \subseteq \Lambda_Z$ and in light of Corollary~\ref{cor:Basis} we have 
$\Lambda_{R\, \check{}} =  \langle B\, \check{} \, \rangle$, for any positive basis $B$.
\end{dfn}

Given a root system $R \subset V^*$ of $V$ with positive roots $R^+$ corresponding to some choice of Weyl chamber, 
the \emph{lowest strongly dominant form} is the element 
$\rho \equiv \frac{1}{2} \sum_{\alpha \in R^+} \alpha$ (i.e., \emph{half} the sum of the positive roots).
This element enjoys the following well-known property which we will find useful in our proof of 
Proposition~\ref{prop:Component Dependence}.
%This element is important in the representation theory of Lie algebras and it has the following property
%that we will find useful in our proof of Theorem~\ref{thm:PoissonRelationNonSimple}.

\begin{lem}\label{lem:Rho}
$\rho$ is integer valued on the co-root lattice.
In fact, $\rho$ assumes the value $1$ on each element of $B\,\check{}$.
\end{lem}

\begin{proof}
Let $B$ be a basis for $R$. Then, since $B\, \check{}$ is a basis for $R\,\check{}$, 
we only need to check the value of $\rho$ on $B\,\check{}$.
Now, we recall that for any $\alpha \in R$, the reflection $S_{\alpha}$ 
permutes the elements in $R^{+} - \{\alpha\}$ and sends $\alpha$ to $-\alpha$.
Therefore, $S_\alpha \cdot \rho = \rho - \alpha$.
On the other hand, by definition, $S_{\alpha} \cdot \rho = \rho - \rho(\alpha\,\check{} \,) \alpha$.
Therefore, since $B$ is a basis, we conclude that $\rho(\alpha\,\check{}\,) = 1$.
\end{proof}

%%%%%%%%%%%%%%%%%%%%%%%%%
\subsection{Classification of Indecomposable Root Systems}\label{SubSec:Classification}
Given a vector space $V$ a root system $R \subset V^*$ is said to be \emph{decomposable} if it can be written as the disjoint union of two non-empty,
orthogonal sets $R_1$ and $R_2$; otherwise, we will say that $R$ is \emph{indecomposable}.
In the event that $R$ is decomposable with orthogonal decomposition $R = R_1 \cup R_2$,
if we let $V_j^*$ be the span of $R_j$, for $j = 1, 2$, then
$V^*= V_1^* \oplus V_2^*$, where $V_j^{*} = \Span_{\R}(R_j)$ for $j = 1,2$, 
and one easily sees that $R_j$ is a root system of $V_j^*$, for $j = 1,2$.
The indecomposable root systems have been classified up to isomorphism.
As it will prove useful in Section~\ref{sec:HearingLengthSpectrum} and all of this information does not appear in 
one convenient location in the literature, we now give an explicit description of these 
root systems along with their co-root lattices, central lattices, center and lowest strongly dominant forms.
%There are the five infinite families denoted by $A_n$, $B_n$, $C_n$, $D_n$ and $BC_n$, where $n$ denotes the rank.
%The root systems $A_n$, $B_n$, $C_n$ and $D_n$ are reduced, while the root system 
%$BC_n$ is the unique non-reduced indecomposable root system of rank $n$.
%Collectively, these root systems are known as the ``classical root systems.''
%In addition to the classical root systems there are five exceptional (reduced) root systems
%denoted by $G_2$, $F_4$, $E_6$, $E_7$ and $E_8$,
%where once again the subscript denotes the rank. 
%As the explicit description of these root systems will be useful later in the paper,
%we pause to catalogue this information.  
Throughout we will let $e_1, \ldots , e_n$ be the standard basis for $\R^n$ and $\epsilon_1, \ldots , \epsilon_n$ 
will denote the standard dual basis: $\epsilon_k (e_j) = \delta_{jk}$.
We also let $\inner$ denote the standard inner product on $\R^n$, so that $\epsilon_j^* = e_j$ for $j = 1, \ldots , n$.
In all cases our vector space $V$ will be a subspace of the inner product space $(\R^n, \inner)$ and $\inner$ will be invariant under the action of the Weyl group. 

%\medskip

%\noindent
%{\bf Type $A_n$:} As a vector space we have $V = \{ v \in \R^{n+1} : \sum_{j = 1}^{n+1} \epsilon_j (v) = 0 \}$.
\subsubsection{Type $A_n$}\label{sec:TypeAn} As a vector space we have $V = \{ v \in \R^{n+1} : \sum_{j = 1}^{n+1} \epsilon_j (v) = 0 \}$.
The corresponding root system is $R = \{ (\epsilon_{\mu} - \epsilon_{\nu}) \upharpoonright V : 1 \leq \mu \neq \nu \leq n+1 \}$
with co-root system $R\,\check{} = \{ e_\mu - e_\nu : 1 \leq \mu \neq \nu \leq n+1 \}$.
For a Weyl chamber we choose $\mathcal{C} = \{ v \in V : \epsilon_\nu (v) > \epsilon_{\nu+1}(v), 1\leq \nu \leq n \}$,
which gives the corresponding positive roots 
$R^{+} = \{ (\epsilon_{\mu} - \epsilon_{\nu}) \upharpoonright V : 1 \leq \mu < \nu \leq n+1 \}$
and positive basis $B = \{ (\epsilon_{\nu} - \epsilon_{\nu +1}) \upharpoonright V : 1 \leq \nu \leq n \}$.
The co-root lattice, central lattice, center and sum of the positive roots (i.e., $2\rho$) of this root system are given by 
\begin{itemize}
\item $\Lambda_{R\,\check{}} = \langle e_{\mu} - e_\nu : 1 \leq \mu \neq \nu \leq n+1 \rangle$
\item $\Lambda_Z = \langle L_j \equiv \frac{n}{n+1} e_j - \frac{1}{n+1} \sum_{\stackrel{k=1}{k \neq j}}^{n+1} e_k : 1 \leq j \leq n \rangle 
= \langle L_1, e_\mu - e_\nu: 1 \leq \mu < \nu \leq n+1 \rangle$
\item $Z = \Lambda_Z / \Lambda_{R\,\check{}} = \langle \overline{L}_1 \rangle \simeq \Z_{n+1}$
\item $2\rho \equiv \sum_{\alpha \in R^+} \alpha = \sum_{\mu =1}^{n} (2n-2\mu +2) \epsilon_\mu$, since $\epsilon_{n+1} = -\sum_{\mu = 1}^{n} \epsilon_{j}$
\end{itemize}
%And, the corresponding compact simple Lie groups are of the form $\Gamma \backslash \SU(n+ 1)$,
%where $\Gamma$ is a subgroup of $Z(\SU(n+1)) \simeq Z$, the center of $\SU(n+1)$.

The corresponding simply-connected compact Lie group is $\SU(n+1)$ and all other Lie groups of type $A_n$ 
are of the form $U = \SU(n+1) / \Gamma$, where $\Gamma \leq Z(\SU(n+1)) \simeq \Z_{n+1}$. 
The integral lattice of $U$ with respect to any bi-invariant metric is given by
\begin{equation}\label{eqn:IntegralLatticeA_n}
\Lambda_{I}(\SU(n+1)/\Gamma) = \langle k L_1\rangle +  \Lambda_{R\,\check{}} \, ,
\end{equation}
where $k = 0, 1, \ldots, n$ is the smallest generator of $\Gamma$.

%\noindent
%{\bf Type $B_n$:} As a vector space we have $V = \R^n$.
\subsubsection{Type $B_n$}\label{sec:TypeBn} As a vector space we have $V = \R^n$.
The corresponding root system is $R = \{ \pm \epsilon_\mu , \pm \epsilon_{\mu}  \pm \epsilon_{\nu} : 1 \leq \mu \neq \nu \leq n, \pm \mbox{ independent } \}$
with co-root system $R\,\check{} = \{ \pm 2 e_\mu , \pm e_\mu \pm e_\nu : 1 \leq \mu \neq \nu \leq n , \pm \mbox{ independent } \}$.
For a Weyl chamber we choose 
$\mathcal{C} = \{ v \in V : \epsilon_\nu (v) > \epsilon_{\nu+1}(v), 1\leq \nu \leq n-1, \epsilon_n(v) > 0 \}$,
which gives the corresponding positive roots 
$R^{+} = \{ \epsilon_\mu: 1 \leq \mu \leq n\} \cup \{ \epsilon_{\mu} \pm \epsilon_{\nu}: 1 \leq \mu < \nu \leq n\}$
and positive basis $B = \{ \epsilon_{\nu} - \epsilon_{\nu +1}, \epsilon_n: 1 \leq \nu \leq n-1 \}$.
The co-root lattice, central lattice, center and sum of the positive roots (i.e., $2\rho$) of this root system are given by 
\begin{itemize}
\item $\Lambda_{R\,\check{}} = \langle 2e_\mu, e_{\mu} \pm e_\nu : 1 \leq \mu \neq \nu \leq n \rangle$
\item $\Lambda_Z = \langle e_1, \ldots , e_n \rangle = \langle e_1, e_1 - e_2, \ldots , e_1 - e_n \rangle$
\item $Z = \Lambda_Z / \Lambda_{R\,\check{}} = \langle \bar{e}_1 \rangle \simeq \Z_{2}$
\item $2\rho \equiv \sum_{\alpha \in R^+} \alpha = \sum_{\nu =1}^{n} (2n - 2\nu +1) \epsilon_\nu$
\end{itemize}
%And, the corresponding compact simple Lie groups are of the form $\Gamma \backslash \Spin(2n + 1)$,
%where $\Gamma$ is a subgroup of $Z(\Spin(2n +1)) \simeq Z$, the center of $\Spin(2n + 1)$.

The corresponding simply-connected compact Lie group is $\Spin(2n+1)$ and all Lie groups of
type $B_n$ are of the form $U = \Spin(2n+1)/ \Gamma$, where 
$\Gamma \leq Z(\Spin(2n+1)) = \langle \overline{e}_1 \rangle \simeq \Z_2$. 
The integral lattice of $U$ with respect to any bi-invariant metric is given by 
\begin{equation}\label{eqn:IntegralLatticeB_n}
\Lambda_{I}(\Spin(2n+1)/ \Gamma) = \left\{ \begin{array}{ll}
\Lambda_{R\,\check{}} & \Gamma \mbox{ trivial} \\
\Lambda_Z  & \Gamma \simeq \Z_2
\end{array}
\right.
\end{equation}

%\noindent
%{\bf Type $C_n$:} As a vector space we have $V = \R^n$.
\subsubsection{Type $C_n$}\label{sec:TypeCn} As a vector space we have $V = \R^n$.
The corresponding root system is $R = \{ \pm 2 \epsilon_\mu , \pm \epsilon_{\mu}  \pm \epsilon_{\nu} : 1 \leq \mu \neq \nu \leq n
, \pm \mbox{ independent } \}$
with co-root system $R\,\check{} = \{ \pm e_\mu , \pm e_\mu \pm e_\nu : 1 \leq \mu \neq \nu \leq n , \pm \mbox{ independent } \}$.
For a Weyl chamber we choose 
$\mathcal{C} = \{ v \in V : \epsilon_1 (v) > \epsilon_{2}(v) > \cdots > \epsilon_n(v) > 0 \}$,
which gives the corresponding positive roots 
$R^{+} = \{ 2 \epsilon_\mu: 1 \leq \mu \leq n\} \cup \{ \epsilon_{\mu} \pm \epsilon_{\nu}: 1 \leq \mu < \nu \leq n \}$
and positive basis $B = \{ \epsilon_1 - \epsilon_2, \cdots , \epsilon_{n-1} - \epsilon_n, 2 \epsilon_n \}$.
The co-root lattice, central lattice, center and sum of the positive roots (i.e., $2\rho$) of this root system are given by 
\begin{itemize}
\item $\Lambda_{R\,\check{}} = \langle e_1, \ldots , e_n \rangle$
\item
$\Lambda_Z = \langle e_1, \ldots , e_n, F \equiv \frac{1}{2} \sum_{\mu = 1}^{n} e_\mu \rangle
= \{(\frac{c_1}{2}, \ldots , \frac{c_n}{2}) : c_j \in \Z \mbox{ and } c_j \equiv c_i \mod 2 \}$
\item $Z = \Lambda_Z / \Lambda_{R\,\check{}} = \langle \overline{F} \rangle \simeq \Z_{2}$
\item $2\rho \equiv \sum_{\alpha \in R^+} \alpha = \sum_{\nu=1}^{n} 2(n-\nu +1) \epsilon_\nu$
\end{itemize}
%And, the corresponding compact simple Lie groups are of the form $\Gamma \backslash \Sp(n)$,
%where $\Gamma$ is a subgroup of $Z(\Spin(n)) \simeq Z$, the center of $\Sp(n)$.

The corresponding simply-connected Lie group is $\Sp(n)$ and all other Lie groups of type $C_n$
are of the form $U = \Sp(n) / \Gamma$, where 
$\Gamma \leq Z(\Sp(n)) = \langle \overline{F} \rangle  \simeq \Z_2$. 
The integral lattice of $U$ with repeat to any bi-invraiant metric is given by 

\begin{equation}\label{eqn:IntegralLatticeC_n}
\Lambda_{I}(\Sp(n)/ \Gamma) = \left\{ \begin{array}{ll}
\Lambda_{R\,\check{}} & , \; \mbox{ when } \Gamma \simeq 1 \\
\Lambda_Z  & , \; \Gamma = \langle \overline{F} \rangle  \simeq \Z_2
\end{array}
\right.
\end{equation}

%\noindent
%{\bf Type $D_n$:} As a vector space we have $V = \R^n$.
\subsubsection{Type $D_n$}\label{sec:TypeDn} As a vector space we have $V = \R^n$.
The corresponding root system is $R = \{ \pm \epsilon_{\mu}  \pm \epsilon_{\nu} : 1 \leq \mu \neq \nu \leq n 
, \pm \mbox{ independent } \}$
with co-root system $R\,\check{} = \{ \pm e_\mu \pm e_\nu : 1 \leq \mu \neq \nu \leq n , \pm \mbox{ independent } \}$.
For a Weyl chamber we choose 
$\mathcal{C} = \{ v \in V : \epsilon_1 (v) > \epsilon_{2}(v) > \cdots > \epsilon_{n-1}(v) > |\epsilon_{n}(v)| \}$,
which gives the corresponding positive roots 
$R^{+} = \{ \epsilon_{\mu} \pm \epsilon_{\nu}: 1 \leq \mu < \nu \leq n \}$
and positive basis $B = \{ \epsilon_1 - \epsilon_2, \cdots , \epsilon_{n-1} - \epsilon_n, \epsilon_{n-1} + \epsilon_n\}$.
The co-root lattice, central lattice, center and sum of the positive roots (i.e., $2\rho$) of this root system are given by 
\begin{itemize}
\item $\Lambda_{R\,\check{}} = \langle e_{\mu} \pm e_\nu : 1 \leq \mu < \nu \leq n \rangle$
\item
$\Lambda_Z = \langle e_1, e_1 - e_2, \ldots , e_1 - e_n, F \equiv \frac{1}{2} \sum_{\mu = 1}^{n} e_\mu \rangle 
= \langle e_1, \ldots, e_n , F \rangle$
\item 
$$Z = \Lambda_Z / \Lambda_{R\,\check{}} = \langle \bar{e}_1, \overline{F} \rangle \simeq 
\left\{ \begin{array}{cc}
\Z_4 & \mbox{$n$ odd}\\ %2n \equiv 2 \mod 4 \\
\Z_2 \oplus \Z_{2} &\mbox{$n$ even} % 2n \equiv 0 \mod 4
\end{array}
\right.$$
\item $2\rho \equiv \sum_{\alpha \in R^+} \alpha = \sum_{\nu =1}^{n} 2(n-\nu) \epsilon_\nu$
\end{itemize}
%And, the corresponding compact simple Lie groups are of the form $\Gamma \backslash \Spin(2n)$,
%where $\Gamma$ is a subgroup of $Z(\Spin(2n)) \simeq Z$, the center of $\Spin(2n)$.

The corresponding simply-connected Lie group is $\Spin(2n)$ and all other groups 
of type $D_n$ are of the form $U = \Spin(2n) /\Gamma$, 
where 
$$\Gamma \leq Z(\Spin(2n)) \simeq \left\{ \begin{array}{cc}
\Z_4 & \mbox{$n$ odd} \\
\Z_2 \oplus \Z_{2} & \mbox{$n$ even}
\end{array}
\right.$$ 

\noindent 
In the event that $2n \equiv 2 \mod 4$, the integral lattice of $U$ with respect to a bi-invariant metric is 

\begin{equation}\label{eqn:IntegralLatticeD_nOdd}
\Lambda_I(\Spin(2n)/\Gamma) = \left\{ \begin{array}{cc}
\Lambda_{R\,\check{}} & \Gamma \simeq 1 \\
\langle 2F \rangle + \Lambda_{R\,\check{}} & \Gamma = \langle 2\overline{F} \rangle \simeq \Z_2 \\
\langle F \rangle + \Lambda_{R\,\check{}} & \Gamma = \langle \overline{F} \rangle \simeq \Z_4
\end{array}
\right.
\end{equation}

\noindent
For $2n \equiv 0 \mod 4$, the integral lattice of $U$ with respect to a bi-invariant metric is 

\begin{equation}\label{eqn:IntegralLatticeD_nEven}
\Lambda_I(\Spin(2n)/\Gamma) = \left\{ \begin{array}{cl}
\Lambda_{R\,\check{}} & \Gamma \simeq 1 \\
\langle e_1 \rangle + \Lambda_{R\,\check{}} & \Gamma = \langle \overline{e}_1 \rangle \simeq \Z_{2} \oplus 1\\
\langle F \rangle + \Lambda_{R\,\check{}} & \Gamma = \langle \overline{F} \rangle \simeq 1 \oplus \Z_2 \\
\langle e_1 + F \rangle + \Lambda_{R\,\check{}} & \Gamma = \langle \overline{e_1 + F} \rangle \simeq \Z_2\\
\Lambda_Z & \Gamma = \langle \overline{e}_1 , \overline{F} \rangle \simeq \Z_2 \oplus \Z_2
\end{array}
\right.
\end{equation}

%\noindent
%{\bf Type $BC_n$:} As a vector space we have $V = \R^n$.
\subsubsection{Type $BC_n$}\label{sec:TypeBCn} As a vector space we have $V = \R^n$.
The corresponding root system is a non-reduced root system which is
the union of the root systems of type $B_n$ and $C_n$:  
$R = \{ \pm \epsilon_\mu, \pm 2 \epsilon_\mu , \pm \epsilon_{\mu}  \pm \epsilon_{\nu} : 1 \leq \mu \neq \nu \leq n \}$
with co-root system $R\,\check{} = \{ \pm 2 \epsilon, \pm e_\mu , \pm e_\mu \pm e_\nu : 1 \leq \mu \neq \nu \leq n \}$.
For a Weyl chamber we choose 
$\mathcal{C} = \{ v \in V : \epsilon_1(v) > \epsilon_2(v) > \cdots \epsilon_n(v) >0 \}$,
which gives the corresponding positive roots 
$R^{+} = \{ \epsilon_j, 2 \epsilon_j: 1 \leq j \leq n\} \cup \{ \epsilon_i \pm \epsilon_j : 1 \leq i < j \leq n\}$
and positive basis $B = \{ \epsilon_1 - \epsilon_2, \ldots , \epsilon_{n-1} - \epsilon_n, 2 \epsilon_n \}$.
The co-root lattice, central lattice, center and sum of the positive roots are given by
\begin{itemize}
\item $\Lambda_{R\,\check{}} = \Lambda_{R\,\check{}}^{C_n} = \langle e_1, \ldots , e_n \rangle$
\item
$\Lambda_Z = \Lambda_Z^{B_n} =  \langle e_1, \ldots , e_n \rangle$
\item $Z = \Lambda_Z / \Lambda_{R\,\check{}} = 1$
\item $2 \rho \equiv \sum_{\alpha \in R^+} \alpha = \sum_{j = 1}^{n} (2(n-j) +3) \epsilon_j$
\end{itemize}

%\noindent
%{\bf Type $F_4$:} As a vector space we have $V = \R^4$.
\subsubsection{Type $F_4$}\label{sec:TypeF4} As a vector space we have $V = \R^4$.
The corresponding root system contains the roots coming from $B_4$:  
$R = \{ \pm \epsilon_\mu, \pm \epsilon_{\mu}  \pm \epsilon_{\nu} : 1 \leq \mu \neq \nu \leq 4 , \pm \mbox{ independent } \} \cup 
\{ \frac{1}{2} \sum_{\mu =1}^{4} \pm \epsilon_\mu : \pm \mbox{ independent }\}$
with co-root system $R\,\check{} = \{ \pm 2 e_\mu, \pm e_{\mu}  \pm e_{\nu} : 1 \leq \mu \neq \nu \leq 4 , \pm \mbox{ independent } \} \cup 
\{ \frac{1}{2} \sum_{\mu =1}^{4} \pm e_\mu : \pm \mbox{ independent } \}$.
For a Weyl chamber $\mathcal{C}$ we choose the component of 
$V - \cup_{\alpha \in R} \ker(\alpha)$ containing the regular vector $(8,3,2,1)$.
Then the positive roots are
$R^{+} = \{ \epsilon_\mu : 1 \leq \mu \leq n\} \cup \{ \epsilon_{\mu}  \pm \epsilon_{\nu} : 1 \leq \mu \leq \nu \leq 4 \} \cup 
\{\frac{1}{2}(\epsilon_1 \pm \epsilon_2 \pm \epsilon_3 \pm \epsilon_4) : \pm \mbox{ independent } \}$
and positive basis $B = \{ \alpha_1 = \frac{1}{2}(\epsilon_1 - \epsilon_2 - \epsilon_3 - \epsilon_4), \alpha_2 = \epsilon_4, 
, \alpha_3 = \epsilon_3 - \epsilon_4, \alpha_4 = \epsilon_2 - \epsilon_3 \}$.
%$R^{+} = \{ \epsilon_\mu, \epsilon_{\mu}  \pm \epsilon_{\nu} : 1 \leq \mu \neq \nu \leq 4 \} \cup 
%\{\frac{1}{2}(\epsilon_1 + \epsilon_2 + \epsilon_3 + \epsilon_4),  \frac{1}{2}(\epsilon_1 + \epsilon_2 + \epsilon_3 - \epsilon_4),
%\frac{1}{2}(\epsilon_1 + \epsilon_2 - \epsilon_3 + \epsilon_4), \frac{1}{2}(\epsilon_1 + \epsilon_2 - \epsilon_3 - \epsilon_4),
%\frac{1}{2}(\epsilon_1 - \epsilon_2 + \epsilon_3 + \epsilon_4),  \frac{1}{2}(\epsilon_1 - \epsilon_2 + \epsilon_3 - \epsilon_4),
%\frac{1}{2}(\epsilon_1 - \epsilon_2 - \epsilon_3 + \epsilon_4), \frac{1}{2}(\epsilon_1 - \epsilon_2 - \epsilon_3 - \epsilon_4)\}$.
%and positive basis $B = \{ \epsilon_1 - \epsilon_2, \ldots , \epsilon_{n-1} - \epsilon_n, 2 \epsilon_n \}$.
The co-root lattice, central lattice, center and sum of the positive roots are given by 
\begin{itemize}
\item $\Lambda_{R\,\check{}} = \langle \frac{1}{2} \sum_{\mu =1}^{4} \pm e_\mu, \pm 2 e_\mu, \pm e_{\mu}  \pm e_{\nu} : 1 \leq \mu \neq \nu \leq 4 \rangle$
\item
$\Lambda_Z = \Lambda_{R\,\check{}}$
\item $Z = \Lambda_Z / \Lambda_{R\,\check{}} \simeq 1$
\item $2\rho \equiv \sum_{\alpha \in R^+} \alpha = 15 \epsilon_1 + 5 \epsilon_2 + 3 \epsilon_3 + \epsilon_4$
\end{itemize}
%And, the unique compact simple Lie group associated to this root system will also be denoted by $F_4$.

The corresponding simply-connected compact Lie group is also denoted by $F_4$ and it is the unique group 
of this type. The integral lattice of this group with respect to any bi-invariant metric given by
\begin{equation}\label{eqn:IntegralLatticeF_4}
\Lambda_I(F_4) =  \Lambda_{R\,\check{}}.
\end{equation}

%\noindent
%{\bf Type $G_2$:} As a vector space we have $V = \{ v \in \R^3 : \epsilon_1(v) + \epsilon_2(v) + \epsilon_3(v) = 0 \}$.
\subsubsection{Type $G_2$}\label{sec:TypeG2}
As a vector space we have $V = \{ v \in \R^3 : \epsilon_1(v) + \epsilon_2(v) + \epsilon_3(v) = 0 \}$.
The corresponding root system is given by 
$R = \{ (\epsilon_\mu - \epsilon_\nu) \upharpoonright V : 1 \leq \mu \neq \nu \leq 3\} \cup 
\{ \pm \epsilon_\mu \upharpoonright V : 1 \leq \mu \leq 3\}$, which contains the roots of $A_2$,
and the associated co-root system is given by 
$R\,\check{} = \{ e_\mu \pm e_\nu : 1 \leq \mu < \nu \leq 3 \} \cup \{ \pm (2e_1-e_2 -e_3), \pm (-e_1 +2e_2-e_3), \pm (-e_1-e_2 + 2e_3) \}$.
For a Weyl chamber we choose the component $\mathcal{C}$ of $V - \cup \ker(\alpha)$ 
containing the regular vector $(3, 2, -5)$, which gives the
positive roots  
$R^{+} = \{ (\epsilon_1 - \epsilon_2) \upharpoonright  V, (\epsilon_1 - \epsilon_3) \upharpoonright  V, 
(\epsilon_2 - \epsilon_3) \upharpoonright  V, \epsilon_1 \upharpoonright  V, \epsilon_2 \upharpoonright  V, - \epsilon_3 \upharpoonright  V \}$
 and, since $\epsilon_1 \upharpoonright V = \frac{1}{3}(2 \epsilon_1 - \epsilon_2 - \epsilon_3) \upharpoonright V$,
 a positive basis $B = \{ (\epsilon_1 - \epsilon_2) \upharpoonright V, \epsilon_2 \upharpoonright V \}$.
 The co-root lattice, central lattice, center and sum of the positive roots are given by 
 \begin{itemize}
 \item $\Lambda_{R\,\check{}} = \langle 2e_1-e_2 -e_3, -e_1 +2e_2-e_3, -e_1-e_2 + 2e_3, e_\mu \pm e_\nu : 1 \leq \mu < \nu \leq 3 \rangle $
 \item $\Lambda_Z = \Lambda_{R\,\check{}}$
 \item $Z = \Lambda_Z / \Lambda_{R\,\check{}} \simeq 1$
\item $2\rho \equiv \sum_{\alpha \in R^+} \alpha = (3\epsilon_1 + \epsilon_2 - 3\epsilon_3) \upharpoonright V = (2\epsilon_2 - 4 \epsilon_3) \upharpoonright V$
 \end{itemize}
 %And, the unique compact simple Lie group associated to this root system will also be denoted by $G_2$.
 
The corresponding simply-connected compact Lie group is also denoted by $G_2$ and it is the unique group 
of this type. The integral lattice of this group with respect to any bi-invariant metric given by
\begin{equation}\label{eqn:IntegralLatticeG_2}
\Lambda_I(G_2) =  \Lambda_{R\,\check{}}.
\end{equation}
 
%\noindent
%{\bf Type $E_8$:} As a vector space we have $V = \R^8$.
\subsubsection{Type $E_8$}\label{sec:TypeE8}
As a vector space we have $V =\R^8$.
The corresponding root system is the union of the $112$ roots of $D_8$ with $128$ additional roots
$R = \{ \pm \epsilon_\mu \pm \epsilon_\nu : 1 \leq \mu < \nu \leq 8 \} \cup 
\{ \frac{1}{2} \sum_{\mu =1}^{8} \pm \epsilon_\mu : \mbox{ there are an even number of minus signs } \} =
\{ \pm \epsilon_\mu \pm \epsilon_\nu : 1 \leq \mu < \nu \leq 8 \} \cup 
\{ \frac{1}{2} \sum_{\mu =1}^{8} (-1)^{k_{\mu}} \epsilon_\mu : k_\mu = 0,1 \mbox{ and } \sum k_\mu \equiv 0 \mod 2 \}$
with co-root system 
$R\,\check{} = \{ e_\mu \pm e_\nu : 1 \leq \mu \neq \nu \leq 8 \} \cup \{ \frac{1}{2} \sum_{\mu =1}^{8} \pm e_\mu : \mbox{ there are an even number of minus signs }\}$.
For a Weyl chamber we choose the component $\mathcal{C}$ of $V - \cup \ker(\alpha)$ containing the vector 
%$(29, 7, 6, 5, 4, 3, 2, 1)$.
$(0,1,2,3,4,5,6, 23)$.
Then the positive roots are $R^{+} = R_{1}^{+} \cup R_2^+ \cup R_3^+$, where 
\begin{itemize}
\item $R_1^+ = \{ \frac{1}{2}( \epsilon_8 + \epsilon_7 + \sum_{\mu =1}^{6} (-1)^{k_\mu} \epsilon_\mu : 
\sum k_\mu \equiv 0 \mod 2 \}$
\item $R_2^+ = \{ \frac{1}{2}( \epsilon_8  - \epsilon_7  + \sum_{\mu =1}^{6} (-1)^{k_\mu} \epsilon_\mu : 
\sum k_\mu \equiv 1 \mod 2 \}$
\item $R_3^+ = \{ \epsilon_\mu \pm \epsilon_\nu : 1 \leq \mu < \nu \leq 8 \}$.
\end{itemize}
%Then the positive roots are $R^{+} = R_{1}^{+} \cup R_2^+ \cup R_3^+$, where 
%\begin{itemize}
%\item $R_1^+ = \{ \frac{1}{2}( \epsilon_1 + \epsilon_2 \pm \epsilon_3 \pm \cdots \pm \epsilon_8) : \mbox{ there are an even number of minus signs } \}$
%\item $R_2^+ = \{ \frac{1}{2}( \epsilon_1 - \epsilon_2 \pm \epsilon_3 \pm \cdots \pm \epsilon_8) : \mbox{ there are an even number of minus signs } \}$
%\item $R_3^+ = \{ \epsilon_\mu \pm \epsilon_\nu : 1 \leq \mu < \nu \leq 8 \}$.
%\end{itemize}
The sets $R_1$ and $R_2$ each contain $32$ elements, while $R_3$ contains $56$.
As a positive basis for this root system we have the set $B$ consisting of the elements
$\alpha_1 = \frac{1}{2}( \epsilon_1 + \epsilon_8 - \sum_{\mu =2}^{7} \epsilon_\mu)$, 
$\alpha_2 = e_1 + e_2$, $\alpha_{j+1} = \epsilon_j - \epsilon_{j-1}$ for $j = 2, \ldots , 7$. 
Now, since the roots of $D_8$ are contained in the roots of $E_8$ we realize that 
the central lattice of $E_8$ is contained in the central lattice of $D_8$. 
The co-root lattice, central lattice, center and sum of the positive roots are then given by 
\begin{itemize}
\item $\Lambda_{R\,\check{}} = \Lambda_Z = \langle v_1, \ldots , v_7,  v_8 \rangle$, 
where $v_1 = \frac{1}{2}( e_1 + e_8 - \sum_{j=2}^{7} e_j)$, $v_2 = e_1 + e_2$ 
and $v_{j+1} = v_j - v_{j-1}$ for $j = 2, \ldots , 7$.

\item $Z = \Lambda_Z / \Lambda_{R\,\check{}}  \simeq 1$

\item $2\rho \equiv \sum_{\alpha \in R^+} \alpha = 32 \epsilon_8 + \sum_{j=1}^{7} 2 (8- j) \epsilon_j$ 
\end{itemize}
%And, the unique compact simple Lie group associated to this root system will also be denoted by $E_8$.

The corresponding simply-connected compact Lie group is also denoted by $E_8$ and it is the unique group 
of this type. The integral lattice of this group with respect to any bi-invariant metric given by
\begin{equation}\label{eqn:IntegralLatticeE_8}
\Lambda_I(E_8) =  \Lambda_{R\,\check{}}.
\end{equation}

%\noindent
%{\bf Type $E_7$:} As a vector space we take $V$ to be the $7$-dimensional subspace of $\R^8$ 
%spanned by $e_1, \ldots , e_6, e_7 + e_8$.
\subsubsection{Type $E_7$}\label{sec:TypeE7}
We will describe this root system in terms of $E_8$.
Letting $v_1, \ldots, , v_8$ and $\alpha_1 , \ldots , \alpha_8$ be as in \ref{sec:TypeE8}, 
the vector space $V$ will be the $7$-dimensional subspace of $\R^8$ spanned by 
$v_1 = \frac{1}{2}(e_1 + e_8 - \sum_{j=2}^{7} e_j)$, $v_2 = e_1 + e_2$, 
$v_{j+1} = e_j - e_{j-1}$ for $j = 2, \ldots 6$.
In other words, $V$ is spanned by $e_1, \ldots , e_6, e_7 - e_8$
and consists of the vectors in $\R^8$ where the $e_7$ and $e_8$ coordinate are opposite.
The corresponding root system $R \subset V^*$ consists of the roots in $E_8$ in the span 
of $\alpha_1, \ldots , \alpha_7$. 
Specifically, we have 
$R = \{ \pm (\epsilon_\mu \pm \epsilon_\nu) \upharpoonright V : 1 \leq \mu < \nu \leq 6 \} \cup 
\{ \pm (\epsilon_7 - \epsilon_8) \upharpoonright V\} \cup
\{ \pm \frac{1}{2}((\epsilon_7 - \epsilon_8) \upharpoonright V + \sum_{j=1}^{6} (-1)^{k_j} \epsilon_j) : \sum k_j \equiv 1 \mod 2 \}$.
The co-root system is  
$R\,\check{} = \{ \pm (e_\mu \pm e_\nu) : 1 \leq \mu < \nu \leq 6 \} \cup 
\{ \pm (e_7 - e_8) \} \cup
\{ \pm \frac{1}{2} (e_7 - e_8 + \sum_{j=1}^{6} (-1)^{k_j} e_j ): \sum k_j \equiv 1 \mod 2 \}$.
For a Weyl chamber we choose the component $\mathcal{C}$ of $V - \cup_{\alpha \in R} \ker(\alpha)$ containing the 
regular vector $6e_1 + 5e_2 + 4e_3 + 3 e_4 + 2e_5 + 1e_6 + 11(e_7 - e_8)$.
Then the positive roots are given by the set $R^{+} = R_{1}^{+} \cup R_2^+ \cup R_3^+$, where 
\begin{itemize}
\item $R_1^+ = \{ (\epsilon_7 - \epsilon_8) \upharpoonright V \}$
\item $R_2^+ = \{ (\epsilon_\mu \pm \epsilon_\nu) \upharpoonright V: 1 \leq \mu < \nu \leq 6 \} \}$
\item $R_3^+ = \{ \frac{1}{2}( \epsilon_7 - \epsilon_8 + \sum_{\mu = 1}^{6} (-1)^{k_\mu} \epsilon_\mu : \sum k_\mu \equiv 1 \mod 2\}$.
\end{itemize}
The set $R_1^+$ has one element, the set $R_2^+$ contains $30$ elements and $R_3^+$ contains $32$ elements.
As a positive basis for this root system we have $B = \{\alpha_1, \ldots , \alpha_7\}$.  
The co-root lattice, central lattice, center and sum of the positive roots are then given by

\begin{itemize}
\item $\Lambda_{R\,\check{}} = \langle v_1, \ldots, v_6, v_7 \rangle$
%\langle \frac{1}{2}(e_7 - e_8 + \sum_{j = 1}^{5}e_j - e_6),e_1 + e_2, e_j - e_{j+1} : j =1, \ldots , 5 \rangle$ {\bf Check!!} 
\item $\Lambda_Z = \langle v_1, \ldots , v_6, v_7, F \rangle$, where 
$$F = \frac{1}{2}(e_7 -e_8) + e_1 + e_2 + e_3 = 
-v_1 +\frac{1}{2} v_2 - v_3 - v_4 - \frac{3}{2} v_5 - v_6$$
%\item $\Lambda_Z$ consists of the vectors  $\frac{1}{2} (\sum_{j=1}^{6} c_j e_j + c_7(e_7 - e_8))$ 
%with $c_j \in \Z$ such that 
%\begin{enumerate}
%\item $c_1 \equiv \cdots \equiv c_6 \equiv 0 \mod 2$
%\item $c_7$ is even if and only if $\sum_{j=1}^{6} c_j \equiv 0 \mod 4$.
%\end{enumerate}

\item $Z = \Lambda_Z / \Lambda_{R\,\check{}} = \langle \overline{F} \rangle  \simeq \Z_2$ %{\bf Check!!} 
\item $2 \rho \equiv \sum_{\alpha \in R^+} \alpha = \sum_{j=1}^{6} 2(6-j) \epsilon_j + 17(\epsilon_7 - \epsilon_8) \upharpoonright V$ 
%\item $2\rho = 2(\epsilon_7 + \epsilon_8) + 4 \sum_{\nu =1}^{6} (6 - \nu)\epsilon_\nu + 32 \epsilon_1$
\end{itemize}
%And, the associated compact simple Lie groups are of the form $\Gamma \backslash E_{7}$, 
%where we let $E_7$ also denote the simply-connected compact  Lie group having this root system 
%and $\Gamma \leq Z(E_7) \simeq Z$.

The associated simply-connected compact Lie group will also be denoted by $E_7$ and 
all other groups of this type are of the form $U = E_7/ \Gamma$, where 
$\Gamma \leq Z(E_7) = \langle \overline{F} \rangle \simeq \Z_2$.
The integral lattice of $U$ with respect to any bi-invariant metric is 
\begin{equation}\label{eqn:IntegralLatticeE_7}
\Lambda_I(E_7 / \Gamma) = \left\{ \begin{array}{ll}
\Lambda_{R\,\check{}} & \Gamma \simeq 1 \\
\Lambda_Z & \Gamma = \langle \overline{F} \rangle \simeq \Z_2
\end{array}
\right. 
\end{equation}

%\noindent
%{\bf Type $E_6$:} As a vector space we take $V$ to be the $6$-dimensional subspace of
%$\R^8$ spanned by $e_1, \ldots , e_5, e_6 + e_7 + e_8$.
\subsubsection{Type $E_6$}\label{sec:TypeE6}
As a vector space we take $V_6$ to be the $6$-dimensional subspace of
$\R^8$ spanned by $v_1, \ldots , v_6$ where the $v_i$'s are as in \ref{sec:TypeE8}.
One can check that in this case $V_6$ consists of the vectors in $V_7$ 
which are orthogonal to $e_6+e_8$ and we see that 
$V_6$ has basis given by $e_1 , \ldots , e_5, e_6+e_7-e_8$.
The dual space $V_6^*$ is the linear span of $\alpha_1, \ldots , \alpha_6$.
The roots in $V_6$ are given by $R = R_8 \cap V_6$. More explicitly, we obtain
$R = \{ \pm (\epsilon_\mu \pm \epsilon_\nu) \upharpoonright V : 1 \leq \mu < \nu \leq 5 \} \cup 
\{ \pm \frac{1}{2}(\epsilon_6 + \epsilon_7 - \epsilon_8 + \sum_{j=1}^{5} (-1)^{k_j} \epsilon_j : \sum_{j=1}^{5} k_j \equiv 1 \mod 2 \}$
and the co-roots are given by 
$R\, \check{} = \{ \pm (e_\mu \pm e_\nu) \upharpoonright V : 1 \leq \mu < \nu \leq 5 \} \cup 
\{ \pm \frac{1}{2}(e_6 + e_7 - e_8 + \sum_{j=1}^{5} (-1)^{k_j} e_j : \sum_{j=1}^{5} k_j \equiv 1 \mod 2 \}$.
We choose our Weyl chamber $\mathcal{C}$ to be the component containing the regular vector 
$5e_1 + 4e_2 + 3e_3 + 2e_4 + 1e_1 + 6(e_6 + e_7 - e_8)$ with corresponding positive roots
$R^+ = R_1^+ \cup R_2^+$, where
\begin{itemize}
\item $R_1^+ = \{ \epsilon_i \pm \epsilon_j : 1 \leq i < j \leq 5 \}$
\item $R_2^+ = \{ \frac{1}{2}(\epsilon_6 + \epsilon_7 - \epsilon_8 + \sum_{j=1}^{5} (-1)^{k_j} e_j : \sum k_j \equiv 1 \mod 2 \}$.
\end{itemize} 
One can check that the co-root lattice, central lattice, center and the sum of the positive roots are as follows:
\begin{itemize}
\item $\Lambda_{R\,\check{}} = \langle v_1, \ldots, v_5, v_6 \rangle$
\item $\Lambda_{Z} = \langle v_1, \ldots , v_5, v_6, F \rangle$, 
where $$F = \frac{2}{3}(e_6 + e_7 - e_8) = \frac{2}{3}(-2v_1 -\frac{3}{2}v_2 -\frac{5}{2}v_3 -3v_4 -2v_5 -v_6 + v_7)$$
\item $Z = \Lambda_Z / \Lambda_{R\,\check{}} = \langle \overline{F} \rangle  \simeq \Z_3$
\item $2 \rho \equiv \sum_{\alpha \in R^+} \alpha = \sum_{j=1}^{5} 2(5-j) \epsilon_j + 8 (\epsilon_6 + \epsilon_7 - \epsilon_8)$.
\end{itemize}
%And, the associated compact simple Lie groups are of the form $\Gamma \backslash E_{6}$, 
%where we let $E_6$ also denote the simply-connected compact  Lie group having this root system 
%and $\Gamma \leq Z(E_6) \simeq Z$.

The associated simply-connected compact Lie group will also be denoted by $E_6$ and 
all other groups of this type are of the form $U = E_6/ \Gamma$, where 
$\Gamma \leq Z(E_6) = \langle \overline{F} \rangle \simeq \Z_3$.
The integral lattice of $U$ with respect to any bi-invariant metric is 

\begin{equation}\label{eqn:IntegralLatticeE_6}
\Lambda_I(E_6 / \Gamma) = \left\{ \begin{array}{ll}
\Lambda_{R\,\check{}} & \Gamma \simeq 1 \\
\Lambda_Z & \Gamma = \langle \overline{F} \rangle \simeq \Z_3.
\end{array}
\right. 
\end{equation}

%%%%%%%%%%%%%%%%%%%%%%%%%%%%%
%%%%%%%%%%%%%%%%%%%%%%%%%%%%%
%%%%%%%%%%%%%%%%%%%%%%%%%%%%%

\section{Constructing an $\Ad(\Delta K)$-invariant Complement}\label{sec:AdK}
Let $U$ be a compact simple Lie group with bi-invariant metric $g_0$. Now, suppose $g$ is a left-invariant metric on $U$ than is naturally reductive with respect to $G\equiv U \times K$ for some $K \leq U$. We now review the procedure for constructing an $\Ad(\Delta K)$-invariant complement $\germ{p}$ of $\Delta \germ{K}$ in $\germ{g} = \germ{u} \oplus \germ{K}$ as discussed in \cite{DZ}.
 
%%%%%%%%%%%%%%%%%%%%%%%%%%%%%
%\begin{qorthogonal}\label{rem:QOrthogonal}
To begin we let $A: \germ{K}_0 \to \germ{K}_0$ denote the $g_0$-symmetric endomorphism satisfying $h(X,Y) = g_0(AX, Y)$ for each $X, Y \in \germ{K}_0$. Then as is described in \cite[p. 9-11]{DZ} there are two cases to consider:

\begin{enumerate}
\item $\alpha$ is not an eigenvalue of $A$ and $\alpha_ j \neq \alpha$ for each $j = 1, \ldots, r$.

In this case we consider the symmetric bi-linear form $Q$ on $\germ{g} \times \germ{K}$ given by 

\begin{eqnarray*}%\label{eqn:BilinearForm}
Q = \beta \, g \upharpoonright \germ{u} \oplus 0 + \bar{h} \upharpoonright 0 \oplus \germ{K}_0 + \beta_1 \, g \upharpoonright 0 \oplus \germ{K}_1 + \cdots + \beta_r \, g \upharpoonright 0 \oplus \germ{K}_r,
\end{eqnarray*}
where $\beta = \alpha$, $\beta_j = \frac{\beta \alpha_j}{\alpha - \alpha_j}$, and 
$\bar{h}(X,Y) = g_0(\overline{A}X, Y)$ is defined by the $g_0$-symmetric endomorphism  
$\overline{A}: \germ{K}_0 \to \germ{K}_0$ satisfying $A = \beta \overline{A} ( \overline{A} + \beta I)^{-1}$. $Q$ can be seen to be non-degenerate on $\germ{g} \times \germ{K}$ and $\Delta \germ{K}$. We then take $\germ{p}$ to be the $Q$-orthogonal complement of $\Delta \germ{K}$ which is given by 
\begin{eqnarray}\label{eqn:QOrthogonal1}
\germ{p} = \germ{p}_1 \oplus \germ{q}_0 \oplus \germ{q}_1 \oplus \cdots \oplus \germ{q}_r,
\end{eqnarray}
where 
\begin{enumerate}
\item $\germ{p}_1 = \{(X,0) : X \in \germ{u} \}$;
\item $\germ{q}_0 = \{ (\overline{A}X, -\beta X ) : X \in \germ{K}_0 \}$;
\item $\germ{q}_j = \{ (\beta_j X , - \beta X) : X \in \germ{K}_j \}$ for $j = 1, \ldots , r$.
\end{enumerate}
From this one may conclude that the metric $g_{\alpha, \alpha_1, \ldots , \alpha_r, h}$ is naturally reductive.

\medskip

\item $\alpha$ is an eigenvalue of $A$ \emph{or} $\alpha_ j = \alpha$ for some $j = 1, \ldots, r$.

We find the $\Ad(\Delta K)$-invariant complement $\germ{p}$ of $\Delta \germ{K}$ in $\germ{u} \times \germ{K}$ by considering a proper subgroup $K' \leq K$ with respect to which the metric $g_{\alpha, \alpha_1, \ldots , \alpha_r, h}$ falls into the previous case. Indeed, consider the Lie algebra $$\germ{K}'' = \germ{K}_0'' \oplus (\oplus_{\alpha_j = \alpha} \germ{K}_j),$$ where $\germ{K}_0'' = \{ X \in \germ{K}_0 : AX = \alpha X \}$. Then we let $\germ{K}'$ denote the $g_0$-orthogonal complement of $\germ{K}''$ in $\germ{K}$ and let $K'$ denote the corresponding connected proper subgroup of $K$.
One can check that $$\germ{K}' = \germ{K}_0' \oplus (\oplus_{\alpha_j \neq \alpha} \germ{K}_j),$$ where $\germ{K}_0'$ is the $g_0$-orthogonal complement of $\germ{K}_0''$ in $\germ{K}_0$. We can then view the metric $g_{\alpha, \alpha_1, \ldots , \alpha_r, h}$ as being induced by the inner product 
$$\alpha g_0 \upharpoonright \germ{m}' \oplus h\upharpoonright \germ{K}_0' \oplus (\oplus_{\alpha_j \neq \alpha} \alpha_j g_0 \upharpoonright \germ{K}_j),$$
where $\germ{m}' = \germ{m} \oplus \germ{K}''$ is the $g_0$ orthogonal complement of $\germ{K}'$. The metric then falls into the previous case with respect to $K'$ and we take $\germ{p}$ to be the corresponding complement of $\Delta \germ{K}'$ in $\germ{g} \times \germ{K}'$:
\begin{eqnarray}\label{eqn:QOrthogonal2}
\germ{p} = \germ{p}_1' \oplus \germ{q}_0' \oplus (\oplus_{\alpha_j \neq \alpha} \germ{q}_j),
\end{eqnarray}
where 
\begin{enumerate}
\item $\germ{p}_1' = \{(X,0) : X \in \germ{m}' \}$;
\item $\germ{q}_0' = \{ (\overline{A}X, -\beta X ) : X \in \germ{K}_0' \}$;
\item $\germ{q}_j = \{ (\beta_j X , - \beta X) : X \in \germ{K}_j \}$ for $j = 1, \ldots , r$.
\end{enumerate} 
However, one can check that $\germ{p}$ is also an $\Ad(\Delta K)$-invariant complement of $\Delta \germ{K}$ in $\germ{u} \times \germ{K}$ and we can then see that the metric is naturally reductive with respect to $U \times K$.
\end{enumerate}

%\end{qorthogonal}
%%%%%%%%%%%%%%%%%%%%%%%%%%%%%

%%%%%%%%%%%%%%%%%%%%%%%%%%%%%%%%%%%%%%%%%%%%%%
%%%%%%%%%%                   BIBLIOGRAPHY        %%%%%%%%%%%%%%%%%%%%
%%%%%%%%%%%%%%%%%%%%%%%%%%%%%%%%%%%%%%%%%%%%%%

%\bibliographystyle{amsplain}
\bibliographystyle{amsalpha}

\begin{thebibliography}{BFSTW}

\bibitem[A]{Abraham}
 R. Abraham,
\newblock {\em Bumpy metrics},
\newblock Global Analysis (Proc. Sympos. Pure Math., Vol. XIV, Berkeley,
Calif., 1968), 1--3, Amer. Math. Soc., Providence, R.I., 1970.

\bibitem[Ad]{Adams}
J.F.~Adams,
\newblock {\em Lectures on Exceptional Lie Groups},
\newblock The University of Chicago Press (Chicago), 1996.

\bibitem[An]{Anosov}
D.V. Anosov,
\newblock{\em Generic properties of closed geodesics},
\newblock Math.USSR Izvestiya {\bf 21} (1983), 1--29.

\bibitem[Be1]{Besse} 
A. L. Besse, 
\newblock \emph{Manifolds all of whose Geodesics are Closed}, 
\newblock Springer-Verlag (Berlin), 1978.

\bibitem[Be2]{Besse2} 
A. L. Besse, 
\newblock \emph{Einstein Manifolds}, 
\newblock Springer-Verlag (Berlin), 1987.

\bibitem[BTZ]{BTZ}
W. Ballmann, G. Thorbergsson and W. Ziller,
\newblock {\em Closed geodesics on positively curved manifolds},
\newblock Ann. Math. {\bf 116} (1982), 213--247.

\bibitem[BtD]{BtD} 
T. Br\"{o}cker and T. tom Dieck, 
\newblock \emph{Representations of Compact Lie Groups}, 
\newblock Springer-Verlag (Berlin), 1985.

\bibitem[BFSTW]{BFSTW}
N. Brown, R. Fink, M. Spencer, K. Tapp and Z. Wu,
\newblock {\em Invariant metrics with nonnegative curvature on compact Lie groups},
\newblock Canad. Math. Bull. {\bf 50} (2007), 24--34.

\bibitem[BPU]{BPU}
R. Brummelhuis, T. Paul and A. Uribe,
\newblock {\em Spectral estimates around a critical level},
\newblock Duke Math. J. {\bf 78} (1995), 477--530.

\bibitem[Ch]{Chazarain}
J. Chazarain,
\newblock {\em Formule de Poisson pour les vari\'{e}t\'{e}s riemanniennes},
\newblock Invent. Math. {\bf 24} (1974), 65--82. %MR0343320, Zbl 0281.35028.

\bibitem[Ci]{Cianci}
D. Cianci,
\newblock {\em On the Poisson relation for lens spaces},
\newblock preprint.

\bibitem[CdV]{CdV}
Y. Colin de Verdi\`{e}re,
\newblock {\em Spectre du laplacien et longueurs des g\'{e}od\'{e}siques
p\'{e}riodiques II}, \newblock Compositio Math. {\bf 27} (1973), 159--184. % MR0348798, Zbl 0281.53036.

\bibitem[D]{DAtri} 
J.E. D'Atri, 
\newblock \emph{Geodesic spheres and symmetries in naturally reductive spaces}, 
\newblock Michigan Math. J. {\bf 22} (1975), 71--76.

\bibitem[DZ]{DZ} 
J.E. D'Atri and W. Ziller, 
\newblock \emph{Naturally reductive metrics and Einstein metrics on compact Lie groups}, 
\newblock Mem. of Amer. Math. Soc., \textbf{18} (1979), no. 215.

\bibitem[DuGu]{DuGu}
J.J. Duistermaat and V. Guillemin,
\newblock {\em The spectrum of positive elliptic operators and periodic
bicharacteristics}, \newblock Invent. Math. {\bf 29} (1975), 39--79. %MR0405514, Zbl 0307.35071.

%\bibitem[DK]{DuKo} 
%J.J. Duistermaat and J.A.C. Kolk, 
%\newblock \emph{Lie Groups}, 
%\newblock Springer-Verlag (Berlin), 2000.

\bibitem[DKV]{DuKoVa}
J.J. Duistermaat, J.A.C. Kolk and V.S. Varadarajan,
\newblock {\em Spectra of compact locally symmetric manifolds of nonnegative curvature}, 
\newblock Invent. Math. {\bf 52} (1979), 27--93. 


\bibitem[Gan]{Gangoli} R. Gangoli, 
\newblock \emph{The length spectra of some compact manifolds of negative curvature}, 
\newblock J. Differential Geom. {\bf 12} (1977), 403--424.

\bibitem[GorM]{GorMao} C.S. Gordon and Y. Mao, 
\newblock \emph{Comparisons of Laplace spectra, length spectra and geodesic flows of some Riemannian manifolds}, 
\newblock Math. Res. Lett. {\bf 1} (1994), 677--688.

%\bibitem[GSS]{GSS} C.S. Gordon, D. Schueth and C.J. Sutton, 
%\newblock \emph{Spectral isolation of bi-invariant metrics on compact Lie groups}, 
%\newblock Ann. Inst. Fourier (Grenoble) {\bf 60} (2010), 1617--1628.

\bibitem[GS]{GorSut} C.S. Gordon and C.J. Sutton, 
\newblock \emph{Spectral isolation of naturally reductive metrics on simple Lie groups}, 
\newblock Math. Z. {\bf 266} (2010), 979--995.

\bibitem[Gt]{Gornet} 
R. Gornet,
\newblock \emph{Riemannian nilmanifolds and the trace formula}, 
\newblock Trans. Amer. Math. Soc. {\bf 357} (2005), no. 11, 4445--4479.

%\bibitem[Gt2]{Gornet2} 
%R. Gornet,
%\newblock \emph{Riemannian nilmanifolds, the wave trace, and the length spectrum}, 
%\newblock Comm. Anal. Geom. {\bf 16} (2008), no. 1, 27--89.

\bibitem[Gu]{Guillemin} V. Guillemin, 
\newblock \emph{Some spectral results on rank one symmetric spaces}, 
\newblock Adv. Math. {\bf 28} (1978), 129--137.

%\bibitem[GuSt]{GuSt}
%V. Guillemin and S. Sternberg,
%\newblock {\em Symplectic techniques in physics},
%\newblock Cambridge University Press, New York, 1984.

\bibitem[Hel]{Helgason}
S.~Helgason,
\newblock {\em Differential geometry, Lie groups, symmetric spaces},
\newblock Academic Press, San Diego, 1978.



%\bibitem[HZ]{HezZel}
%H. Hezari and S. Zelditch,
%\newblock {\em Inverse spectral problem for analytic $(\Z/2 \Z)^n$-symmetric domains in $\R^n$},
%\newblock GAFA {\bf 20} (2010), 160--191.

\bibitem[Hub1]{Huber1}
H. Huber,
\newblock {\em Zur analytischen Theorie hyperbolischer Raumformen und Bewegungsgruppen I},
\newblock Math. Ann. {\bf 138} (1959), 1--26.

\bibitem[Hub2]{Huber2}
H. Huber,
\newblock {\em Zur analytischen Theorie hyperbolischer Raumformen und Bewegungsgruppen II},
\newblock Math. Ann. {\bf 143} (1961), 463--464.

\bibitem[Hum]{Humphreys} 
J. Humphreys, 
\newblock \emph{Introduction to Lie Algebras and Representation Theory}, 
\newblock Springer-Verlag (New York), 1994.

%\bibitem[KT]{KT}
%W. Klingenberg and F. Takens,
%\newblock {\em Generic properties of geodesic flows},
%\newblock Math. Ann. {\bf 197} (1972), 323--334.

\bibitem[KS]{KS}
O. Kowalski and J. Szenthe,
\newblock {\em On the existence of homogeneous geodesics in homogeneous Riemannian manifolds},
\newblock Geom. Dedicata {\bf 81} (2000), 209--214.

%\bibitem[Ku]{Ku}
%R. Kuwabara,
%\newblock {\em On the characterization of flat metrics by the spectrum},
%\newblock Comm. Math. Helv. {\bf 55} (1980), 427--444.

\bibitem[Lo]{Loos}
O.~Loos,
\newblock {\em Symmetric Spaces II: Compact Spaces and Classification},
\newblock W.A. Benjamin Inc. (New York), 1969.

\bibitem[M]{McKean}
H.P.~McKean,
\newblock {\em Selberg's trace formula as applied to a compact Riemann surface},
\newblock Comm. Pure Appl. Math. {\bf 25} (1972), 225--246.

\bibitem[MR]{Miatello-Rossetti}
R.J. Miatello and J.P. Rossetti,
\newblock {\em Length spectra and $p$-spectra of compact flat manifolds},
\newblock J. Geom. Anal. {\bf 13} (2003), 631--657.

\bibitem[Pes]{Pesce}
H. Pesce,
\newblock {\em Une formule de Poisson pour les vari\'{e}t\'{e}s Heisenberg},
\newblock Duke Math. J. {\bf 73} (1994), 79--95.

\bibitem[PR]{Prasad-Rapinchuk}
G. Prasad and A.S. Rapinchuk,
\newblock {\em Weakly commensurable arithmetic groups and isospectral locally symmetric spaces},
\newblock Publ. Math. Inst. Hautes \'{E}tudes Sci. No. 109 (2009), 113--184.

\bibitem[Sa]{Sakai} 
T. Sakai, 
\newblock \emph{Riemannian Geometry}, 
\newblock Translations of Mathematical Monographs {\bf 149}, American Mathematical Society (Providence), 1996.

\bibitem[Sa]{Samelson}
H.~Samelson,
\newblock {\em Notes on Lie Algebras},
\newblock Spring-Verlag (New York), 1990.

%\bibitem[Sch]{Schueth}
%D. Schueth,
%\newblock {\em Isospectral manifolds with different local geometries},
%\newblock J. reine angew. Math., {\bf 534} (2001), 41--94.

%\bibitem[SS]{SS}
%B. Schmidt and C. J. Sutton,
%\newblock {\em Two remarks on the length spectrum of a Riemannian manifold},
%\newblock Proc. Amer. Math. Soc. {\bf 139} (2011), no. 11, 4113--4119.

\bibitem[SS]{SS}
B. Schmidt and C. J. Sutton,
\newblock {\em Detecting the moments of inertia of a molecule via its rotational spectrum},
\newblock preprint, 29 pages.


\bibitem[Sun]{Sunada}
T. Sunada,
\newblock {\em Riemannian coverings and isospectral manifolds},
\newblock Ann. Math. {\bf121} (1985), 169--186.


%\bibitem[Sut]{Sut}
%C. J. Sutton
%\newblock {\em Isospectral simply-connected homogeneous spaces and the spectral
%rigidity of group actions},
%\newblock Comment. Math. Helv. {\bf 77} (2002), 701--717.

%\bibitem[Sut]{Sut2}
%C. J. Sutton
%\newblock {\em Detecting the moments of inertia of a molecule via its rotational spectrum},
%\newblock preprint, 37 pages.

\bibitem[Tak]{Tak} 
L. Takhtajan, 
\newblock \emph{Quantum Mechanics for Mathematicians}, 
\newblock Graduate Studies in Mathematics {\bf 95}, American Mathematical Society (Providence), 2008.

%\bibitem[T1]{Tanno1}
%S. Tanno,
%\newblock {\em Eigenvalues of the Laplacian of Riemannian manifolds},
%\newblock Tohoku Math. J. (2)  {\bf 25} (1973), 391--403.

%\bibitem[T2]{Tanno2}
%S. Tanno,
%\newblock {\em A characterization of the canonical spheres by the spectrum},
%\newblock Math. Z. {\bf 175} (1980), 267--274.

\bibitem[W]{Wilking}
B. Wilking,
\newblock {\em Index parity of closed geodesics},
\newblock Invent. Math. {\bf 144} (2001), 281--295.

\bibitem[Wo]{Wolf} 
J. Wolf, 
\newblock \emph{Spaces of Constant Curvature}, (sixth edition) 
\newblock AMS Chelsea Publishing (Providence, RI), 2011.

\bibitem[Z1]{Ziller1}
W. Ziller,
\newblock {\em Closed geodesics on homogeneous spaces},
\newblock Math. Z. {\bf 152} (1976), 67--88.

\bibitem[Z2]{Ziller2}
W. Ziller,
\newblock {\em The Jacobi equation on naturally reductive compact Remannian homogeneous spaces},
\newblock Comment. Math. Helv. {\bf 52} (1977), 573--590

\bibitem[Z3]{Ziller3}
W. Ziller,
\newblock {\em The free loop space of a globally symmetric space},
\newblock Invent. Math. {\bf 41} (1977), 1--22.

\end{thebibliography}

\end{document}